\documentclass[11pt,reqno]{amsproc}
\usepackage[margin=1in]{geometry}
\usepackage{amsmath, amsthm, amssymb}
\usepackage[colorlinks=true, pdfstartview=FitV, linkcolor=blue,citecolor=blue, urlcolor=blue]{hyperref}
\usepackage[abbrev,lite,nobysame]{amsrefs}
\usepackage{times}
\usepackage[usenames,dvipsnames]{color}

\definecolor{labelkey}{rgb}{0,0,1}
\usepackage{dsfont}

\DeclareMathOperator*{\esssup}{ess\,sup}
\DeclareMathOperator*{\essinf}{ess\,inf}

\def\div{{\rm div}}

\def\R {\mathbb{R}}

\def\div{\mathrm{div}\,}

\newtheorem{proposition}{Proposition}[section]
\newtheorem{theorem}[proposition]{Theorem}
\newtheorem{corollary}[proposition]{Corollary}
\newtheorem{lemma}[proposition]{Lemma}
\theoremstyle{definition}
\newtheorem{definition}[proposition]{Definition}
\newtheorem{remark}[proposition]{Remark}

\numberwithin{equation}{section}
\title{On renormalized solutions to elliptic inclusions with nonstandard growth}
\author{Anna Denkowska}
\address{Cracow University of Economics, Department of Mathematics, Rakowicka 27,
31-510 Krak\'{o}w, Poland}
\email{anna.denkowska@uek.krakow.pl}
\author{Piotr Gwiazda}
\address{Institute of Mathematics, Polish Academy of Sciences, ul. \'Sniadeckich 8, 00-656 Warsaw, Poland}
\email{pgwiazda@mimuw.edu.pl}
\author{Piotr Kalita}
\address{Faculty of Mathematics and Computer Science, Jagiellonian University, ul. \L{}ojasiewicza 6, 30-348 Krak\'{o}w, Poland}
\email{piotr.kalita@ii.uj.edu.pl}
\begin{document}

	\begin{abstract}
		We study the elliptic inclusion given in the following divergence form 
		 \begin{align*}
		 & -\div A(x,\nabla u) \ni f\quad \mathrm{in}\quad \Omega,\\
		 & u=0\quad \mathrm{on}\quad \partial \Omega.
		 \end{align*}
		As we assume that $f\in L^1(\Omega)$, the solutions to the above problem are understood in the renormalized sense. We also assume nonstandard, possibly nonpolynomial, heterogeneous and anisotropic growth and coercivity conditions on the maximally monotone multifunction $A$ which necessitates the use of the nonseparable and nonreflexive Musielak--Orlicz spaces. We prove the existence and uniqueness of the renormalized solution as well as, under additional assumptions on the problem data, its boundedness. The key difficulty, the lack of a Carath\'{e}odory selection of the maximally monotone multifunction is overcome with the use of the Minty transform.
		
	\end{abstract}
	\thanks{Work of PG has been supported by the National Science Center of the Republic of Poland by the project no UMO-2018/31/B/ST1/02289. Work of PK has been supported by the National Science Center of the Republic of Poland by the project UMO-2016/22/A/ST1/00077.}
	
	\thanks{We thank Iwona Chlebicka for discussion and useful remarks. We also thank anonymous reviewers for valuable comments.}
\maketitle

\tableofcontents

\section{Introduction.}
Let $\Omega\subset \R^d$ be an open and bounded domain with sufficiently smooth boundary and let $f:\Omega \to \R$ and $A:\Omega\times \R^d\to 2^{\R^d}$ be a maximally monotone multifunction. We study the existence and uniqueness of solutions to a problem governed by a quasilinear elliptic inclusion
\begin{align}
& -\div A(x,\nabla u) \ni f\quad \mathrm{in}\quad \Omega,\label{prblm_1}\\
& u=0\quad \mathrm{on}\quad \partial \Omega.\label{prblm_2}
\end{align}
We assume that $f\in L^1(\Omega)$ and hence we need to employ the machinery of \textit{renormalized} solutions.

The existence of renormalized solutions for problems with nonstandard growth has already been established for elliptic equations in \cite{gwiazda_wittbold,gwiazda_corrigendum,gwiazda_skrzypczak} (and generalized in \cite{swierczewska1, swierczewska2, gwiazda_wittbold2, chlebicka_parabolic, chlebicka_parabolic2} to the case of evolutionary, parabolic, problems) but, to our surprise, it appears that no results on the existence of renormalized solutions for the differential inclusions with a multivalued leading term have been obtained so far, even with the standard polynomial growth, i.e. in classical Sobolev spaces (note, however, that in \cite{gwiazda_wittbold} there appears the lower order term which can be multivalued). This paper fills this gap.  
 
As we assume the \textit{nonstandard growth condition} on $A$,  our solution belongs to nonreflexive and nonseparable Musielak--Orlicz spaces. To deal with the difficulties associated with the lack of reflexivity and separability we use the results of \cite{kalousek, gwiazda_skrzypczak, gwiazda_wittbold,
gwiazda_corrigendum}. The existence result we present is in fact a generalization of the results of \cite{gwiazda_skrzypczak} and \cite{gwiazda_wittbold,
gwiazda_corrigendum} to the situation when the single valued mapping  $A$ present there becomes a multivalued map. 

Classically, in the framework of elliptic problems given in divergence form such as 
\begin{align}
& -\div A(x,\nabla u) = f\quad \mathrm{in}\quad \Omega,\label{prblm_1_1}\\
& u=0\quad \mathrm{on}\quad \partial \Omega,\label{prblm_2_1}
\end{align}
one needs to assume that the function $A:\Omega \times \R^d \to \R^d$ satisfies the \textit{coercivity} condition which could be possibly given as 
$$
A(x,\xi) \cdot \xi \geq C |\xi|^p - m(x)\quad \textrm{for a.e}\ x\in \Omega\ \textrm{and every}\ \ \xi\in \R^d, 
$$
for some $p>1$ and a given function $m\in L^1(\Omega)$ and the \textit{growth} condition having the form 
$$
|A(x,\xi)| \leq C|\xi|^{p-1} + n(x)\quad \textrm{for a.e}\ x\in \Omega\ \textrm{and every}\ \ \xi\in \R^d, 
$$
with a given function $n\in L^{p'}(\Omega)$. The framework of Musielak--Orlicz spaces replaces the growth of order $p$ with the arbitrary, not necessarily polynomial one. This growth is described by the so called $N$-function $M:\Omega\times \R^d\to \R$. In the simplest situation of polynomial growth of order $p$ there simply holds $M(x,\xi) = |\xi|^p$. We deal with the general situation allowing for the dependence of $M$ on $x$ (\textit{heterogeneity}), on the full vector $\xi$ rather than only its norm (\textit{anisotropy}), and taking into account that the dependence on $\xi$ can be possibly nonpolynomial (\textit{nonstandard growth}). Numerous examples of problems governed by the divergence type operators with such nonstandard growth have been recently presented in the review article  \cite{chlebicka_guide}. We refer to this article for the up-to-date overview of   results, description of key underlying difficulties, and the list of recent literature. For overview of results concerning elliptic PDEs in Orlicz and Musielak--Orlicz spaces we also refer to the monograph \cite{HH} and to articles \cite{Hasto1, Hasto2, Hasto3, byun2, IC1, CDF} where the existence and regularity of solutions for elliptic problems in such spaces is studied. 

Our framework is general: with our results we cover natural scope of Orlicz spaces. We require that either the complementary function $\widetilde{M}$ of the $N$-function $M$ satisfies the $\Delta_2$ condition or the condition which we name (C2) holds. This condition (C2) coming from \cite{chlebicka_parabolic2,chlebicka_parabolic3}, guaranteeing the modular density of smooth functions, always holds when the $N$ function is independent of $x$ variable. So, in such case, i.e.  considering the possibly anisotropic Orlicz growth we are not restricted to any class satisfying doubling conditions. Our results are valid, for example, for the following cases,
\begin{align*}
&M(x,\xi) = |\xi|\ln(1+|\xi|),\\
&M(x,\xi) = |\xi|(\exp|\xi|-1).
\end{align*}
By anisotropy we may mean the different behavior of gradient of 
a function in directions of various coordinates, so we could take
$$
M(x,\xi) = \sum_{i=1}^dB_i(\xi_i),
$$
where $\xi=(\xi_1,\ldots,\xi_d)$ and $B_i$ are Young functions. But our results also cover far more general cases, e.g., in two dimensional case we can consider examples such as
$$
M(x,\xi) = |\xi_1-\xi_2|^\alpha + |\xi_1|^\beta\ln^\delta(c+|\xi_1|), \ \ \alpha,\beta\geq 1,
$$
with $\delta\in \R$ if $\beta > 1$ or $\delta>0$ if $\beta=1$ with large enough $c$. cf. \cite{Trudinger}. We refer to the articles \cite{Cianchi1, Chianci_Alberico} for further results and discussion on existence of generalized solutions, comparison principle, and regularity for elliptic problems of divergent type (with single valued $A$) for the fully anisotropic case.  
For ''$x$-dependent'' spaces we need either doubling of $\widetilde{M}$ or the condition (C2). Our framework covers the case of weighted Sobolev spaces as well as variable spaces governed by
$$
M(x,\xi) = |\xi|^{p(x)}, 1\ll p\ll \infty,
$$
where no log-H\"{o}lder continuity is needed, or double phase spaces
$$
M(x,\xi) = |\xi|^p + a(x) |\xi|^q, \ \ 1<p<q<\infty,
$$
without any conditions on $p$ and $q$, and with bounded $a$ including the borderline situation
$$
M(x,\xi) = |\xi|^p + a(x) |\xi|^p\ln(e+|\xi|), \ \ 1<p<q<\infty,
$$
where $a(x) \geq  0$ is bounded and, typically, $a(x) = 0$ on some subset of the problem domain $\Omega$. We stress that double phase case with single valued $A$, as concerns the regularity of minimizers for associated variational problems, have been recently intensively investigated, cf. \cite{Colombo_Mignone, Baroni_Colombo_Mignone, Baroni_Colombo_Mignone_2, defili_Mignone, byun_oh}. In the present work it is also allowed to consider $M$ governed by various combinations of the above examples, such as, for example
$$
M(x,\xi) = a(x)|\xi|\ln(1+|\xi|),
$$ 
$$
M(x,\xi) = a(x)|\xi|(\exp|\xi|-1),
$$ 
or
$$
M(x,\xi) = a_1(x)|\xi_1|^{p_1(x)}(\exp|\xi|-1)+a_2|\xi_2(x)|^{p_2(x)}\ln(1+|\xi|).
$$
To our knowledge, for variable exponent spaces, double phase spaces, or Orlicz spaces without the growth restrictions, the question of the existence, uniqueness and regularity even of weak solution for the case of multivalued leading term $A$, is open. We cover all these results, and far more, in the generalized framework of renormalized solutions.    

The natural space related to the modular function $M$ is the Musielak--Orlicz space $L_{M}(\Omega)$, and the space to which the solutions of the elliptic boundary value problems are expected to belong is $\{ u\in W^{1,1}_0(\Omega)\, :\ \nabla u\in L_{M}(\Omega)\}$. Once the space is established one asks to which space should the forcing $f$ belong. There are two paths one can follow. The first path is to seek the optimal Sobolev embedding of the solution space  in the space $L^p(\Omega)$ or in some Orlicz or Musielak--Orlicz space involving only function itself and not its derivatives. Then $f$ would belong to the dual of this space. Although the results that characterize the optimal Orlicz space such that the Sobolev embedding holds exist (see e.g. \cite{Cianchi} for anisotropic but homogeneous case, i.e. when $M(x,\xi)$ does not depend on $x$ but depends on the full vector $\xi$ rather that its norm only), we avoid this difficult question by following another possibility, namely we pursue the path of defining the generalized notion of solutions, in our case the so called \textit{renormalized solutions}. This notion allows us to proceed  if $f\in L^1(\Omega)$. The concept of renormalization of the solution, now standard, has been defined first by DiPerna and Lions in context of the transport problems \cite{diperna}, and later generalized by Benilan et al. \cite{Boccardo} to elliptic problems in divergence form in the situations where the classical concept of weak solutions is insufficient. It is worth to add, that the notion of renormalized solutions is only one of possibilities to work with very weak solution notion which, on one hand can be proved to exist, and, on the other hand, under appropriate assumptions, are expected to enjoy further desirable properties such as, for instance, uniqueness or regularity of solutions, or comparison principles. The other solution notions could be SOLA (solutions obtained as limits of approximation) or entropy solutions. We refer to \cite{BG, C3, ZZ, BDS} for some recent results on these types of solutions for elliptic problems of divergence type, in particular the equivalence of these notions for the case of $p$-Laplacian has been obtained in  \cite{kilpel}.

Although the existence results which correspond to the main theorem of the present paper have been obtained in \cite{gwiazda_wittbold,gwiazda_corrigendum,gwiazda_skrzypczak} for the case of a single-valued map in the leading term, the proof of existence in the present article is not a straightforward generalization of the arguments of these papers. The key difficulty and novelty of the present argument lies in showing that the sequences 
$\{a^{\epsilon}(x,\nabla T_k(u^\epsilon))\cdot \nabla T_k(u^{\epsilon})\}_{\epsilon},$
where $T_k$ is the truncation operator, $a^\epsilon$ is the mollification of the multifunction present in the leading term, and $u^\epsilon$ is the approximative sequence, are equiintegrable.
 This property, which is later needed to apply the Minty trick and pass to the limit in the approximative problems, was obtained in  \cite{gwiazda_wittbold,gwiazda_corrigendum,gwiazda_skrzypczak} with the use of the Young measures. 
 The theory of Young measures, however, is in a natural way ''compatible'' with Carath\'{e}odory functions, and to apply it directly one would need the maximal monotone multivalued map $\Omega\times \R^d \ni(x,\xi)\mapsto A(x,\xi)\in \R^d$ to have a Carath\'{e}odory  selection, which would imply it to be single-valued. 
 We deal with the lack of such a selection by using the Minty transformation \cite{Francfort} which permits to associate,
  by a clockwise rotation of the graph by $45^\circ$, the Carath\'{e}odory function with a maximal monotone multifunction. We stress that the method of Minty transformation has been already successfully used in a different context for the elliptic inclusions in divergence form in Sobolev spaces in \cite{Gwiazda_zatorska}.  
  
Our existence result can be seen as a generalization of the results from \cite{gwiazda_wittbold,gwiazda_corrigendum,gwiazda_skrzypczak} to the case of inclusions. We  stress that in our coercivity and growth condition named in the sequel (A3) 
$$
\eta \cdot \xi \geq c_A(M(x,\xi)+\widetilde{M}(x,\eta)) - m(x) \ \textrm{for almost every}\ x\in \Omega \ \textrm{and for every}\ \xi\in \R^d, \eta\in A(x,\xi).
$$
we meticulously deal with the case of a function $m\in L^1(\Omega)$ thus generalizing even the result for the equation from \cite{gwiazda_skrzypczak} where is was assumed that $m=0$. 

 The second main result of the present article, the uniqueness Theorem \ref{thm:uniqueness}, follows by an argument of testing the renormalized form of the equation by a truncation of the difference of two truncations. Although the main assumption, strict monotonicity of the leading operator, is a classical condition to obtain uniqueness, and similar argument has been used, for example, in the context of variable exponent Sobolev spaces in \cite{bendahmahe}, the contribution of the present paper is the adaptation to the case of general Musielak--Orlicz spaces. Note that in \cite{gwiazda_wittbold,gwiazda_corrigendum} the presence of the strictly monotone lower order term ensures uniqueness, while here we rely on the strict monotonicity of the leading operator.  
  
Finally in our last main result, Theorem \ref{thm:renorm_weak}, we establish the $L^\infty(\Omega)$ regularity of the renormalized solution which allows us to deduce that this renormalized solution is in fact weak. The argument is based on the method of \cite{Cianchi1} which was later used and extended in \cite{Chianci_Alberico, Alberico_Cianchi} and which relies on the comparison with the solution of the symmetrized problem which can be calculated explicitly. The application of this method in the context of inclusions and renormalized solutions for the elliptic problems in divergence form, is, to our knowledge, another novelty of the present article. 
We stress that the corresponding results in  
\cite[Proposition 4.3]{gwiazda_wittbold}, \cite[Lemma 2.5]{benilan_wittbold}, \cite[Proposition 5.2]{wittbold_zimmerman} 
are based on the  Stampacchia argument, which needs much stronger assumptions on the problem data. 
Namely, the aforementioned papers need that $f\in L^d(\Omega)$, and the $N$-function satisfies $|\xi|^{r} \leq cM(x,\xi) + C$ for some $r>1$ and constants $C, c > 0$. We significantly relax these assumptions, cf. (W1)--(W2) in Section \ref{sec:main}, where the assumption (W1) on $f$ appears, at least in homogeneous case, sharp. Thus we not only show that the symmetrization argument is valid for elliptic inclusions in Musielak--Orlicz spaces in framework of renormalized solutions, but we improve the results of the articles \cite{gwiazda_wittbold, benilan_wittbold, wittbold_zimmerman} obtained there for the case of equations.    

Although the existence of the solution understood in the generalized (in our case, renormalized) sense for elliptic and parabolic inclusions with the multivalued divergence operator has not been studied before, according to our knowledge, generalized notions of solutions have been considered for the problems governed by scalar conservation laws with a discontinuous and possibly multivalued flux function. In particular, Carillo in \cite{Carillo} proves the existence of the entropy solution for the scalar conservation law with discontinuous flux, the result later generalized in \cite{gwiazda_functional} to the case of solution dependent source and further in \cite{gwiazda_2} to the case where the source term can also be possibly multivalued. 


\section{Main results.}\label{sec:main}

\noindent In this section we formulate three main results of the article, the theorem on the existence of a renormalized solution, its uniqueness, and the one on the relation between the renormalized solution and the weak one. The definitions needed to understand this chapter, such as the definition of an $N$-function $M$, its Fenchel conjugate $\widetilde{M}$, $\Delta_2$ condition, and of Musielak--Orlicz space $L_M(\Omega)$ are contained in Appendix \ref{appa}. 

\medskip

\noindent \textbf{Assumptions.} We start from the assumptions we will need for the data of our problem. First we introduce the multivalued operator $A:\Omega\times \R^d\to 2^{\R^d}$ on which we impose the following conditions.
\begin{itemize}
	\item[(A1)] $A$ is measurable with respect to the $\sigma$-algebra $\mathcal{L}(\Omega)\otimes\mathcal{B}(\R^d)$ on its domain $\Omega\times \R^d$ and the $\sigma$-algebra $\mathcal{B}(\R^d)$ on its range, i.e.
	$$
	\textrm{for every closed set}\ \  D\subset \R^d \ \ \textrm{there holds}\ \ \{(x,\xi)\in \Omega\times \R^d\, :\ A(x,\xi)\cap D \neq \emptyset \}\in \mathcal{L}(\Omega)\otimes \mathcal{B}(\R^d)
	$$  Here $\mathcal{B}(\R^d)$ is the Borel $\sigma$-algebra and $\mathcal{L}(\Omega)$ is the Lebesgue $\sigma$-algebra. 
	\item[(A2)] the multivalued map $A(x,\cdot)$ is maximally monotone for a.e. $x\in \Omega$,
	\item[(A3)] there exists an $N$-function $M$, a constant $c_A\in(0,1]$, and a nonnegative function $m\in L^1(\Omega)$ such that
	$$
	\eta \cdot \xi \geq c_A(M(x,\xi)+\widetilde{M}(x,\eta)) - m(x) \ \textrm{for almost every}\ x\in \Omega \ \textrm{and for every}\ \xi\in \R^d, \eta\in A(x,\xi).
	$$
\end{itemize}
We will discuss these assumptions and some properties of $A$ in Section \ref{multivalued}. Note that $N$-function $M$ which appears in (A3) is given by Definition \ref{def:N}, and its complimentary function $\widetilde{M}$ by \eqref{fenchel}. 

We will seek for our \textit{renormalized solution} with gradients of truncations in Musielak--Orlicz spaces with an $N$-function $M$. To guarantee the existence of the renormalized solution we need an additional assumption on the $N$-function. To this end, we introduce the conditions (C1) and (C2). For our results to hold we need only one of these conditions to hold, which is \textit{either (C1) or (C2)}. We stress that we do not require ${M}$ to satisfy $\Delta_2$ condition. If, in turn, complementary function $\widetilde{M}$ satisfies $\Delta_2$ then our results hold, but $\Delta_2$ condition on $\widetilde{M}$ is not required - if it does not hold then we need  to assume (C2) in its place. So, assuming (C2), our results are valid if neither $M$ nor $\widetilde{M}$ satisfy $\Delta_2$.   
\begin{itemize}
	\item[(C1)] The complementary function $\widetilde{M}$ satisfies the $\Delta_2$ condition given by \eqref{delta2}.
	\item[(C2)] There exists a function $\Theta:[0,1]^2\to [0,\infty)$ nondecreasing with respect to each of the variables, such that 
	\begin{equation}\label{eq:condition_smoothness}
	\limsup_{\delta\to 0^+}\Theta(\delta,\delta^{-d}) < \infty,
	\end{equation}
	which expresses the relation between $M(x,\xi)$ and the function
	$$
	M_{Q}(\xi):=\essinf_{x\in \Omega\cap Q} M(x,\xi).
	$$
	We assume that there exist $\xi_0\in\R^d$ and $\delta_0>0$ such that for every $\delta < \delta_0$ and every cube $Q\subset \R^d$ with $\textrm{diam}\, Q < 4\delta \sqrt{d}$ there holds 
	$$
	\frac{M(x,\xi)}{\widetilde{\widetilde{M_Q}}(\xi)}\leq \Theta(\delta,|\xi|)\ \ \textrm{for a.e.}\ \ x\in Q\cap \Omega\ \ \textrm{and every}\ \ \xi\in \R^d, |\xi| > |\xi_0|,
	$$
	where $\widetilde{\widetilde{M_Q}}$ is the greatest convex minorant of $M_Q$, coinciding with its second complementary function. 
\end{itemize}

\begin{remark}
	The condition (C2) comes from \cite{chlebicka_parabolic2, chlebicka_parabolic3} (also see \cite{Ahmida} for the isotropic version) and it guarantees the modular density of smooth functions in the Musielak--Orlicz--Sobolev space (or, in other words, it excludes the so called Lavrentiev phenomenon \cite{Lavrentiev}). If in addition we assume that
	$$
	M(x,\xi) \geq c_{gr}|\xi|^p\ \ \textrm{with}\ \ p>1\ \ \textrm{and}\ \ c_{gr}>0,
	$$
	then \eqref{eq:condition_smoothness} in (C2) can be replaced with the condition
	$$
	\limsup_{\delta\to 0^+}\Theta(\delta,\delta^{-\frac{d}{p}}) < \infty.
	$$
	We stress that this latter condition is known to be sharp for the modular density of smooth function at least in the variable exponent spaces \cite{Zhikov1, Zhikov} and double-phase spaces \cite{Colombo_Mignone, Baroni_Colombo_Mignone}. 
\end{remark}

\noindent \textbf{Renormalized solution and its existence.} We pass to the key definition of this article: we will define the renormalized solution of the problem \eqref{prblm_1}--\eqref{prblm_2}. To this end, first we note that the $k$-th truncation of the measurable function $f:\Omega\to \R$ is defined by
$$
T_kf(x) = \begin{cases}
f(x) \quad \textrm{if} \quad |f(x)| \leq k,\\
k\frac{f(x)}{|f(x)|}\quad \textrm{otherwise}.
\end{cases}
$$
Define the space 
 $$
V_0^M = \{ v\in W^{1,1}_0(\Omega)\, :\ \nabla v \in L_M(\Omega) \},
$$
where $L_M(\Omega)$ is the Musielak--Orlicz space defined by the $N$-function $M$, cf. Definition \ref{def:Musielak} in Appendix \ref{appa}. 
We are ready to formulate the definition of the renormalized solution to our problem.
\begin{definition}\label{def:renormalized}
	A measurable function $u:\Omega \to \R$ is called a \textit{renormalized solution} to the problem \eqref{prblm_1}--\eqref{prblm_2} if
	\begin{itemize}
		\item[1.] For every $k>0$ there holds $T_k(u)\in V_0^M\cap L^\infty(\Omega)$.
		\item[2.]  There exists a measurable function $\alpha:\Omega\to \R^d$ such that for a.e. $x\in \Omega$ there holds $\alpha(x)\in A(x,\nabla u(x))$ (i.e. $\alpha$ is a measurable selection of $A(\cdot,\nabla u(\cdot))$) such that for  any $h\in C^1_c(\mathbb{R})$ and for any test function 
		$w\in W^{1,\infty}_0(\Omega)$ we have 
		\begin{equation}\label{renorm}
		\int_{\Omega}\alpha \cdot \nabla (h(u)w)\, dx= \int_{\Omega}f h(u) w\, dx.
		\end{equation}
		\item[3.] There holds
		\begin{equation}\label{renorm_control}
		\lim_{k\to\infty}\int_{\{k< |u(x)| < k+1\}} \alpha\cdot \nabla u \, dx = 0.
		\end{equation}
	\end{itemize}
\end{definition}

\begin{remark}
	The generalized gradient of the function $u$ such that $T_k u\in V^M_0$ is defined in the sense of \cite{Boccardo}: there exists a measurable function $v:\Omega\to \R^d$ such that $v \chi_{\{|v|< k\}} = v \chi_{\{|v| \leq k\}} = \nabla T_k(u)$ for almost every $x\in \Omega$ for each  $k>0$.
\end{remark}

\begin{remark}
	The selection $\alpha$ in item 2. of the above definition is also understood in the following sense of \cite{Boccardo}: namely $\alpha:\Omega\to \R^d$ is a measurable function such that for every $k>0$ there exists the selection $\alpha_k\in L_{\widetilde{M}}(\Omega)$ of the multifunction $A(\cdot, \nabla T_k u)$ such that $\alpha_k\chi_{\{ |u| <k \}} = \alpha \chi_{\{ |u| <k \}}$.  We also note, that using the Minty transform the fact that $\alpha_k$ is a selection of  $A(\cdot, \nabla T_k u)$ can be equivalently expressed as 
	$$
	\alpha_k(x) - \nabla T_k u(x) = \varphi_{A(x)}(	\alpha_k(x) + \nabla T_k u (x))
	$$
	for almost every $x\in \Omega$, where $A(x):\R^d\to 2^{R^d}$ is defined as $A(x)(\xi) = A(x,\xi)$, and $\varphi_{A(x)}$ is given by \eqref{minty_tr}. 
\end{remark}

\begin{remark}
	If the condition (C2) holds (this condition is the natural assumption that guarantees the modular density of smooth and compactly supported functions in $L^\infty(\Omega)\cap V_0^M$), then in place of test functions $w\in W^{1,\infty}_0(\Omega)$ we can take them from the broader space $V^M_0 \cap L^\infty(\Omega)$. Indeed, this possibility follows directly from Theorem \ref{thm:approx}.
\end{remark}
\noindent The following theorem is the main result of the article. It will be proved in Section \ref{sec:ex_proof}.
\begin{theorem}\label{thm:existence}
	Suppose that an $N$-function $M$ satisfies either (C1) or (C2). If $f\in L^1(\Omega)$ and $A$ satisfies (A1)--(A3) then the problem \eqref{prblm_1}--\eqref{prblm_2} has a renormalized solution.
\end{theorem}
\noindent \noindent \textbf{Uniqueness of renormalized solution.} In the next theorem on renormalized solution uniqueness we distinguish between cases (C1) and (C2). In the case (C2) Theorem \ref{thm:approx} directly implies that we can take any function belonging to $L^\infty(\Omega)\cap V^M_0$ as the test function in  \eqref{renorm}. This is not the case if we assume (C1). If (C1) holds, we can test \eqref{renorm} only by those functions $w\in L^\infty(\Omega)\cap V^M_0$ which are limits of the sequences $\{ w_n \}_{n=1}^\infty$  belonging to $W^{1,\infty}_0(\Omega)$ convergent in the sense that for every $v\in E_{\widetilde{M}}(\Omega)$, $z\in L^1(\Omega)$, and $k\geq 0$ there holds
\begin{align}
\lim_{n\to \infty} \int_{\Omega} v\cdot \nabla T_k(w_n) \, dx = \int_{\Omega} v\cdot \nabla T_k(w) \, dx\qquad \textrm{and}\qquad \lim_{n\to \infty} \int_{\Omega} z  T_k(w_n) \, dx = \int_{\Omega} z  T_k(w) \, dx,\label{eq:convergence_c1}
\end{align}
i.e. $w_n$ should converge to $w$ weakly-* in $L^\infty(\Omega)$ and $\nabla T_k(w_n)$ should converge to $\nabla T_k(w)$ weakly-* in $L_{\widetilde{M}}(\Omega)$. Thus for the case (C1) we obtain the uniqueness and $L^\infty$ regularity in the class of renormalized solutions which are obtained as the limits as $\epsilon\to 0$ of solutions to the approximative problems \eqref{approx_1}-\eqref{approx_2} used in the proof of Theorem \ref{thm:existence}. 

\bigskip

\noindent The next result establishes the uniqueness of the renormalized solution. The proof is contained in Section \ref{sec:uniq}.
\begin{theorem}\label{thm:uniqueness}
	Assume that $f\in L^1(\Omega)$ and $A$ satisfies assumptions (A1)--(A3). In addition,  assume that $A$ is strictly monotone, i.e. if $\xi\neq \eta$, then 
	for every $g\in A(x,\xi), h\in A(x,\eta)$ and a.e.	$x\in \Omega$ there holds $(g-h)\cdot (\xi-\eta) > 0$.
	\begin{itemize}
		\item If $M$ satisfies  (C1) then the renormalized solution to the problem \eqref{prblm_1}--\eqref{prblm_2} is unique in the class of solutions obtained as the limit as $\epsilon\to 0$ of solutions of the approximative problems \eqref{approx_1}-\eqref{approx_2}.
		\item If an $N$-function $M$ satisfies  (C2) then the renormalized solution to the problem \eqref{prblm_1}--\eqref{prblm_2} is unique. 
	\end{itemize} 
\end{theorem}

\noindent \textbf{Boundedness of renormalized solutions.}   We establish that under reinforced assumptions the renormalized solution belongs to $L^\infty(\Omega)$. To this end we use the techniques of \cite{Cianchi, Chianci_Alberico, Cianchi1, Alberico_Cianchi} where the main concepts rely on using the generalization of Sobolev embedding into the case of finding the optimal Orlicz type space in which the nonhomogeneous Sobolev--Orlicz space embeds, and on comparing of the solution of the original problem with the solution of the corresponding one-dimensional symmetrized one. We begin with some definitions. If $f\in \mathcal{M}(\Omega)$ then by $\mu_f:[0,\infty)\to [0,\infty)$ we will denote the distribution function of $f$ given by
$$
\mu_f(t) = |\{ x\in \Omega\, :\ |f(x)| > t  \}|\ \ \textrm{for}\ \ t\geq 0,
$$
and by  $f^*:[0,\infty)\to [0,\infty]$, the decreasing rearrangement of $f$ given by
$$
f^*(s) = \inf \{t\geq 0\, :\ \mu_f(t) \leq s\}.
$$ 
Finally, by $f^{**}:(0,\infty) \to [0,\infty]$ we denote the maximal rearrangement of $f$ defined as
$$
f^{**}(s) = \frac{1}{s}\int_0^s f^*(r)\, dr.
$$ 
Now, for a homogeneous $N$-function $L:\R^d\to [0,\infty)$ by $L_{\circ}:[0,\infty)\to [0,\infty)$ we denote the homogeneous one dimensional $N$-function obeying
$$
|\{ \xi\in\R^d \,:\ L_{\circ}(|\xi|) \leq t \}| = |\{ \xi\in\R^d \,:\ L(\xi) \leq t \}|\ \ \textrm{for every}\ \ t\geq 0,
$$
i.e. $L_\circ$ is such that for every level $t$, the measure of sublevel sets for $L_{\circ}$ is the same as for $L$, $L_{\circ}$ can be understood as a kind of averaging of $L$. It can be used to define the radially increasing symmetral $L_{\bigstar}:\R^d\to [0,\infty)$ of $L$ by
$$
L_{\bigstar}(\xi) = L_\circ(|\xi|).
$$  
Finally by $L_\blacklozenge$ we denote the result of application to $L$, in sequence, the Fenchel transform, the operation of taking the radially increasing symmetral, and the Fenchel transform, again. Then, $L_\blacklozenge$ is the one dimensional homogeneous $N$-function and it is defined as 
$$
L_\blacklozenge(|\xi|) = \widetilde{\left(\widetilde{L}_{\bigstar}\right)}(\xi). 
$$
We will use these operations to the $N$-function $M_1$ given in Remark  \ref{rem:homo_minorant}, which is the homogeneous (but not necessarily isotropic) minorant of $M$, that is
$$
m_1(|\xi|) \leq M_1(\xi) \leq M(x,\xi)\ \ \textrm{for a.e.}\ \ x\in \Omega\ \ \textrm{and every}\ \ \xi\in\R^d.
$$
Thus, we will need the $N$-function
$$
(M_1)_\blacklozenge(|\xi|) = \widetilde{\left((\widetilde{M_1})_{\bigstar}\right)}(\xi). 
$$
Also, we call $\Psi_\blacklozenge:[0,\infty)\to [0,\infty)$ the increasing function given by
$$
\Psi_\blacklozenge(s) = \frac{(M_1)_\blacklozenge(s)}{s}\ \ \textrm{if}\ \ s>0\ \ \textrm{and}\ \ \Psi_\blacklozenge(0) = 0.
$$
In order to establish that the renormalized solution is bounded we need to reinforce our assumptions by the following ones.
\begin{itemize}
	\item[(W1)] There exists a constant $\lambda>1$ such that
	\begin{equation}\label{ass:f}
	\int_0^{|\Omega|}s^{\frac{1}{d}-1}\Psi_\blacklozenge^{-1}\left(\frac{\lambda}{c_A d\omega_d^{\frac{1}{d}}} s^{\frac{1}{d}}f^{**}(s)\right)\, ds < \infty, 
	\end{equation}
	where $c_A$ is given in \eqref{eq:epsilon_est}, and $\omega_d$ is the Lebesgue measure on one dimensional unit ball in $\R^d$, i.e., $\omega_d=\pi^{d/2}/\Gamma(1+\frac{d}{2})$. 
	
	\item[(W2)] The function $m$ in (A3) belongs to $L^\infty(\Omega)$. 
\end{itemize}

The following result will be proved in Section \ref{sec:renorm_weak}. Note that we obtain the result only for 'approximable' renormalized solutions, i.e., the ones which are obtained by the limiting procedure in the approximating problems.   This is similar as in the uniqueness Theorem \ref{thm:uniqueness} but there this restriction was needed only under assumption (C1). Here we need it in both cases (C1) and (C2).  
\begin{theorem}\label{thm:renorm_weak}
	Assume (A1)--(A3) and let $u$ be a renormalized solution to the problem \eqref{prblm_1}--\eqref{prblm_2} given by Definition \ref{def:renormalized}.  Assume (W1)--(W2). If either (C1) or (C2) holds and if $u$ is obtained as the limit of solutions to approximative problems \eqref{approx_1}--\eqref{approx_2} then $u$ belongs to $L^\infty(\Omega)$.   
\end{theorem}

\begin{remark}
	Condition (W1) appeared in \cite[Section 3]{Cianchi1}, where, in particular, its discussion for particular cases of $M_1$, including the polynomial growth, is provided. 
\end{remark}
\begin{remark}
	It seems to us that the assumption (W2) is not optimal and the regularity of $m$ can be relaxed. The proof in Section \ref{sec:renorm_weak}, however, appears to require that $m\in L^\infty(\Omega)$. For now we leave open the question, whether the result holds with weaker assumptions on $m$. 
\end{remark}

\begin{remark}
	An important and hard question is whether the averaging procedure used to construct $(M_1)_\blacklozenge$ could be applied to nonhomogeneous $M$ directly, rather than to its homogenous (but anisotropic) minorant $M_1$. The key issue is whether the anisotropic P\'{o}lya--Szeg\"{o}  inequality, cf. \cite[Section 4]{Cianchi1}, can be generalized to such situation.
\end{remark}

\begin{remark}
	The proof of renormalized solution existence is based on the construction of  approximating problems and on compactness methods. In the proof of  Theorem \ref{thm:renorm_weak} $L^\infty$ bound  is obtained for the approximating problems and is proved to be uniform with respect to the approximation parameter $\epsilon$. Thus, only those renormalized solutions which are the limits of approximation scheme can enjoy the $L^\infty$ bound obtained in Theorem \ref{thm:renorm_weak}. Of course, under the assumptions of Theorem  \ref{thm:uniqueness} the unique solution has to be the limit of the approximating problems and thus, it is admissible. We leave open the question if one can work with the renormalized solution directly to obtain the result of  Theorem \ref{thm:renorm_weak}.
\end{remark}

We conclude this section by the remark on the weak solutions to the  problem \eqref{prblm_1}--\eqref{prblm_2}. 

\begin{definition}\label{def:weak}
	Let $f\in L^1(\Omega)$. A function $u\in V_0^M$ is called a \textit{weak solution} to the problem \eqref{prblm_1}--\eqref{prblm_2} if
	there exists a selection $\alpha\in L_{\widetilde{M}}(\Omega)$ of $A(\cdot,\nabla u(\cdot))$ (i.e. for a.e. $x\in \Omega$ there holds $\alpha(x)\in A(x,\nabla u(x))$) such that for  any 
	$w\in W^{1,\infty}_0(\Omega)$ there holds  
	\begin{equation}\label{weak}
	\int_{\Omega}\alpha \cdot \nabla w\, dx= \int_{\Omega}f w\, dx.
	\end{equation}
\end{definition} 
The next corollary is a straightforward consequence of Theroem \ref{thm:renorm_weak}. Indeed, if a renormalized solution $u$ belongs to $L^\infty(\Omega)$ the result follows by taking $h(s)$ in \eqref{renorm} equal to one for $-\|u\|_{L^\infty(\Omega)}\leq s \leq \|u\|_{L^\infty(\Omega)}$ and the truncation level $k$ in Definition \ref{def:renormalized} greater than $\|u\|_{L^\infty(\Omega)}$. 
\begin{corollary}
Under assumptions of Theorem \ref{thm:renorm_weak} if a renormalized solution $u$ to the problem \eqref{prblm_1}--\eqref{prblm_2} is obtained as the limit of solutions to approximative problems \eqref{approx_1}--\eqref{approx_2} then it is weak. 
\end{corollary}

\section{Multivalued term and its regularization.}\label{multivalued}

\noindent \textbf{Discussion of assumptions (A1)-(A3).} In this section we discuss the assumptions and properties of the multifunction $A$ which constitutes the leading term in the studied equation. 
First of all we remark that will use interchangeably the notation $A:\Omega\times \R^d \to 2^{\R^d}$ for a multifunction that takes $x\in\Omega$ and $\xi\in \R^d$ as its arguments and $A(x)$ as a graph of a multifunction leading from $\R^d$ to $2^{\R^d}$ given by $A(x)(\xi) = A(x,\xi)$. 

We will first show that our assumptions imply that the image of a bounded set through $A(x,\cdot)$ remains bounded. Indeed, for each $\xi\in \R^d$ and $\eta\in A(x,\xi)$
$$
-m(x)+c_A(M(x,\xi)+\widetilde{M}(x,\eta)) \leq \eta\cdot \xi \leq M\left(x,\frac{2\xi}{c_A}\right)+\widetilde{M}\left(x,\frac{c_A\eta}{2}\right).
$$ 
By convexity
$$
 c_A\widetilde{M}(x,\eta) \leq  m_2\left(\frac{2|\xi|}{c_A}\right)+\widetilde{M}\left(x,\frac{c_A\eta}{2}\right)+m(x)\leq m_2\left(\frac{2|\xi|}{c_A}\right)+\frac{c_A}{2}\widetilde{M}(x,\eta)+m(x).
$$ 
It follows that
$$
 \frac{c_A}{2}\widetilde{m_2}(|\eta|) \leq \frac{c_A}{2} \widetilde{M}(x,\eta) \leq  m_2\left(\frac{2|\xi|}{c_A}\right)+m(x).
$$
Because $\widetilde{m_2}(s)\neq 0$ for $s\neq 0$, the function $\widetilde{m_2}$ must be strictly increasing, hence it is invertible, with concave inverse $(\widetilde{m_2})^{-1}$. We get
\begin{equation}\label{eq:growth_bound}
 |\eta| \leq (\widetilde{m_2})^{-1}\left(\frac{2}{c_A}m_2\left(\frac{2|\xi|}{c_A}\right)+\frac{2}{c_A}m(x)\right)\leq C_1 \frac{2}{c_A}m_2\left(\frac{2|\xi|}{c_A}\right)+C_1\frac{2}{c_A}m(x) + C_2,
\end{equation}
where $C_1$ and $C_2$ are some nonnegative constants, the existence of which follows from the concavity of $(\widetilde{m_2})^{-1}$. Now, as $m_2$ is bounded on bounded sets we obtained the desired property for $A(x,\cdot)$. 

Assumption (A3) encompasses in one formula the coercivity and growth conditions typically assumed to get the solution existence for elliptic problems. Indeed, suppose that
 for a.e. $x\in \Omega$, every $\xi\in \R^d$ and $\eta\in A(x,\xi)$ there hold the two conditions  which are anisotropic and nonhomogeneous versions of the coercivity and growth conditions, respectively
\begin{align}
& c_A M(x,\xi) - m_A(x) \leq \eta\cdot\xi,\label{coerc}\\
& \widetilde{M}(x,\eta) \leq c_G M\left(x,\xi\right) + m_G(x),\label{growth}
\end{align}
with the constants $c_A, c_G > 0$ and $m_A, m_G\in L^1(\Omega)$. Then 
\begin{align*}
&\eta\cdot \xi \geq \frac{c_A}{2}M(x,\xi) + \frac{c_A}{2c_G}\widetilde{M}(x,\eta) - m_A(x)  - \frac{1}{c_G}m_G(x) \\
& \qquad \geq \frac{1}{2}\min\left\{1, c_A,\frac{c_A}{c_G}\right\}(M(x,\xi)+\widetilde{M}(x,\eta)) - m_A(x)  - \frac{1}{c_G}m_G(x) .
\end{align*}
Clearly, conditions \eqref{coerc} and \eqref{growth} imply (A3). On the other hand it is visible that (A3) implies \eqref{coerc}. As for \eqref{growth}, from (A3) we can deduce its weaker version
\begin{align}
& \widetilde{M}(x,\eta) \leq \frac{2}{c_A} M\left(x,\frac{2}{c_A}\xi\right) + \frac{2}{c_A}m(x).\label{weaker_growth}
\end{align}
Indeed, if (A3) holds, then, if only $\eta\in a(x,\xi)$,
\begin{align*}
& c_A(M(x,\xi) + \widetilde{M}(x,\eta)) \leq \xi\cdot\eta  + m(x)  \leq \widetilde{M}\left(x,\frac{c_A}{2}\eta\right) + M\left(x,\frac{2}{c_A}\xi\right) + m(x) \\
& \qquad \leq \frac{c_A}{2}\widetilde{M}\left(x,\eta\right) + M\left(x,\frac{2}{c_A}\xi\right) + m(x),
\end{align*}
and \eqref{weaker_growth} follows. 
If we suppose that $M$ satisfies $\Delta_2$ then (A3) becomes equivalent to \eqref{coerc}--\eqref{growth}. Indeed, $\Delta_2$ implies
$$
M\left(x,\frac{2}{c_A}\xi\right) \leq c_1 M(x,\xi) +h_1(x),
$$
with $c_1>0$ and $h_1\in L^1(\Omega)$, whence \eqref{weaker_growth} implies \eqref{growth}.

From the maximal monotonicity of $A(x,\cdot)$ it immediately follows that sets $A(x,\xi)$ must be closed (and hence compact, by boundedness) and convex for every $\xi\in \R^d$ and almost every $x\in \Omega$. Moreover the graphs $A(x)$ must be closed sets in $\R^d\times \R^d$. Finally, for every $\xi\in \R^d$ and a.e. $x\in \Omega$ the set $A(x,\xi)$ must be nonempty, cf. \cite[Corollary 12.39]{rockafellar}.

The fact that sets $A(x)$ are closed implies that assumption (A1) is equivalent to measurability of the graph of $A$, i.e. the set 
\begin{equation}\label{set:e}
\mathcal{E} = \{ (x,\xi,\eta)\in \Omega\times \R^d\times \R^d\, :\ \eta\in A(x,\xi)  \}
\end{equation}
in $\mathcal{L}(\Omega)\otimes \mathcal{B}(\R^d)\otimes \mathcal{B}(\R^d)$, cf. \cite[Theorem 1.3]{Chiado}. Moreover, this assumption implies that, by Aumann and von Neumann theorem, there exists a measurable (from $\mathcal{L}(\Omega)\otimes \mathcal{B}(\R^d) $ to $\mathcal{B}(\R^d)$) selection $a:\Omega\times \R^d \to \R^d$, cf. \cite[Theorem 1.4]{Chiado}.   

\medskip

\noindent \textbf{Regularization of $a$.} By $B(x_0,r)$ we denote the open Euclidean ball in $\R^d$. Let $\phi\in C^\infty_0(B(0,1))$ be a standard mollifier (i.e. a symmetric and nonnegative function such that $\int_{B(0,1)} \phi(s)\, ds = 1$) and let $\phi^\epsilon(s) = \epsilon^{-n}\phi(x/\epsilon)$. Then, of course, $\textrm{supp}\, \phi^\epsilon\subset B(0,\epsilon)$. Define
\begin{equation}\label{aeps}
a^\epsilon(x,\xi) = (a(x,\cdot)\ast \phi^\epsilon)(\xi) = \int_{B(0,\epsilon)}\phi^\epsilon(\lambda) a(x,\xi-\lambda)\, d\lambda.
\end{equation}
The following result summarizes the properties of $a^\epsilon$.
\begin{lemma}
The regularized function $a^\epsilon:\Omega \times \R^d\to \R^d$ satisfies the following properties
\begin{itemize}
\item[1.] $a^\epsilon$ is a Carath\'{e}odory function.
\item[2.] $a^\epsilon(x,\cdot)$ is maximally monotone for almost every $x\in \Omega$.
\item[3.]  For almost every $x \in \Omega$  and almost every $\xi\in \R^d$ we have the pointwise convergence $a^\epsilon(x,\xi)\to a(x,\xi)$ as $\epsilon\to 0$. 
\item[4.] There exists the nonnegative function $m\in L^1(\Omega)$ and the constant $c_A\in (0,1]$ (independent of $\epsilon$ but different than in (A3)) such that
\begin{equation}\label{eq:epsilon_est}
a^\epsilon(x,\xi)\cdot \xi \geq c_A(M(x,\xi)+\widetilde{M}(x,a^\epsilon(x,\xi))) - m(x)\ \textrm{for almost every}\
 x\in \Omega\ \textrm{and for every}\ \xi\in \R^d. 
\end{equation}
\end{itemize}
\end{lemma}
\begin{proof}
Clearly, $a^\epsilon(x,\cdot)\in C^\infty(\R^d)$. We prove that $a^\epsilon(\cdot,\xi)$ is measurable for every $\xi\in \R^d$. As $a^\epsilon(x,\cdot)$ is continuous, it is sufficient to prove the assertion for almost every $\xi\in \R^d$. By the Fubini--Tonelli theorem it suffices to show that for a fixed $\xi\in \R^d$ the function $(x,\lambda)\mapsto \phi^\epsilon(\lambda)a(x,\xi-\lambda)$ is summable over  $\Omega\times B(0,\epsilon)$. Indeed,
$$
|\phi^\epsilon(\lambda)a(x,\xi-\lambda)| \leq \frac{C}{\epsilon^{n}}\left(m_2\left(C(|\xi|+\epsilon)\right) + m(x)+1\right),
$$
and the assertion follows. To prove 2. it is enough to verify that $a^\epsilon(x,\cdot)$ is monotone, as we already know that it is continuous and single-valued. We have
$$
(a^\epsilon(x,\xi) - a^\epsilon(x,\eta))\cdot (\xi-\eta) = \int_{B(0,\epsilon)}(a(x,\xi-\lambda) - a(x,\eta-\lambda))\cdot ((\xi-\lambda)-(\eta-\lambda))\phi^\epsilon(\lambda)\, d\lambda \geq 0.
$$ 
The assertion 3. is clear. To prove 4. we use the Young and the Fenchel--Young inequalities and we estimate
\begin{align*}
& a^\epsilon(x,\xi)\cdot \xi = \int_{B(0,\epsilon)}\phi^{\epsilon}(\lambda)a(x,\xi-\lambda)\cdot (\xi-\lambda+\lambda)\, d\lambda\\
&\qquad  \geq \int_{B(0,\epsilon)}\phi^{\epsilon}(\lambda)\left(c_A(M(x,\xi-\lambda)+\widetilde{M}(x,a(x,\xi-\lambda)))-m(x)+a(x,\xi-\lambda)\cdot \lambda\right)\, d\lambda \\
& \qquad \geq  \int_{B(0,\epsilon)}\phi^{\epsilon}(\lambda)\left(c_AM(x,\xi-\lambda)+\frac{c_A}{2}\widetilde{M}(x,a(x,\xi-\lambda))-M\left(x,\frac{2}{c_A}\lambda\right)\right)\, d\lambda -m(x) \\
&\qquad \geq c_AM\left(x,\int_{B(0,\epsilon)}\phi^{\epsilon}(\lambda)(\xi-\lambda)\, d\lambda\right)+\frac{c_A}{2}\widetilde{M}\left(x,\int_{B(0,\epsilon)}\phi^{\epsilon}(\lambda)a(x,\xi-\lambda)\, d\lambda\right)\\
&\qquad\qquad -\int_{B(0,\epsilon)}\phi^{\epsilon}(\lambda)m_2\left(\frac{2}{c_A}|\lambda|\right)\, d\lambda -m(x)\\
& \qquad \geq c_AM\left(x,\xi\right)+\frac{c_A}{2}\widetilde{M}\left(x,a^\epsilon(x,\xi)\right) -m_2\left(\frac{2}{c_A}\right) -m(x).
\end{align*}
The proof is complete.
\end{proof}
\begin{remark}
	Note that $m$ for $a^\epsilon$ does not depend on $\epsilon$ and is strictly greater than the corresponding function for $A$. Hence one can choose the same $m$ both for  $A$, and for  $a^\epsilon$. 
\end{remark}
\medskip

\noindent \textbf{Minty transformation and its properties.} Following \cite{Francfort} for almost every $x\in \Omega$ we define the mapping 
$\varphi_{A(x)}:\R^d\to 2^{\R^d}$ by the following \textit{Minty transformation}
\begin{equation}\label{minty_tr}
\mu \in \varphi_{A(x)}(\nu)\  \Leftrightarrow \ \exists (\xi,\eta)\in A(x)\, :\ \nu=\xi+\eta, \mu =\eta-\xi.
\end{equation}

The following Lemma was proved in \cite[Lemma 2.1]{Francfort}.
\begin{lemma}
The domain of $\varphi_{A(x)}$ is the whole $\R^d$, its values are singletons, and it is $1$-Lipschitz, i.e., 
$$
|\varphi_{A(x)}(\nu_1)-\varphi_{A(x)}(\nu_2)| \leq |\nu_1-\nu_2|\quad \textrm{for every}\quad \nu_1,\nu_2\in \R^d. 
$$ 
\end{lemma}
We can define the function $\varphi_A:\Omega\times \R^d\to \R^d$ by the formula $\varphi_A(x,\nu) = \varphi_{A(x)}(\nu)$. Following \cite[Remark 2.2]{Francfort} we prove the following lemma.
\begin{lemma}
The function $\varphi_A$ is Carath\'{e}odory.
\end{lemma}
\begin{proof}
Define the set
$$
\mathcal{F} = \{ (x,\nu,\mu)\in \Omega\times \R^d\times \R^d\,:\ \mu = \varphi_A(x,\nu) \}.
$$
Clearly, $\mathcal{F} = \Psi^{-1}(\mathcal{E})$, where
$$
\Psi(x,\nu,\mu) = \left(x,\frac{\nu-\mu}{2},\frac{\nu+\mu}{2}\right).
$$
It follows that $\mathcal{F}$ belong to the $\sigma$-algebra $\mathcal{L}(\Omega)\otimes \mathcal{B}(\R^d)\otimes \mathcal{B}(\R^d)$. By \cite[Theorem 1.3]{Chiado} we deduce that $\varphi_A$ is measurable with respect to the $\sigma$-algebra $\mathcal{L}(\Omega)\otimes \mathcal{B}(\R^d)$ on the domain and $\mathcal{B}(\R^d)$ on its range. Now, \cite[Theorem 1.2]{Chiado} implies that $\varphi_A(\cdot,\lambda)$ is measurable, and, in consequence, $\varphi_A$ in Carath\'{e}odory.
\end{proof}

\section{Proof of Theorem \ref{thm:existence}: existence.}\label{sec:ex_proof}

The proof is an adaptation to the case of multivalued leading term $A$ of the corresponding results from \cite{gwiazda_wittbold,gwiazda_corrigendum} and \cite{gwiazda_skrzypczak}. The main difficulty, which does not appear in \cite{gwiazda_wittbold,gwiazda_corrigendum,gwiazda_skrzypczak} is the fact that the methodology based on Young measures requires the mapping $\xi \mapsto a(x,\xi)$ to be continuous, which does not hold in our case, and hence we need to use the Minty transform to make it possible to use the Young measure techniques. The proof of the existence result consists of ten steps.  
\medskip

\noindent \textit{Step 1. Approximate problem.} We define the truncated problem
\begin{align}
& -\textrm{div}\, a^\epsilon(x,\nabla u_\epsilon) = T_{1/\epsilon}f \qquad \textrm{in}\qquad \Omega,\label{approx_1}\\
& u_{\epsilon}(x) = 0\qquad \textrm{on}\qquad \partial\Omega. \label{approx_2}
\end{align}
Existence of a weak solution  $u_\epsilon \in V_0^M$ with $a^\epsilon(\cdot, \nabla u_\epsilon) \in L_{\widetilde{M}}(\Omega)$ to the above problem follows from \cite[Theorem 1.5]{gwiazda_minakowski} (also see \cite[Theorem 2.1]{gwiazda_skrzypczak}) if (C2) holds and from \cite[Theorem 4.1, Proposition 4.3]{gwiazda_wittbold, gwiazda_corrigendum} if (C1) holds.  In both cases   and  for every test function $w\in C_0^{\infty}(\Omega)$ there holds
\begin{equation}\label{eq:weak}
\int_{\Omega}a^\epsilon(x,\nabla u_\epsilon)\cdot \nabla w\, dx = \int_{\Omega} T_{1/\epsilon}f w\, dx.
\end{equation}
 Note that in \cite[Theorem 1.5]{gwiazda_minakowski} the case $m\equiv 0$ is considered but the argument there can be easily generalized to the case of nonnegative function $m\in L^1(\Omega)$. We also stress that in case (C1) in \cite[Theorem 4.1, Proposition 4.3]{gwiazda_wittbold, gwiazda_corrigendum}
 the existence of the weak solution to the above problem is proved using the Galerkin method where the finite dimensional space is spanned by the eigenfuctions of $-\Delta$ with Dirichlet boundary conditions, and hence the approximative sequence $\{ u_{\epsilon m}\}_{m=1}^\infty$ constructed in the proof has regularity $C^{\infty}_0(\Omega)$. 

\medskip
 
 \noindent \textit{Step 2. A priori estimates.} We can test \eqref{eq:weak} with $w=T_{k}u_\epsilon$, for any $k>0$. 
 Indeed, if (C1) holds then, in \cite[Section 5.1]{gwiazda_wittbold} it has been proved using the Galerkin argument that there exists a sequence $\{ u_{\epsilon m} \}_{m=1}^\infty \subset C_0^\infty(\Omega)$ of smooth functions 
 such that 
 \begin{equation}\label{eq:galerkin_conv}
 u_{\epsilon m} \xrightarrow{\textrm{strongly in}\ L^1(\Omega)} u_\epsilon\quad \textrm{and}\quad\nabla u_{\epsilon m}\xrightarrow{\textrm{weakly}-*\ \textrm{in}\ L_M(\Omega)} \nabla u_\epsilon.
 \end{equation}
 We first show that \eqref{eq:weak} holds for $w$ replaced with $T_k w$. Indeed, let $ \{ T_{k,\delta} \}_{\delta}$ be a sequence of smooth functions which converge pointwise to $T_k$ and for every $s\in \R$ there holds $|T_{ k,\delta } s| \leq |T_k s|$ and $T_{ k,\delta }'(s)\leq 1$. Then
$$
 \int_{\Omega}a^\epsilon(x,\nabla u_\epsilon)\cdot \nabla T_{k,\delta} w\, dx = \int_{\Omega} T_{1/\epsilon} f T_{k,\delta} w\, dx.
$$
 We can pass to the limit $\delta\to 0$ on the right-hand side using the Lebesgue dominated convergence theorem. To pass to the limit on the left-hand side first take $a\in C^\infty(\Omega)^d$. There holds
  $$
  \int_{\Omega}a\cdot \nabla T_{k,\delta} w\, dx = -\int_{\Omega} \textrm{div} a\, T_{k,\delta} w\, dx 
  \xrightarrow{\delta \to 0} 
  -\int_{\Omega} \textrm{div} a\, T_{k} w\, dx = \int_{\Omega}a\cdot \nabla T_{k} w\, dx.
  $$
  Now, take $a \in L_{\widetilde{M}}(\Omega) = E_{\widetilde{M}}(\Omega)$. By Theorem \ref{thm:density} there exists a sequence $a_\eta \in C^\infty(\Omega)^d$ converging to $a$ in $L_{\widetilde{M}}(\Omega)$. Hence
   $$
  \int_{\Omega}a\cdot \nabla T_{k,\delta}w\, dx = \int_{\Omega}a\cdot \nabla T_{k}w\, dx + \int_{\Omega}(a-a_\eta)\cdot (\nabla T_{k,\delta}w-\nabla T_{k}w)\, dx + \int_\Omega a_\eta\cdot (\nabla T_{k,\delta}w-\nabla T_{k}w)\, dx.
  $$
  So, by the generalized H\"{o}lder inequality, cf. Lemma \ref{lem:hol},
  \begin{align*}
& \limsup_{\delta\to 0} \left|\int_{\Omega}a\cdot \nabla T_{k,\delta}w\, dx  -\int_{\Omega}a\cdot \nabla T_{k}w\, dx\right|\\
& \qquad  \leq  \limsup_{\delta\to 0} 2\|a-a_\eta\|_{L_{\widetilde{M}}(\Omega)}(\|\nabla T_{k,\delta}w\|_{L_M(\Omega)}+\|\nabla T_{k}w\|_{L_M(\Omega)}).
  \end{align*}
  Now 
  $$
  M(x,\nabla T_{k,\delta}w) \leq M(x, \nabla w) \leq m_2\left(\sup_{x\in \Omega}|\nabla w(x)|\right).
  $$
  Hence
  $$
  \int_{\Omega}a\cdot \nabla T_{k,\delta}w\, dx   \to \int_{\Omega}a\cdot \nabla T_{k}w\, dx\quad \textrm{for every}\quad a\in E_{\widetilde{M}}(\Omega),
  $$
  and 
  $$
  \int_{\Omega}a^\epsilon(x,\nabla u_\epsilon)\cdot \nabla T_{k} w\, dx = \int_{\Omega} T_{1/\epsilon} f T_{k} w\, dx.
  $$
  We can take $w=u_{\epsilon m}$, the approximative sequence from \eqref{eq:galerkin_conv}, and pass to the limit with $m$ to infinity. Passing to the limit on the right-hand side is again straightforward. Hence
  $$
  \lim_{m\to\infty}\int_{\Omega}a^\epsilon(x,\nabla u_\epsilon)\cdot \nabla T_{k} u_{\epsilon m}\, dx = \int_{\Omega} T_{1/\epsilon} f T_{k} u_\epsilon\, dx,
  $$
  and we can pass to the limit in the left-hand side from density of $C^\infty(\Omega)^d$ in $E_{\widetilde{M}}(\Omega)$ and a priori estimates
  $$
  \int_{\Omega} M(x,\nabla T_k u_{\epsilon m})\, dx \leq \int_{\Omega} M(x,\nabla u_{\epsilon m})\, dx \leq C,
  $$
 derived in  \cite[Section 5.1]{gwiazda_wittbold}, where a constant $C$ does not depend on $m$.  
 In turn, if (C2) holds, then, as 
 $T_{k} u_\epsilon \in V_0^M \cap L^\infty(\Omega)$, by Theorem \ref{thm:approx} we can approximate $T_{k} u_\epsilon$ by a sequence of functions $\{ u_{\epsilon k m} \}_{m=1}^\infty\subset C_0^\infty(\Omega)$ such that $$
 u_{\epsilon k m} \xrightarrow{\textrm{strongly in}\ L^1(\Omega)} T_k u_\epsilon\quad \textrm{and moreover}\quad  \nabla u_{\epsilon k m} \xrightarrow{M} \nabla T_{k}(u_\epsilon).$$ 
 We can substitute $u_{\epsilon k m}$ as the test function in \eqref{eq:weak} and pass to the limit $m\to \infty$ by  Lemma \ref{lem:modular}. In both cases (C1) and (C2) we obtain
\begin{equation}\label{eq:test}
\int_{\Omega}a^\epsilon(x,\nabla T_{k} u_\epsilon)\cdot \nabla T_{k} u_\epsilon\, dx = \int_{\Omega} T_{1/\epsilon} f T_{k} u_\epsilon\, dx.
\end{equation}
It follows that 
\begin{align}
& \int_{\Omega}M(x,\nabla T_k u_\epsilon)\, dx \leq C ( \|m\|_{L^1(\Omega)} + k \|f\|_{L^1(\Omega)}),\label{conv:1}\\
& \int_{\Omega}\widetilde{M}(x,a^\epsilon(x,\nabla T_k u_\epsilon))\, dx \leq C ( \|m\|_{L^1(\Omega)} + k \|f\|_{L^1(\Omega)}).\label{conv:2}
\end{align}
We also deduce that
$$
\int_{\Omega}m_1(|\nabla T_k u_\epsilon|)\, dx \leq C ( \|m\|_{L^1(\Omega)} + k \|f\|_{L^1(\Omega)}),
$$
and, by Theorem \ref{thm:convergence},
$$
\int_{\Omega}m_1\left(\frac{| T_k u_\epsilon|}{\lambda}\right)\, dx \leq C ( \|m\|_{L^1(\Omega)} + k \|f\|_{L^1(\Omega)}).
$$
It follows that
$$
\int_{\Omega}|\nabla T_k u_\epsilon|\, dx \leq C ( 1 + k )\quad \textrm{and}\quad  \int_{\Omega}| T_k u_\epsilon|\, dx \leq Ck ,
$$
and hence the sequence $\{ T_k u_\epsilon \}_{\epsilon>0}$ is bounded in $W^{1,1}(\Omega)$. Moreover, the sequences $\{ \nabla T_k u_\epsilon \}_{\epsilon>0}$ and  $ \{ a^{\epsilon}(x,\nabla T_k u_\epsilon) \}_{ \epsilon > 0}$ are uniformly integrable. 

\bigskip

\noindent \textit{Step 3. Controlled radiation.}
We first estimate the Lebesgue measure of the set
$
\{  x\in \Omega\,:\ |u_\epsilon(x)| \geq k \}.
$ By Theorem \ref{thm:convergence} there holds
\begin{align*}
& |\{  x\in \Omega\,:\ |u_\epsilon(x)| \geq k \}| \leq \int_{\Omega}\frac{m_1\left(\frac{|T_k u_\epsilon|}{\lambda}\right)}{m_1\left(\frac{k}{\lambda}\right)}\, dx\leq \frac{C}{m_1\left(\frac{k}{\lambda}\right)}\int_{\Omega}m_1(|\nabla T_k u_\epsilon|)\, dx\\
& \qquad \leq \frac{C}{m_1\left(\frac{k}{\lambda}\right)}\int_{\Omega}M(x,\nabla T_k u_\epsilon)\, dx \leq \frac{C}{m_1\left(\frac{k}{\lambda}\right)}(\|m\|_{L^1(\Omega)}+k\|f\|_{L^1(\Omega)})
\end{align*}
So, there exists a continuous function $H:(0,\infty)\to [0,\infty)$ such that $\lim_{k\to \infty} H(k) = 0$ and $$|\{  x\in \Omega\,:\ |u_\epsilon(x)| \geq k \}| \leq H(k)$$ for every $k > 0$ and $\epsilon \in (0,1)$. We deduce that
$$
\int_{\{|u_\epsilon| \geq k\}} |f(x)|\, dx \leq \omega(|\{  x\in \Omega\,:\ |u_\epsilon(x)| \geq k \}|) \leq \omega(H(k)) := \gamma(k),
$$
where $\gamma:(0,\infty)\to [0,\infty)$ is a continuous function such that $\lim_{k\to \infty} \gamma(k) = 0$. Now, define 
$$\psi_k(r) = T_{k+1}(r) - T_{k}(r) = \begin{cases}
-1 & \quad \textrm{if}\quad r \leq -k-1,\\
r+k & \quad \textrm{if}\quad -k-1<r < -k,\\
0 & \quad \textrm{if}\quad |r|\leq k,\\
r-k & \quad \textrm{if}\quad k< r < k+1,\\
1 & \quad \textrm{if}\quad k+1 \leq r.
\end{cases}$$
Subtracting \eqref{eq:test} written for $k+1$ and for $k$ we obtain
$$
\int_{\Omega} a^\epsilon(x,\nabla u_\epsilon)\cdot \nabla \psi_k(u_\epsilon)\, dx = \int_{\Omega} T_{1/\epsilon} f\psi_k(u_\epsilon)\, dx.
$$
Hence
\begin{align*}
&\int_{\{k<|u_\epsilon|<k+1\}}a^\epsilon(x,\nabla u_\epsilon)\cdot \nabla u_\epsilon\, dx =  \int_{\{k<|u_\epsilon|<k+1\}}a^\epsilon(x,\nabla u_\epsilon)\cdot \nabla T_{k+1}u_\epsilon\, dx\\
& \quad = \int_{\{k<|u_\epsilon|\}}a^\epsilon(x,\nabla u_\epsilon)\cdot \nabla T_{k+1}u_\epsilon\, dx = \int_{\{k<|u_\epsilon|\}}a^\epsilon(x,\nabla u_\epsilon)\cdot \nabla( T_{k+1}u_\epsilon-T_ku_\epsilon)\, dx\\
& \quad = \int_{\Omega}a^\epsilon(x,\nabla u_\epsilon)\cdot \nabla( T_{k+1}u_\epsilon-T_ku_\epsilon)\, dx = \int_{\Omega} a^\epsilon(x,\nabla u_\epsilon)\cdot \nabla \psi_k(u_\epsilon)\, dx = \int_{\Omega} T_{1/\epsilon} f\psi_k(u_\epsilon)\, dx.
\end{align*}
It follows that
\begin{equation}\label{cont_rad}
\int_{\{k<|u_\epsilon|<k+1\}}a^\epsilon(x,\nabla u_\epsilon)\cdot \nabla u_\epsilon\, dx \leq \int_{\{|u_\epsilon|\geq k\}}|T_{1/\epsilon}f|\, dx \leq \int_{\{|u_\epsilon|\geq k\}}|f|\, dx \leq \gamma(k).
\end{equation}

\bigskip

\noindent \textit{Step 4. Convergences which follow directly from a priori estimates.} 
Let $p\in (1, \frac{d}{d-1})$   be any fixed exponent. For every $k$ there exists a subsequence of $\epsilon\to 0$ such that, for this nonrenumbered subsequence, there hold the following convergences  
\begin{align*}
& T_k(u_\epsilon) \xrightarrow{\epsilon\to 0} g_k \quad \textrm{strongly in}\quad L^p(\Omega)\quad \textrm{and weakly-* in}\quad L^\infty(\Omega),\\
& \nabla T_k u_\epsilon \xrightarrow{\epsilon\to 0} \nabla g_k \quad \textrm{weakly in}\quad L^1(\Omega)^d \quad \textrm{and weakly-*  in}\quad L_M(\Omega).
\end{align*}
We prove that, for a nonrenumbered subsequence, the sequence $\{ u_{\epsilon} \}$ is Cauchy in measure. Indeed
$$\{|u_{\epsilon_1}-u_{\epsilon_2}|\geq \delta\}\subset \{|T_k u_{\epsilon_1}-T_ku_{\epsilon_2}|\geq \delta\}\cup\{|u_{\epsilon_1}|>k\}\cup \{|u_{\epsilon_2}|>k\},$$
whence
$$
|\{|u_{\epsilon_1}-u_{\epsilon_2}|\geq \delta\}| \leq |\{|T_ku_{\epsilon_1}-T_ku_{\epsilon_2}|\geq \delta\}| + 2H(k).
$$
The fact that, for some subsequence (possibly different for different $k$), $\{ T_k u_\epsilon \}_{\epsilon}$ is Cauchy in measure, together with a diagonal argument, implies that there exists a measurable function $u:\Omega\to \R$ such that, for a subsequence 
\begin{align}
& u_\epsilon \xrightarrow{\epsilon\to 0} u\quad \textrm{in measure and for a.e.}\quad x\in \Omega,\label{conv:pointwise}\\
& \textrm{for every}\quad  r > 0 \quad \lim_{\epsilon_0\to 0}\left|\bigcup_{\epsilon\in (0,\epsilon_0)} \{ |u-u_\epsilon| > r \}\right| = 0,\nonumber
\end{align}
where the last assertion means that $u_\epsilon\to u$ almost uniformly. It follows that $g_k = T_ku$ and the subsequences such that the convergences 
\begin{align}
& T_ku_\epsilon \xrightarrow{\epsilon\to 0} T_ku \quad \textrm{strongly in}\quad L^p(\Omega)\quad \textrm{and weakly-* in}\quad L^\infty(\Omega),\\
& \nabla T_ku_\epsilon \xrightarrow{\epsilon\to 0} \nabla T_ku \quad \textrm{weakly in}\quad L^1(\Omega)^d \quad \textrm{and weakly-*  in}\quad L_M(\Omega)\label{eq:lmconv}
\end{align}
hold, coincide for all $k$. 
Moreover, for this subsequence,
\begin{align*}\label{hconv}
&\nabla h(u_\epsilon) \to \nabla h(u)\quad \textrm{for every}\quad h\in C^1_c(\Omega) \quad \textrm{weakly in}\quad L^1(\Omega)^d \quad \textrm{and weakly-*  in}\quad L_M(\Omega),\\
& h(u_\epsilon) \to h(u)\quad \textrm{for every}\quad h\in C^1_c(\Omega) \quad \textrm{pointwise in}\quad \Omega. 
\end{align*}
Fixing $k$, we deduce from \eqref{conv:2} that there exists a subsequence of indexes (depending on $k$) such that, for this subsequence
\begin{equation}\label{aconv}
a^\epsilon (x,\nabla T_ku_\epsilon) \xrightarrow{\epsilon\to 0} \alpha_k\quad  \textrm{weakly-* in}\quad L_{\widetilde{M}}(\Omega)\quad  \textrm{and weakly in}\quad L^1(\Omega)^{d}.
\end{equation}
Fix $k>0$. Denote $B_k = \{ |u| < k \}$ and $B_{\epsilon_0,r} = \bigcup_{\epsilon<\epsilon_0}\{ |u_\epsilon-u| > r \}$. Let $l>k$ be given and fix a subsequence of $\epsilon$ such that $\eqref{aconv}$ holds for the index $l$. Choose $m>k$ and $m\neq l$. We will prove that, for another subsequence of $\epsilon$ there holds $\alpha_m=\alpha_l$ on $B_k$.  Fix $k_1\in (k,\min\{l,m\})$. Choose  $\delta > 0$. There exists $\epsilon_0>0$ such that $|B_{\epsilon_0,k_1-k}| < \delta$.  It is easy to see that $B_k\setminus B_{\epsilon_0,k_1-k} \subset \{ |u_\epsilon| < k_1 \} $. This means that there holds $\nabla T_lu_\epsilon = \nabla T_mu_\epsilon$  on the set $B_k\setminus B_{\epsilon_0,k_1-k}$, whence $\alpha_l=\alpha_m$ a.e. on this set. Diagonal argument with respect to $\delta$ implies that there exists a subsequence of $\epsilon$ such that $\alpha_m = \alpha_l$ a.e. on whole $B_k$. Let $Z\subset[0,\infty)$ be a countable and dense set. From the diagonal argument with respect to $k\in Z$ and $l,m\in Z$ we deduce that there exists a sequence $\epsilon\to 0$ such that for every $k\in Z$ and every $l,m \in Z$ with $l,m > k$ there holds $\alpha_l=\alpha_m$ on $B_k$. It follows that there exists a measurable function $\alpha:\Omega \to \R^d$ such that $\alpha=\alpha_k$ on $B_k$ for every $k\in Z$. If $k\notin Z$, we can find a sequence $k_n\in Z$ such that $k_n\to k^-$, and, as $\alpha_{k} = \alpha$ on $B_{k_n}$, we deduce that $\alpha_k = \alpha$ on $B_k$ for every $k>0$. Note that the sequence of $\epsilon$ is the same for every $k\in Z$ and may depend of $k$ for $k\notin Z$.

\bigskip

\noindent \textit{Step 5. An auxiliary equality.} In this step we should prove that for every $v\in C^\infty(\overline{\Omega})$
\begin{equation}\label{limsup}
\lim_{\epsilon\to 0^+}\int_{\Omega} a^\epsilon (x,\nabla T_ku_\epsilon)\cdot \nabla T_ku_\epsilon v \, dx =  \int_{\Omega}\alpha_k \nabla T_ku v \, dx.
\end{equation}
The proof will proceed separately for the case of (C1) and (C2).  

\medskip

\noindent \textit{Step 5.1. The case (C1).} The proof follows the lines of the argument in \cite[Step 2]{gwiazda_corrigendum}. Define
$$
h_l(r) = \min \{ (l+1-|r|)^+,1 \} = \begin{cases}
1\ \textrm{if}\ |r| \leq l, \\
l+1-|r|\ \textrm{if}\ l < |r| <  l+1,\\
0\ \textrm{if}\ l+1\leq |r|. 
\end{cases}
$$
Let $v\in C^\infty(\overline{\Omega})$. We take $w = v h_l(u_\epsilon)(T_{k} u_\epsilon-T_k u_{\delta m})$ as a test function in \eqref{eq:weak}. Indeed, such choice is possible. To justify this we proceed similarly as in Step 2. If $\{h_{l,i}\}_{i=1}^\infty$ and $\{ T_{k,i}\}_{i=1}^\infty$ are smooth nonnegative functions, which approximate $h_l$ and $T_k$ pointwise from below, $a \in L_{\widetilde{M}}(\Omega)$, and $a_\eta\in C_0^\infty(\Omega)^d$ is the approximative sequence which exists by Theorem \ref{thm:density}, we only have to estimate
\begin{align*}
& \left|\int_{\Omega}(a-a_\eta)\cdot (\nabla (v h_{l,i}(u_{\epsilon i})(T_{k,i} u_{\epsilon i}-T_{k,i} u_{\delta m})) - \nabla (v h_l(u_{\epsilon})(T_k u_{\epsilon}-T_k u_{\delta m}))) \, dx\right|\\
&\quad \leq \left|\int_{\Omega}(a-a_\eta)\cdot \nabla v h_{l,i}(u_{\epsilon i})(T_{k,i} u_{\epsilon i}-T_{k,i}u_{\delta m})\, dx\right|\\
& \qquad \qquad  + \left|\int_{\Omega}(a-a_\eta)\cdot \nabla u_{\epsilon i} v h_{l,i}'(u_{\epsilon i})(T_{k,i}u_{\epsilon i}-T_{k,i}u_{\delta m}) \, dx\right|\\
&\qquad \qquad + \left|\int_\Omega(a-a_\eta)\cdot (\nabla T_{k,i}u_{\epsilon i}- \nabla T_{k,i}u_{\delta m}) v h_{l,i}(u_{\epsilon i})\, dx\right|\\
&\quad + \left|\int_{\Omega}(a-a_\eta)\cdot \nabla v h_{l}(u_{\epsilon})(T_{k}u_{\epsilon}-T_{k}u_{\delta m})\, dx\right|\\
& \qquad \qquad  + \left|\int_{\Omega}(a-a_\eta)\cdot \nabla u_{\epsilon} v h_{l}'(u_{\epsilon})(T_{k}u_{\epsilon}-T_{k}u_{\delta m}) \, dx\right|\\
&\qquad \qquad + \left|\int_\Omega(a-a_\eta)\cdot (\nabla T_{k}u_{\epsilon}- \nabla T_{k}u_{\delta m}) v h_{l}(u_{\epsilon})\, dx\right|\\
&\leq \|a-a_\eta\|_{L_{\widetilde{M}}(\Omega)}\|v\|_{W^{1,\infty}(\Omega)}\left(4k+(4k+1)(\|\nabla u_{\epsilon i}\|_{L_M(\Omega)}+\|\nabla u_{\epsilon}\|_{L_M(\Omega)})+ 2\|\nabla u_{\delta m}\|_{L_M(\Omega)}\right)\\
&\qquad \leq  \|a-a_\eta\|_{L_{\widetilde{M}}(\Omega)} c(\epsilon,\delta,k),
\end{align*} 
where the last constant does not depend on $i$ from a priori estimates derived in \cite[Section 5.1]{gwiazda_wittbold}. Now, \eqref{eq:weak} with $w = v h_l(u_\epsilon)(T_{k}u_\epsilon-T_ku_{\delta m})$ takes the form 
$$
\int_{\Omega}a^\epsilon(x,\nabla u_\epsilon)\cdot \nabla (v h_l(u_\epsilon)(T_{k}u_\epsilon-T_ku_{\delta m}))\, dx = \int_{\Omega} T_{1/\epsilon}f v h_l(u_\epsilon)(T_{k}u_\epsilon-T_ku_{\delta m})\, dx.
$$
Since $|v T_{1/\epsilon}f h_l(u_\epsilon)(T_{k}u_\epsilon-T_ku_{\delta m})| \leq 2k|f|\|v\|_{L^\infty(\Omega)}$ we can pass to the limit using the Lebesgue dominated convergence theorem on the right-hand side, whence
$$
\lim_{\delta\to 0^+}\lim_{m\to \infty}\lim_{\epsilon\to 0^+} \int_{\Omega}v T_{1/\epsilon}f h_l(u_\epsilon)(T_{k}u_\epsilon-T_ku_{\delta m})\, dx = 0.
$$
To deal with the left-hand side note that
\begin{align*}
& \int_{\Omega}a^\epsilon(x,\nabla u_\epsilon)\cdot \nabla (v h_l(u_\epsilon)(T_{k}u_\epsilon-T_ku_{\delta m}))\, dx \\
& \ \ = \int_{\Omega}a^\epsilon(x,\nabla u_\epsilon)\cdot \nabla u_\epsilon v h_l'(u_{\epsilon})(T_{k}u_\epsilon-T_ku_{\delta m})\, dx\\
&\qquad \qquad   + \int_{\Omega}a^\epsilon(x,\nabla u_\epsilon)\cdot \nabla v h_l(u_\epsilon)(T_{k}u_\epsilon-T_ku_{\delta m})\, dx\\
&\qquad \qquad  + \int_{\Omega}a^\epsilon(x,\nabla u_\epsilon)\cdot (\nabla T_{k}u_\epsilon-\nabla T_ku_{\delta m}) v h_l(u_\epsilon) \, dx\\
& \ \ = I_1 + I_2 + I_3.
\end{align*}
Now 
\begin{align*}
&
I_1 = -\int_{\{ l< u_\epsilon < l+1\}} a^\epsilon(x,\nabla u_\epsilon) \cdot \nabla u_\epsilon v(T_{k}u_\epsilon-T_ku_{\delta m})\, dx\\
&\qquad  + \int_{\{ -l-1< u_\epsilon < -l \}} a^\epsilon(x,\nabla u_\epsilon)\cdot  \nabla u_\epsilon v(T_{k}u_\epsilon-T_ku_{\delta m})\, dx.
\end{align*}
Using \eqref{eq:epsilon_est}, we obtain
\begin{align*}
&|I_1| \leq \|v\|_{L^\infty(\Omega)}\int_{\Omega}m(x)|T_ku_\epsilon-T_ku_{\delta m}|\, dx + 2k\|v\|_{L^\infty(\Omega)}\int_{\{l < |u_\epsilon| < l+1\}} a^\epsilon(x,\nabla u_\epsilon) \nabla u_\epsilon + m(x)\, dx\\
&\quad \leq \|v\|_{L^\infty(\Omega)}\int_{\Omega}m(x)|T_ku_\epsilon-T_ku_{\delta m}|\, dx + 2k\|v\|_{L^\infty(\Omega)}\int_{\{l \leq |u_\epsilon|\} } |f(x)|+ m(x)\, dx\\
&\quad \leq \|v\|_{L^\infty(\Omega)}\int_{\Omega}m(x)|T_ku_\epsilon-T_ku_{\delta m}|\, dx + 2k\|v\|_{L^\infty(\Omega)}\gamma(l).
\end{align*}
We can pass to the limit in the first term by the Lebesgue dominated convergence theorem, whence
$$
\lim_{l\to \infty}\lim_{\delta\to 0^+}\lim_{m\to\infty}  \lim_{\epsilon\to 0^+}|I_1| = 0.
$$
We deal with $I_2$. 
$$
I_2 = \int_{\Omega}a^\epsilon(x,\nabla T_{l+1}u_\epsilon)\cdot \nabla v h_l(u_\epsilon)(T_{k}u_\epsilon-T_ku_{\delta m})\, dx
$$
Now note that $ a^\epsilon(x,\nabla T_{l+1}u_\epsilon)\cdot \nabla v  \to \alpha_l \cdot \nabla v$ weakly in  $L^1$ as $\epsilon\to 0$. Indeed if $\psi\in L^{\infty}(\Omega)$ then
$$
\int_{\Omega} a^\epsilon(x,\nabla T_{l+1}u_\epsilon)\cdot \nabla v \psi\, dx \to \int_{\Omega} \alpha_l \cdot \nabla v \psi\, dx.
$$
As
$$
h_l(u_\epsilon)(T_{k}u_\epsilon-T_ku_{\delta m}) \to h_l(u)(T_{k}u-T_ku_{\delta m})\quad \textrm{pointwise}\quad \textrm{as}\quad \epsilon\to 0,
$$
and $|h_l(u_\epsilon)(T_{k}u_\epsilon-T_ku_{\delta m})|\leq 2k$,
by Lemma \ref{convergence_lemma} we deduce that
$$
\lim_{\epsilon\to 0^+} I_2 = \int_{\Omega}\alpha_l\cdot \nabla v h_l(u)(T_{k}u-T_ku_{\delta m})\, dx.
$$
We can pass with $m$ to infinity and with $\delta \to 0$, whence
$$
\lim_{\delta\to 0^+} \lim_{m\to\infty} \lim_{\epsilon\to 0^+} I_2 = 0.
$$
We deduce that 
$$
\lim_{l\to \infty} \lim_{\delta\to 0^+} \lim_{m\to\infty} \lim_{\epsilon\to 0^+}  \int_{\Omega}a^\epsilon(x,\nabla u_\epsilon)\cdot (\nabla T_{k}u_\epsilon-\nabla T_ku_{\delta m})v h_l(u_\epsilon) \, dx= 0,
$$
whence
$$
\lim_{l\to \infty} \lim_{\delta\to 0^+}\lim_{m\to\infty}  \lim_{\epsilon\to 0^+}  \int_{\Omega}a^\epsilon(x,\nabla T_{l+1}u_\epsilon)\cdot  (\nabla T_{k}u_\epsilon-\nabla T_ku_{\delta m})v h_l(u_\epsilon)\, dx = 0.
$$
We need to show that
\begin{equation}\label{eq:aux10}
\lim_{l\to \infty} \lim_{\delta\to 0^+} \lim_{m\to\infty}\lim_{\epsilon\to 0^+} \int_{\Omega}a^\epsilon(x,\nabla T_k u_\epsilon)\cdot (\nabla T_{k}u_\epsilon-\nabla T_ku_{\delta m})v h_l(u_\epsilon) \, dx = 0.
\end{equation}
To this end we choose $k<l$ and we study the expression
\begin{align*}
& \int_{\Omega}(a^\epsilon(x,\nabla T_{k}u_\epsilon)-a^\epsilon(x,\nabla T_{l+1}u_\epsilon)) \cdot  (\nabla T_{k}u_\epsilon-\nabla T_ku_{\delta m})v h_l(u_\epsilon)\, dx\\
&\quad = \int_{\Omega}(a^\epsilon(x,\nabla T_{l+1}u_\epsilon) - a^\epsilon(x,0))\cdot \nabla T_ku_{\delta m} v \chi_{\{ |u_\epsilon| > k \}} h_l(u_\epsilon) \, dx\\
& \quad = \int_{\Omega\setminus\{|u|=k\}}(a^\epsilon(x,\nabla T_{l+1}u_\epsilon) - a^\epsilon(x,0))\cdot  \nabla T_ku_{\delta m} v \chi_{\{ |u_\epsilon| > k \}} h_l(u_\epsilon)\, dx \\
& \qquad \qquad + \int_{\{ |u|=k \}}(a^\epsilon(x,\nabla T_{l+1}u_\epsilon) - a^\epsilon(x,0))\cdot \nabla T_ku_{\delta m} v \chi_{\{ |u_\epsilon| > k \}} h_l(u_\epsilon) \, dx\
\end{align*}
We pass to the limit $\epsilon \to 0$. To this end note that
\begin{align*}
&(a^\epsilon(x,\nabla T_{l+1}u_\epsilon)-a^\epsilon(x,0))\cdot   \nabla T_ku_{\delta m} \xrightarrow{\textrm{weakly in}\ L^1(\Omega\setminus\{|u|=k\})^d\ \textrm{as}\ \epsilon\to 0}  (\alpha_{l+1}-b(x))) \cdot \nabla T_ku_{\delta m},
\end{align*}  
where $b(x)\in L^1(\Omega\setminus\{ |u| = k \})$. Hence, by Lemma \ref{convergence_lemma}, 
\begin{align*}
&\lim_{\epsilon\to 0^+}\int_{\Omega\setminus\{|u|=k\}}(a^\epsilon(x,\nabla T_{l+1}u_\epsilon)-a^\epsilon(x,0))\cdot \nabla T_ku_{\delta m}  v \chi_{\{ |u_\epsilon| > k \}} h_l(u_\epsilon) \, dx \\
&\qquad = \int_{\Omega\setminus\{|u|=k\}}(\alpha_{l+1}-b(x))\cdot  \nabla T_ku_{\delta m} v \chi_{\{ |u| > k \}} h_l(u)\, dx.
\end{align*} 
We easily deduce
\begin{align*}
&
\lim_{\delta\to 0^+}\lim_{m\to\infty}\lim_{\epsilon\to 0^+}\int_{\Omega\setminus\{|u|=k\}}(a^\epsilon(x,\nabla T_{l+1}u_\epsilon)-a^\epsilon(x,0))\cdot  \nabla T_ku_{\delta m} v \chi_{\{ |u_\epsilon| > k \}} h_l(u_\epsilon)\, dx\\
& \qquad  = \int_{\Omega\setminus\{|u|=k\}}(\alpha_{l+1}-b(x)) \cdot \nabla T_ku v \chi_{\{ |u| > k \}} h_l(u) \, dx = 0.
\end{align*} 
To deal with the second term note that for some function $F_{k,l}\in L_{\widetilde{M}}(\{ |u| = k\})$, and for subsequence of $\epsilon$ there holds 
\begin{align*}
&(a^\epsilon(x,\nabla T_{l+1}u_\epsilon) - a^\epsilon(x,0)) \chi_{\{ |u_\epsilon| > k \}} h_l(u_\epsilon) \xrightarrow{\textrm{weakly-* in}\ L_{\widetilde{M}}(\{|u|=k\})\ \textrm{as}\ \epsilon\to 0}  F_{k,l},
\end{align*}
where on the left-hand side we consider restrictions of functions $(a^\epsilon(x,\nabla T_{l+1}u_\epsilon) - a^\epsilon(x,0)) \chi_{\{ |u_\epsilon| > k \}} h_l(u_\epsilon)$ to the set $\{|u|=k\}$. Hence
\begin{align*}
&
\lim_{\delta\to 0^+}\lim_{m\to\infty}\lim_{\epsilon\to 0^+}\int_{\{|u|=k  \}}(a^\epsilon(x,\nabla T_{l+1}u_\epsilon)-a^\epsilon(x,0))  \chi_{\{ |u_\epsilon| > k \}} h_l(u_\epsilon)\cdot  \nabla T_ku_{\delta m}v \, dx\\
& \qquad  = \int_{\{|u|=k  \}}F_{k,l}\cdot \nabla T_ku v \, dx = 0,
\end{align*} 
where the last equality holds due to the fact that
$$
\{|u|=k  \} \subset \{|T_ku|=k  \},
$$
and 
$$
\nabla T_k u = 0\quad \textrm{a.e. on the set}\quad \{|T_ku|=k  \},
$$
cf. \cite[Theorem 4.4 (iv)]{evans_gariepy}. So, \eqref{eq:aux10} holds. Now, there holds
\begin{align*}
& \int_{\Omega}a^\epsilon(x,\nabla T_k u_\epsilon)\cdot  (\nabla T_{k}u_\epsilon-\nabla T_ku_{\delta m}) v h_l(u_\epsilon) \, dx \\
& \ \ = \int_{\Omega}a^\epsilon(x,\nabla T_k u_\epsilon)\cdot (\nabla T_{k}u_\epsilon-\nabla T_ku_{\delta m})v \, dx - \int_{\Omega}a^\epsilon(x,0)\cdot \nabla T_ku_{\delta m} v (h_l(u_\epsilon)-1) \, dx,
\end{align*}

But
$$
\lim_{\epsilon\to 0^+}\int_{\Omega}a^\epsilon(x,0) \cdot  \nabla T_ku_{\delta m} v (h_l(u_\epsilon)-1)\, dx = 
\int_{\Omega}b(x)\cdot \nabla T_ku_{\delta m} v (h_l(u)-1) \, dx
$$
$$
\lim_{\delta\to 0^+}\lim_{m\to\infty}\lim_{\epsilon\to 0^+}\int_{\Omega}a^\epsilon(x,0)\cdot \nabla T_ku_{\delta m} v (h_l(u_\epsilon)-1) \, dx = 
\int_{\Omega}b(x) \cdot \nabla T_ku v(h_l(u)-1) \, dx,
$$
and due to pointwise convergence $h_l(u)-1$ to zero as $l\to \infty$ we obtain
$$
\lim_{l\to \infty}\lim_{\delta\to 0^+}\lim_{m\to\infty}\lim_{\epsilon\to 0^+}\int_{\Omega}a^\epsilon(x,0)\cdot \nabla T_ku_{\delta m} v (h_l(u_\epsilon)-1) \, dx = 0.
$$
We deduce
$$
\lim_{\delta\to 0^+} \lim_{m\to\infty}\lim_{\epsilon\to 0^+} \int_{\Omega}a^\epsilon(x,\nabla T_k u_\epsilon)\cdot (\nabla T_{k}u_\epsilon-\nabla T_ku_{\delta m})v \, dx = 0,
$$
whence
$$
\lim_{\epsilon\to 0^+} \int_{\Omega}a^\epsilon(x,\nabla T_k u_\epsilon)\cdot \nabla T_{k} u_\epsilon v \, dx = \lim_{\delta\to 0^+} \lim_{m\to\infty}\lim_{\epsilon\to 0^+} \int_{\Omega}a^\epsilon(x,\nabla T_k u_\epsilon)\cdot \nabla T_k u_{\delta m} v \, dx.
$$
It follows that
$$
\lim_{\epsilon\to 0^+} \int_{\Omega}a^\epsilon(x,\nabla T_k u_\epsilon)\cdot \nabla T_{k}(u_\epsilon)v \, dx = \int_{\Omega}\alpha_k \cdot \nabla T_k(u)v \, dx.
$$
which is the required assertion \eqref{limsup}. 

\medskip

\noindent \textit{Step 5.2. The case (C2).} The proof follows the lines of the argument in \cite[Proposition 3.2]{gwiazda_skrzypczak}.
	We first show that we can take as a test function in \eqref{eq:weak} the function
$w=v h_l(u_\epsilon)(T_k u_\epsilon-(T_ku)_\delta)$
where $v\in C^{\infty}(\overline{\Omega})$ and $\{(T_ku)_\delta\}_{\delta>0}\subset C_0^\infty(\Omega)$ is the approximating sequence of $T_k u$ which exists by Theorem \ref{thm:approx}. Here $l\geq k$ are fixed. Such a choice of $w$ is possible: clearly, $w\in L^\infty(\Omega)$. Since $v$ and $h_l(u_\epsilon)$ are bounded, while $T_k(u_\epsilon)\in V_0^M\cap L^\infty(\Omega)$, we easily check by computing the weak gradient of $w$ that $w\in W^{1,1}_0(\Omega)$ and that $\nabla w\in L_M(\Omega)$, i.e. $w\in V_0^M$. Therefore, we can apply Theorem \ref{thm:approx} in order to get an approximating sequence $w_\nu\in C_0^\infty(\Omega)$. Now, we can test \eqref{eq:weak} against each $w_\nu$ and it follows directly from Theorem \ref{thm:approx} that the right-hand side of \eqref{eq:weak} converges to the integral with $w_\nu$ replaced by $w$. On the other hand, by Lemma \ref{lem:modular} the left-hand side converges, too (observe that $\nabla w_\nu$ is indeed bounded in $L_M(\Omega)$ by the triangle inequality, since we have that $\nabla w_\nu\stackrel{M}{\longrightarrow}\nabla w$ and we know that $\nabla w\in L_M(\Omega)$) and thus we get
$$
\int_\Omega  a^\epsilon(x,\nabla T_k u_\epsilon)\cdot\nabla(v h_l(u_\epsilon)(T_k u_\epsilon-(T_ku)_\delta))\ dx=\int_\Omega T_{1/\epsilon} f v h_l(u_\epsilon)(T_k u_\epsilon-(T_ku)_\delta)\ dx.
$$
Now, due to the Lebesgue dominated convergence theorem applied twice we obtain that the right-hand side converges to zero
$$
\lim_{\epsilon\to 0^+}\lim_{\delta\to 0^+}\int_\Omega T_{1/\epsilon} f v h_l(u_\epsilon) (T_ku_\epsilon-(T_ku)_\delta)\ dx=0.
$$
Indeed, it is enough to observe that from the estimate in Theorem \ref{thm:approx}, we have the pointwise bound
\begin{align*}
|T_k u_\epsilon-(T_ku)_\delta|\leq k+\|(T_ku)_\delta\|_{L^\infty(\Omega)}\leq k+c(\Omega)\|T_ku\|_{L^\infty(\Omega)}\leq k(1+c(\Omega)).
\end{align*}
In order to deal with the left-hand side we compute the weak gradient and obtain three integrals to estimate
\begin{align*}
&\int_\Omega  a^\epsilon(x,\nabla T_k u_\epsilon)\cdot\nabla( v h_l(u_\epsilon) (T_ku_\epsilon-(T_ku)_\delta))\ dx\\
& \qquad =
\int_\Omega  a^\epsilon(x,\nabla T_k u_\epsilon)\cdot v h'_l(u_\epsilon) \nabla u_\epsilon (T_ku_\epsilon-(T_ku)_\delta)\ dx\\
&\qquad \quad +\int_\Omega a^\epsilon(x,\nabla T_k u_\epsilon)\cdot\nabla v h_l(u_\epsilon) (T_k u_\epsilon-(T_ku)_\delta)\ dx\\
& \qquad\qquad +
\int_\Omega  a^\epsilon(x,\nabla T_k u_\epsilon)\cdot \nabla(T_k u_\epsilon-(T_ku)_\delta) v h_l(u_\epsilon) \ dx=I_1+I_2+I_3.
\end{align*}
We deal with these three terms as in Step 5.1, leading to  the same assertion \eqref{limsup}. The calculations are analogous to those from Step 5.1, only to pass to the limit in the term $I_3$ we use Lemma \ref{lem:modular} and the fact that $\nabla (T_ku)_\delta  \stackrel{M}{\longrightarrow}\nabla T_ku$  in place of the weak-* convergence in $L_M(\Omega)$. 

\bigskip

\noindent \textit{Step 6. Commutator estimate.} In this step we derive a simple commutator estimate which will be used several times in the following steps. We will estimate the following difference of two expressions 
\begin{align*}
& \int_{B(0,\epsilon)}a(x,\nabla T_k u_\epsilon - s)\cdot (\nabla T_k u_\epsilon - s) \phi^\epsilon(s)\, ds - a^\epsilon(x,\nabla T_ku_\epsilon) \cdot \nabla T_k u_\epsilon\\
& \ \ = \int_{B(0,\epsilon)}a(x,\nabla T_k u_\epsilon - s)\cdot (\nabla T_k u_\epsilon - s) \phi^\epsilon(s)\, ds - \left(\int_{B(0,\epsilon)}a(x,\nabla T_k u_\epsilon - s) \phi^\epsilon(s)\, ds\right) \cdot \nabla T_k u_\epsilon.
\end{align*}
This difference is equal to 
$$
- \int_{B(0,\epsilon)}a(x,\nabla T_k u_\epsilon - s)\cdot s \phi^\epsilon(s)\, ds 
$$
We will in fact prove that there holds
\begin{equation}\label{commutator}
\lim_{\epsilon\to 0} \int_\Omega\int_{B(0,\epsilon)}|a(x,\nabla T_k u_\epsilon - s)\cdot s| \phi^\epsilon(s) \, ds\, dx = 0.
\end{equation}
Denote the double integral in the expression \eqref{commutator} by $I$. There holds 
\begin{align*}
&I = \int_{\Omega}\int_{B(0,\epsilon)}\phi^\epsilon(s)|a(x,\nabla T_ku_\epsilon-s)\cdot s| \, ds\, dx\\
&\qquad  \leq \epsilon \int_{\Omega}\int_{B(0,\epsilon)}\phi^\epsilon(s)\left|a(x,\nabla T_ku_\epsilon-s)\cdot \frac{s}{|s|}\right|\, ds\, dx\\
& \qquad  \leq \epsilon \int_{\Omega}\int_{B(0,\epsilon)}\phi^\epsilon(s)\left(\widetilde{M}\left(x,a(x,\nabla T_ku_\epsilon-s)\right) + M\left(x,\frac{s}{|s|}\right)\right)\, ds\, dx\\
&\qquad \leq \epsilon \int_{\Omega}\int_{B(0,\epsilon)}\phi^\epsilon(s)\left(\widetilde{M}\left(x,a(x,\nabla T_ku_\epsilon-s)\right) + m_2(1)\right)\, ds\, dx\\
&\qquad \leq \epsilon\left( C +  \frac{1}{c_A}\int_{\Omega}\int_{B(0,\epsilon)}\phi^\epsilon(s)\left(a(x,\nabla T_ku_\epsilon-s)\cdot(\nabla T_ku_\epsilon-s) + m(x)\right)\, ds\, dx\right)\\
&\qquad \leq  \epsilon\left( C +  \frac{1}{c_A}\int_{\Omega}\widetilde{M}(x,a^\epsilon(x,\nabla T_ku_\epsilon))+M(x,\nabla T_ku_\epsilon) \, dx\right)+  \frac{\epsilon}{c_A}I,
\end{align*}
where $C$ is a generic constant dependent only on $\eta, c_A, \|m\|_{L^1(\Omega)}$, and $m_2(1)$. This means that
\begin{align*}
& I  \leq C\frac{c_A}{c_A-\epsilon}\epsilon \left( 1 +  \int_{\Omega}\widetilde{M}(x,a^\epsilon(x,\nabla T_ku_\epsilon))+M(x,\nabla T_ku_\epsilon) \, dx\right),
\end{align*}
whence, by \eqref{conv:1} and \eqref{conv:2} we deduce that
\begin{equation}
\lim_{\epsilon\to 0}I = 0,
\end{equation}
and the assertion is proved.

\bigskip

 \noindent \textit{Step 7. Weak convergence in $L^1$ of $a^\epsilon(x,\nabla T_ku_\epsilon)\cdot \nabla T_ku_\epsilon$.} Remembering, that $a^\epsilon$ defined by \eqref{aeps} has the form
$$a^\epsilon(x,\xi) =  \int_{B(0,\epsilon)}\phi^\epsilon(\lambda) a(x,\xi-\lambda)\, d\lambda,
 $$
 where $\phi^\epsilon$ is a mollifier kernel, there holds 
 $$
 a^\epsilon(x,\nabla T_ku_\epsilon)\cdot \nabla T_ku_\epsilon = \int_{B(0,\epsilon)}\phi^\epsilon(\lambda) a(x,\nabla T_ku_\epsilon-\lambda)\cdot \nabla T_ku_\epsilon\, d\lambda.
 $$
We need to show that this sequence converges weakly in $L^1(\Omega)$, i.e. that
\begin{equation}\label{weak_l1}
a^\epsilon(x,\nabla T_ku_\epsilon) \cdot \nabla T_ku_\epsilon\to \alpha_k\cdot \nabla T_ku\quad \textrm{weakly in}\quad L^1(\Omega)\quad \textrm{as}\quad \epsilon\to 0.
\end{equation}
There exists a sequence of sets $E_1 \subset E_2 \subset \ldots \subset E_l \subset \ldots \subset \Omega$ with $\lim_{l\to \infty} |\Omega\setminus E_l| = 0 $ such that $|\nabla T_ku| + |a(x,\nabla T_ku)| + |m(x)|\leq c(l)$ on the set $E_l$ (where $m$ is the function present in assumption (A3)) and at the same time  functions $a^\epsilon(x,\nabla T_ku_\epsilon)\cdot \nabla T_ku_\epsilon$ are equiintegrable on $E_l$ (such sets exist by Lemma \ref{chacon} and $L^1$ boundedness of $a^\epsilon(x,\nabla T_ku_\epsilon)\cdot \nabla T_ku_\epsilon$ which is a consequence of \eqref{eq:test}). Note that the sets $E_l$ may depend on $k$ but the argument in this step is conducted for any fixed number $k$. 

We rewrite the following inequality which is a consequence of the monotonicity of $A$
$$
0\leq \int_{B(0,\epsilon)}\phi^\epsilon(s)(a(x,\nabla T_ku_\epsilon-s)-a(x,\nabla T_ku))\cdot (\nabla T_ku_\epsilon-s-\nabla T_ku)\, ds \ \   \textrm{for almost every}\ \ x\in \Omega.
$$
as 
$$
0\leq \int_{\R^n}\phi^\epsilon(\nabla T_ku_\epsilon-\lambda)(a(x,\lambda)-a(x,\nabla T_ku))\cdot (\lambda-\nabla T_ku)\, d\lambda \quad \textrm{for almost every}\quad x\in \Omega.
$$
Now denote by $\mu_x^{k,\epsilon}$ the probability measure absolutely continuous with respect to the Lebesgue measure with the density given by $\phi^\epsilon(\nabla T_ku_\epsilon-\lambda)$, Using this notation we can rewrite the above formula as 
$$
0\leq \int_{\R^n}(a(x,\lambda)-a(x,\nabla T_ku))\cdot (\lambda-\nabla T_ku)\, d\mu_x^{k,\epsilon}(\lambda) \quad \textrm{for almost every}\quad x\in \Omega.
$$
Define the map $g_x(\lambda) = \lambda + a(x,\lambda)$. This is a bijection from $\R^d$ to $\mathrm{im}\, g_x$. Now we define the Borel measures $\nu_{x}^{k,\epsilon}$ as the push-forward of measures $\mu_x^{k,\epsilon}$ through the functions $g_x$, i.e. by the formula $\nu_{x}^{k,\epsilon}(S) = \mu_x^{k,\epsilon}(g_x^{-1}(S\cap \textrm{im}\, (g_x)))$ for Borel sets $S\subset \R^d$ (cf. \cite[Section 5.2]{Ambrosio_Gigli_Savare}). Then
\begin{equation}\label{eq:positiv}
0\leq \int_{\R^d}\left(\frac{\lambda+\varphi_{A(x)}(\lambda)}{2}-a(x,\nabla T_ku)\right)\cdot \left(\frac{\lambda-\varphi_{A(x)}(\lambda)}{2}-\nabla T_ku\right)\, d\nu_x^{k,\epsilon}(\lambda) \quad \textrm{for a. e.}\quad x\in \Omega.
\end{equation}
Consider the maps $\Omega \ni x\to \nu_x^{k,\epsilon} \in \mathcal{M}(\R^d)$. The definition of $\nu_x^{k,\epsilon}$ implies, by the Fubini theorem, that these maps are weak-* measurable (see \cite[Section 2]{Ball_fundamental}). Moreover, all measures $\nu_x^{k,\epsilon}$ are probability measures and hence it follows that $\|\nu_x^{k,\epsilon}\|_{L^\infty_w(\Omega;\mathcal{M}(\R^d))} = 1$.
We will use Theorem \ref{fundamental}. To use this result we need to verify two claims. 

\medskip

\noindent \textit{Claim 1.} The first claim is the tightness condition 
$$
\lim_{R\to \infty} \sup_{\epsilon>0} |\{x\in \Omega\,:\ \textrm{supp}(\nu^{\epsilon,k}_x)\setminus B(0,R)\neq \emptyset\}| = 0.
$$
To verify this condition define the function $\gamma^{k,\epsilon}(x)$ by
$$
\gamma^{k,\epsilon}(x) = \max_{\lambda\in \textrm{supp}\, (\nu^{\epsilon,k}_x)}|\lambda| = \max_{\lambda\in \textrm{supp}\, (\mu^{\epsilon,k}_x)}|\lambda+a(x,\lambda)|\leq \sup_{s\in B(0,\epsilon)} |\nabla T_ku_\epsilon+s + a(x,\nabla T_ku_\epsilon+s)|.
$$
It follows that
$$
\gamma^{k,\epsilon}(x) \leq 3\max\left\{\epsilon , |\nabla T_ku_\epsilon|,\sup_{s\in B(0,\epsilon)} |a(x,\nabla T_ku_\epsilon+s)|\right\}.
$$
Now define 
$$g(s) = \min\left\{\frac{\widetilde{m_2}\left(\frac{s}{3}\right)}{\max\left\{ \frac{s}{3} ,1\right\}}, m_1\left(\frac{s}{3} \right)\right\}.
$$
Note that $g$ is a continuous function and $\lim_{s\to \infty}g(s) = \infty$, as both $m_1$ and $\widetilde{m_2}$ are $N$-functions. Then
$$
g(\gamma^{k,\epsilon}(x)) \leq m_1(\epsilon) + m_1(|\nabla T_ku_\epsilon|) + \frac{\widetilde{m_2}\left(\sup_{s\in B(0,\epsilon)} |a(x,\nabla T_ku_\epsilon+s)|\right)}{\max\left\{ \sup_{s\in B(0,\epsilon)} |a(x,\nabla T_ku_\epsilon+s)| ,1\right\}}.
$$
It follows that
\begin{align*}
& g(\gamma^{k,\epsilon}(x)) \leq m_1(\epsilon) + m_1(|\nabla T_ku_\epsilon|) + \frac{\sup_{s\in B(0,\epsilon)} \widetilde{m_2}\left( |a(x,\nabla T_ku_\epsilon+s)|\right)}{\max\left\{ \sup_{s\in B(0,\epsilon)} |a(x,\nabla T_ku_\epsilon+s)| ,1\right\}}\\
&\qquad  \leq  m_1(\epsilon) + m_1(|\nabla T_ku_\epsilon|) + \frac{\sup_{s\in B(0,\epsilon)} \widetilde{M}\left(x, a(x,\nabla T_ku_\epsilon+s)\right)}{\max\left\{ \sup_{s\in B(0,\epsilon)} |a(x,\nabla T_ku_\epsilon+s)| ,1\right\}}\\
& \qquad \leq  m_1(\epsilon) + m_1(|\nabla T_ku_\epsilon|) + \frac{\sup_{s\in B(0,\epsilon)} |a(x,\nabla T_k(u_\epsilon)+s)||\nabla T_k(u_\epsilon)+s| + m(x)}{\max\left\{ \sup_{s\in B(0,\epsilon)} |a(x,\nabla T_ku_\epsilon+s)| ,1\right\}}\\
& \qquad \leq  m_1(\epsilon) + m_1(|\nabla T_ku_\epsilon|) + m(x) + \frac{\sup_{s\in B(0,\epsilon)} |a(x,\nabla T_ku_\epsilon+s)||\nabla T_ku_\epsilon+s|}{\max\left\{ \sup_{s\in B(0,\epsilon)} |a(x,\nabla T_ku_\epsilon+s)| ,1\right\}}\\
& \qquad \leq  m_1(1) + M(x,\nabla T_ku_\epsilon) + m(x) + 1 + |\nabla T_ku_\epsilon|.
\end{align*}
We deduce that there exists a constant $C_1>0$ such that
$$
g(\gamma^{k,\epsilon}(x)) \leq C_1(1+M(x,\nabla T_ku_\epsilon)) + m(x),
$$
whence
$$
\int_{\Omega} g(\gamma^{k,\epsilon}(x))\, dx \leq C_2(k),
$$
where the constant $C_2(k)$ does not depend of $\epsilon$. 
Now, 
\begin{align*}
&\sup_\epsilon |\{x\in \Omega\,:\ \textrm{supp}(\nu^{\epsilon,k}_x)\setminus B(0,R)\neq \emptyset\}| =  \sup_\epsilon |\{x\in \Omega\,:\gamma^{k,\epsilon}(x)\geq R\}|\\
&\qquad \leq \frac{1}{\min_{s\geq R} \{g(s)\}} \int_{E_n}g(\gamma^{k,\epsilon}(x))\, dx \leq \frac{ C_2(k)}{\min_{s\geq R} \{g(s)\}},
\end{align*}
whence the tightness condition follows. 

\medskip

\noindent \textit{Claim 2.} The second claim we need to verify is that for functions  
\begin{align*}
& \qquad F_1(x,r) = \frac{r+\varphi_{A(x)}(r)}{2}\cdot \frac{r-\varphi_{A(x)}(r)}{2}, &&F_2(x,r) = a(x,\nabla T_ku)\cdot \frac{r-\varphi_{A(x)}(r)}{2},\\
& \qquad  F_3(x,r) = \frac{r+\varphi_{A(x)}(r)}{2}\cdot \nabla T_k(u), &&F_4(x,r) = a(x,\nabla T_ku)\cdot \nabla T_k(u).
\end{align*}
there holds 
\begin{equation} \label{eq:equiint}
\lim_{R\to 0} \sup_{\epsilon > 0} \int_{E_l} \int_{\{r\in \R^d : |F_i(x,r)| > R\}} \left|F_i(x,r)\right| \, d\nu^{k,\epsilon}_x(r)\,dx = 0\quad i=1,2,3,4.
\end{equation}
We begin with the proof for $F_1$. There holds 
$$
\int_{\{r\in \R^d : |F_1(x,r)| > R\}} \left|\frac{r+\varphi_{A(x)}(r)}{2}\cdot \frac{r-\varphi_{A(x)}(r)}{2}\right| \, d\nu^{k,\epsilon}_x(r)
= \int_{\{r\in \R^d : |a(x,r)\cdot r| > R \}} |a(x,r)\cdot r| \, d\mu^{k,\epsilon}_x(r).
$$
Now assume that $(x,r)\in E_l \times \mathrm{supp}\, (\mu_x^{k,\epsilon})$ is such that 
$|a(x,r)\cdot r| > R$. There holds
$$
R < |a(x,r)\cdot r|\leq  M(x,r) + \widetilde{M}(x,a(x,r)) \leq  m_2(|r|) + \widetilde{m_1}(|a(x,r)|).
$$
Then  
$$
\textrm{either}\ \  \frac{R}{2} < m_2(|r|) \quad \textrm{or} \quad \frac{R}{2} < \widetilde{m_1}(|a(x,r)|) \leq \widetilde{m_1}\left(\frac{2C_1}{c_A} m_2\left(\frac{2|r|}{C_A}\right)+\frac{2C_1}{c_A}c(l) + C_1\right),
$$
where we used \eqref{eq:growth_bound} in the last bound. We deduce that there exists $R_1(R)$ with $\lim_{R\to \infty} R_1(R) = \infty$ such that $|r|>R_1(R)$. It follows that if only $|a(x,r)\cdot r| > R$ then $\textrm{supp}\ (\mu_x^{k,\epsilon} )\subset \mathbb{R}^d \setminus B(0,R_1 - 1)$. Let us estimate 
$$
\int_{\R^d}|a(x,r)\cdot r|d\mu^{k,\epsilon}_x(r) \geq \int_{\R^d}\widetilde{M}(x,a(x,r))+M(x,r) - m(x)d\mu^{k,\epsilon}_x(r) \geq m_1(R_1(R) - 1) - c(l).
$$
This means that there exists $R_2(R)$ with $\lim_{R\to \infty} R_2(R) = \infty$ such that 
$$
\int_{\R^d}|a(x,r)\cdot r|d\mu^{k,\epsilon}_x(r) \geq \int_{\R^d}\widetilde{M}(x,a(x,r))+M(x,r) - m(x)d\mu^{k,\epsilon}_x(r) \geq R_2(R).
$$
It follows that 
\begin{align*}
&\int_{E_l}\int_{\{r\in \R^d : |a(x,r)\cdot r| > R \}} |a(x,r)\cdot r| \, d\mu^{k,\epsilon}_x(r)\, dx \\
& \qquad \leq \int_{\left\{x\in E_l\,:\ \int_{\mathbb{R}^d}|a(x,r)\cdot r|\, d\mu_x^{k,\epsilon}(r) \geq R_2(R)\right\}} \int_{\mathbb{R}^d}|a(x,r)\cdot r|\, d\mu_x^{k,\epsilon}(r)\, dx. 
\end{align*}
But the functions 
$$
E_l\ni x \mapsto \int_{\mathbb{R}^d}|a(x,r)\cdot r|\, d\mu_x^{k,\epsilon}(r)
$$
are equiintegrable on $E_l$. Indeed
\begin{align*}
& \int_{\mathbb{R}^d}|a(x,r)\cdot r|\, d\mu_x^{k,\epsilon}(r)\leq  \int_{\mathbb{R}^d}a(x,r)\cdot r\, d\mu_x^{k,\epsilon}(r) + 2m(x) \\
& \qquad \leq 
\int_{\mathbb{R}^d}\phi^\epsilon(\nabla T_ku_\epsilon-r)a(x,r)\cdot r\, dr + 2m(x)\\
& \qquad \leq a^\epsilon(x,\nabla T_ku_\epsilon)\cdot \nabla T_ku_\epsilon- \int_{\mathbb{R}^d}\phi^\epsilon(s)a(x,\nabla T_ku_\epsilon-s)\cdot s\, ds + 2c(l),
\end{align*}
 and the assertion follows from the commutator estimate \eqref{commutator} and equiintegrability on $E_l$ of $ a^\epsilon(x,\nabla T_ku_\epsilon)\cdot \nabla T_ku_\epsilon$. The assertion \eqref{eq:equiint} for $i=1$ is proved.    To get the assertion for $i=2$ observe that 
 $$
 \int_{\{r\in \R^d : |F_2(x,r)| > R\}} \left|F_2(x,r)\right| \, d\nu^{k,\epsilon}_x(r) = \int_{\{r\in \R^d : |a(x,\nabla T_k(u))\cdot r| > R \}} |a(x,\nabla T_ku)\cdot r| \, d\mu^{k,\epsilon}_x(r).
 $$
 If $(x,r)\in E_l\times \mathrm{supp}\, (\mu^{k,\epsilon}_x)$ is such that $|a(x,\nabla T_ku)\cdot r| > R$, then, clearly 
 $$
 \int_{\mathbb{R}^d} |a(x,\nabla T_ku)\cdot s|\, d\mu^{k,\epsilon}_x(s) \leq   \int_{\mathbb{R}^d} |a(x,\nabla T_ku)\cdot r| - |a(x,\nabla T_ku)\cdot (s-r)|\, d\mu^{k,\epsilon}_x(s)\geq R - 2 \epsilon c(l).  
 $$
 Hence 
\begin{align*}
 & \int_{E_l}\int_{\{r\in \R^d : |a(x,\nabla T_ku)\cdot r| > R \}} |a(x,\nabla T_ku)\cdot r| \, d\mu^{k,\epsilon}_x(r)\, dx \\
 & \leq \int_{\left\{x\in E_l\,:\ \int_{\mathbb{R}^d}|a(x,\nabla T_ku)\cdot r|\, d\mu_x^{k,\epsilon}(r) \geq R - 2 \epsilon c(l)\right\}} \int_{\mathbb{R}^d}|a(x,\nabla T_ku)\cdot r|\, d\mu_x^{k,\epsilon}(r)\, dx. 
 \end{align*}
 We prove the equiintegrability of functions 
 $$
 E_l \ni x\mapsto \int_{\mathbb{R}^d}|a(x,\nabla T_ku)\cdot r|\, d\mu_x^{k,\epsilon}(r).
 $$
Indeed 
 \begin{align*}
 &\int_{\mathbb{R}^d}|a(x,\nabla T_ku)\cdot r|\, d\mu_x^{k,\epsilon}(r) \leq  c(l) \int_{\mathbb{R}^d}\phi^\epsilon(\nabla T_ku_\epsilon-r)|r|\, dr \\
 &\qquad = c(l)\int_{\mathbb{R}^d}\phi^\epsilon(s)|\nabla T_ku_\epsilon-s|\, ds \leq c(l)(|\nabla T_ku_\epsilon| + 1),
 \end{align*}
 and the assertion follows from equiintegrability of $\nabla T_ku_\epsilon$, cf. \eqref{conv:1}.
 
We proceed with the assertion for $i=3$. Observe that  
$$
\int_{\{r\in \R^d : |F_3(x,r)| > R\}} \left|F_3(x,r)\right| \, d\nu^{k,\epsilon}_x(r) = \int_{\{r\in \R^d : |a(x,r)\cdot \nabla T_ku| > R \}} |a(x,r)\cdot \nabla T_ku| \, d\mu^{k,\epsilon}_x(r).
$$
If $(x,r)\in E_l\times \mathrm{supp}\, (\mu^{k,\epsilon}_x)$ is such that $|a(x,r)\cdot \nabla T_ku| > R$, then, clearly 
$$
R < |a(x,r)\cdot \nabla T_ku| \leq c(l) |a(x,r)|.
$$
 We estimate
\begin{align*}
& \int_{\{r\in \R^d : |F_3(x,r)| > R\}} |a(x,r)\cdot \nabla T_ku|\, d\mu^{k,\epsilon}_x(r)\\
& \qquad  \leq c(l)\int_{\{r\in \R^d : |F_3(x,r)| > R\}}\frac{|a(x,r)|}{\widetilde{m_2}(|a(x,r)|)}  \widetilde{M}(x,a(x,r))\, d\mu^{k,\epsilon}_x(r)\\
& \qquad \leq \sup_{s > R/c(l)} \frac{s}{\widetilde{m_2}(s)} \int_{\R^d}  \widetilde{M}(x,a(x,r))\, d\mu^{k,\epsilon}_x(r) = C(R)\int_{\R^d}  \widetilde{M}(x,a(x,r))\, d\mu^{k,\epsilon}_x(r),
\end{align*}
where $C(R) \to 0$ as $R\to \infty$. It suffices to show that the last integral is bounded uniformly with respect to $\epsilon$ in $L^1(E_l)$. We estimate the integrand from above 
\begin{align*}
&
\int_{\R^d}  \widetilde{M}(x,a(x,r))\, d\mu^{k,\epsilon}_x(r) \leq m(x) + \int_{\R^d}  a(x,r)\cdot r\, d\mu^{k,\epsilon}_x(r)\\
& \qquad \leq c(l) + \int_{\R^d}  a(x,r)\cdot r\, \phi^\epsilon (\nabla T_ku_\epsilon-r)\, dr \\
& \qquad = c(l) + \int_{\R^d}  a(x,\nabla T_ku_\epsilon-s)\cdot (\nabla T_ku_\epsilon-s)\, \phi^\epsilon(s)\, ds\\
&  \qquad = c(l) + a^\epsilon(x,\nabla T_ku_\epsilon)\cdot \nabla T_ku_\epsilon - \int_{\R^d}  a(x,\nabla T_ku_\epsilon-s)\cdot s\, \phi^\epsilon(s)\, ds 
\end{align*}
The required assertion follows from \eqref{commutator} and the estimates \eqref{conv:1}--\eqref{conv:2}. Finally, the assertion \eqref{eq:equiint} for $F_4$ holds trivially, as $|\nabla T_ku| + |a(x,\nabla T_ku)| \leq c(l)$ on $E_l$. 

We are in position to use Theorem \ref{fundamental} to assert that there exists a Young measure $\nu^k:\Omega \to \mathcal{M}(\R^d)$ 
with $\|\nu^k_x\|_{\mathcal{M}(\R^d)} = 1$ for a.e. $x\in \Omega$. 
such that 
\begin{align*}
& \int_{\R^d}\left(\frac{\lambda+\varphi_{A(x)}(\lambda)}{2}-a(x,\nabla T_ku)\right)\cdot \left(\frac{\lambda-\varphi_{A(x)}(\lambda)}{2}-\nabla T_ku\right)\, d\nu_x^{k,\epsilon}(\lambda)\\
& \ \ \stackrel{b}{\to} \int_{\R^d}\left(\frac{\lambda+\varphi_{A(x)}(\lambda)}{2}-a(x,\nabla T_ku)\right)\cdot \left(\frac{\lambda-\varphi_{A(x)}(\lambda)}{2}-\nabla T_ku\right)\, d\nu_x^{k}(\lambda)\ \mathrm{as}\ \epsilon\to 0.
\end{align*}
Now \eqref{eq:positiv} implies that 
$$
0\leq \int_{\R^d}\left(\frac{\lambda+\varphi_{A(x)}(\lambda)}{2}-a(x,\nabla T_ku)\right)\cdot \left(\frac{\lambda-\varphi_{A(x)}(\lambda)}{2}-\nabla T_ku\right)\, d\nu_x^{k}(\lambda) \ \ \textrm{for almost every}\ \ x\in \Omega. 
$$
On the other hand, 
\begin{align*}
& \int_{\R^d}a(x,\nabla T_ku)\cdot \left(\frac{\lambda-\varphi_{A(x)}(\lambda)}{2}-\nabla T_ku\right)\, d\nu_x^{k,\epsilon}(\lambda)\\
& \ \ \stackrel{b}{\to} \int_{\R^d}a(x,\nabla T_ku)\cdot \left(\frac{\lambda-\varphi_{A(x)}(\lambda)}{2}-\nabla T_ku\right)\, d\nu_x^{k}(\lambda)\ \mathrm{as}\ \epsilon\to 0.
\end{align*}
But 
\begin{align*}
& \int_{\R^d}a(x,\nabla T_ku)\cdot \left(\frac{\lambda-\varphi_{A(x)}(\lambda)}{2}-\nabla T_ku\right)\, d\nu_x^{k,\epsilon}(\lambda) = \int_{\R^d}a(x,\nabla T_ku)\cdot \left(\lambda-\nabla T_ku\right)\, d\mu_x^{k,\epsilon}(\lambda)\\
& \ \ = \int_{\R^d}a(x,\nabla T_ku)\cdot \lambda\, d\mu_x^{k,\epsilon}(\lambda) - a(x,\nabla T_ku)\cdot \nabla T_ku\\
& \ \  = \int_{\R^d}a(x,\nabla T_ku)\cdot \nabla T_ku^\epsilon\, \phi^\epsilon(s)ds - \int_{\R^d}a(x,\nabla T_ku)\cdot s\, \phi^\epsilon(s)ds - a(x,\nabla T_ku)\cdot \nabla T_ku.
\end{align*}
As on sets $E_l$ there holds $a(x,\nabla T_ku) \in L^\infty(E_l)$ by \eqref{eq:lmconv} and by the estimate 
$$
\left| \int_{\R^d}a(x,\nabla T_ku)\cdot s\, \phi^\epsilon(s)ds \right| \leq \epsilon c(l),
$$
we deduce that 
\begin{align*}
&
\int_{\R^d}a(x,\nabla T_ku)\cdot \left(\frac{\lambda-\varphi_{A(x)}(\lambda)}{2}-\nabla T_ku\right)\, d\nu_x^{k,\epsilon}(\lambda)\\
& \qquad  = \int_{\R^d}a(x,\nabla T_ku)\cdot \left(\lambda-\nabla T_ku\right)\, d\mu_x^{k,\epsilon}(\lambda) \stackrel{b}{\to} 0,
\end{align*}
and hence 
$$
\int_{\R^d}a(x,\nabla T_ku)\cdot \left(\frac{\lambda-\varphi_{A(x)}(\lambda)}{2}-\nabla T_ku\right)\, d\nu_x^{k}(\lambda) = 0\quad \textrm{for almost every}\ x\in \Omega.
$$
It follows that 
$$
0\leq \int_{\R^d}\frac{\lambda+\varphi_{A(x)}(\lambda)}{2}\cdot \left(\frac{\lambda-\varphi_{A(x)}(\lambda)}{2}-\nabla T_ku\right)\, d\nu_x^{k}(\lambda) \ \ \textrm{for almost every}\ \ x\in \Omega. 
$$
Now 
$$ \int_{\R^d}\frac{\lambda+\varphi_{A(x)}(\lambda)}{2}\cdot \nabla T_ku\, d\nu_x^{k,\epsilon}(\lambda) \stackrel{b}{\to} \int_{\R^d}\frac{\lambda+\varphi_{A(x)}(\lambda)}{2}\cdot \nabla T_ku\, d\nu_x^{k}(\lambda)\quad \mathrm{as}\quad \epsilon\to 0.
$$
But
$$
\int_{\R^d}\frac{\lambda+\varphi_{A(x)}(\lambda)}{2}\cdot \nabla T_ku\, d\nu_x^{k,\epsilon}(\lambda) = a^\epsilon(x,\nabla T_ku_\epsilon)\cdot \nabla T_ku \stackrel{b}{\to} \alpha_k \cdot \nabla T_ku.
$$
Hence
\begin{equation}\label{eq:important}
\alpha_k \cdot \nabla T_k(u)  \leq \int_{\R^d}\frac{\lambda+\varphi_{A(x)}(\lambda)}{2}\cdot \frac{\lambda-\varphi_{A(x)}(\lambda)}{2}\, d\nu_x^{k}(\lambda)  \ \ \textrm{for almost every}\ \ x\in \Omega. 
\end{equation}
Using Lemma \ref{lower} (note that $\frac{\lambda+\varphi_{A(x)}(\lambda)}{2}\cdot \frac{\lambda-\varphi_{A(x)}(\lambda)}{2} \geq -m(x)$)  we deduce that
$$  \int_{\Omega}\int_{\R^d}\frac{\lambda+\varphi_{A(x)}(\lambda)}{2}\cdot \frac{\lambda-\varphi_{A(x)}(\lambda)}{2}\, d\nu_x^{k}(\lambda)\, dx \\
 \leq \liminf_{\epsilon\to 0}\int_{\Omega}\int_{\R^d}\frac{\lambda+\varphi_{A(x)}(\lambda)}{2}\cdot \frac{\lambda-\varphi_{A(x)}(\lambda)}{2}\, d\nu_x^{k,\epsilon}(\lambda) \, dx.
$$
From the definition of measures $\nu_x^{k,\epsilon}$, we obtain 
$$
\int_{\Omega}\int_{\R^d}\frac{\lambda+\varphi_{A(x)}(\lambda)}{2}\cdot \frac{\lambda-\varphi_{A(x)}(\lambda)}{2}\, d\nu_x^{k,\epsilon}(\lambda) \, dx = \int_{\Omega}\int_{\R^d}a(x,\nabla T_ku_\epsilon-s)\cdot (\nabla T_ku_\epsilon-s)\phi^\epsilon(s)\, ds \, dx.
$$
Using \eqref{commutator}, it follows that 
$$
\int_{\Omega}\int_{\R^d}\frac{\lambda+\varphi_{A(x)}(\lambda)}{2}\cdot \frac{\lambda-\varphi_{A(x)}(\lambda)}{2}\, d\nu_x^{k}(\lambda)\, dx \leq \liminf_{\epsilon\to 0}\int_{\Omega}a^\epsilon(x,\nabla T_ku_\epsilon)\cdot \nabla T_ku_\epsilon \, dx.
$$
We use \eqref{limsup} to deduce that
$$
\int_{\Omega}\int_{\R^d}\frac{\lambda+\varphi_{A(x)}(\lambda)}{2}\cdot \frac{\lambda-\varphi_{A(x)}(\lambda)}{2}\, d\nu_x^{k}(\lambda)\, dx \leq \int_{\Omega}\alpha^k \cdot \nabla T_ku \, dx.
$$
This inequality together with \eqref{eq:important} imply that 
$$
\alpha_k \cdot \nabla T_k(u)  = \int_{\R^d}\frac{\lambda+\varphi_{A(x)}(\lambda)}{2}\cdot \frac{\lambda-\varphi_{A(x)}(\lambda)}{2}\, d\nu_x^{k}(\lambda)  \ \ \textrm{for almost every}\ \ x\in \Omega. 
$$
But as 
$$ \int_{\R^d}\frac{\lambda+\varphi_{A(x)}(\lambda)}{2}\cdot \frac{\lambda-\varphi_{A(x)}(\lambda)}{2}\, d\nu_x^{k,\epsilon}(\lambda)\\
\stackrel{b}{\to} \int_{\R^d}\frac{\lambda+\varphi_{A(x)}(\lambda)}{2}\cdot \frac{\lambda-\varphi_{A(x)}(\lambda)}{2}\, d\nu_x^{k}(\lambda)\ \mathrm{as}\ \epsilon\to 0,
$$
we deduce that
$$ \int_{\R^d}\frac{\lambda+\varphi_{A(x)}(\lambda)}{2}\cdot \frac{\lambda-\varphi_{A(x)}(\lambda)}{2}\, d\nu_x^{k,\epsilon}(\lambda)\\
\stackrel{b}{\to} \alpha_k \cdot \nabla T_ku \ \mathrm{as}\ \epsilon\to 0,
$$
In other words 
$$ \int_{\R^d}a(x,\nabla T_k(u_\epsilon)-s)\cdot (\nabla T_k(u_\epsilon)-s)\phi^\epsilon(s)\, ds\\
\stackrel{b}{\to} \alpha_k \cdot \nabla T_ku \ \mathrm{as}\ \epsilon\to 0,
$$
Now, \eqref{commutator} implies that 
$$
a^\epsilon(x,\nabla T_ku_\epsilon)\cdot \nabla T_ku_\epsilon\\
\stackrel{b}{\to} \alpha_k \cdot \nabla T_ku \ \mathrm{as}\ \epsilon\to 0.
$$
By \eqref{limsup} and \eqref{eq:epsilon_est} we are in position to use Lemma \ref{biting_weak} which implies the required assertion \eqref{weak_l1}.
 
\bigskip 

\noindent \textit{Step 8. Minty trick.} The aim of this step is to prove that $\alpha_k(x)\in A(x,(\nabla T_ku)(x))$ for a.e. $x\in \Omega$, i.e., $\alpha_k$ is a selection of $A(\cdot,\nabla T_ku)$.  
To this end consider $\xi\in L^\infty(\Omega)^d$. Then $a(x,\xi(x)) \in L_{\widetilde{M}}(\Omega)$ is a  selection of $A(x,\xi(x))$. The monotonicity of $A$ implies that 
$$
\int_{B(0,\epsilon)}\phi^\epsilon(s)(a(x,\nabla T_ku_\epsilon-s)-a(x,\xi(x)))\cdot (\nabla T_ku_\epsilon-s-\xi(x))\, ds \geq 0\quad \textrm{for almost every}\quad x\in \Omega.
$$
It follows that
\begin{align*}
&(a^\epsilon(x,\nabla T_ku_\epsilon)-a(x,\xi(x)))\cdot (\nabla T_ku_\epsilon-\xi(x)) \\
&\quad -  \int_{B(0,\epsilon)}\phi^\epsilon(s)a(x,\nabla T_ku_\epsilon-s)\cdot s\, ds + \int_{B(0,\epsilon)}\phi^\epsilon(s)a(x,\xi(x))\cdot s\, ds  \geq 0\ \  \textrm{for almost every}\ \  x\in \Omega.
\end{align*} 
There exists a sequence of sets $E_1\subset E_2 \subset \ldots \subset E_l \subset \ldots \subset \Omega$ with $\lim_{l\to \infty}|\Omega\setminus E_l| = 0$ such that on every $E_l$ there holds $a(x,\xi(x))\in L^\infty (E_l)$. We multiply the above inequality by a nonnegative function $\eta\in C^\infty(\overline{\Omega})$, and integrate it with respect to $x$ over $E_l$, whence
\begin{align}
&\int_{E_l}(a^\epsilon(x,\nabla T_ku_\epsilon)-a(x,\xi(x)))\cdot (\nabla T_ku_\epsilon-\xi(x))\eta(x)\, dx \label{eq:limitaux_1} \\
& \qquad \geq  \int_{E_l}\int_{B(0,\epsilon)}\phi^\epsilon(s)a(x,\nabla T_ku_\epsilon-s)\cdot s\, \eta(x)ds\, dx  - \int_{E_l}\int_{B(0,\epsilon)}\phi^\epsilon(s)a(x,\xi(x))\cdot s\eta(x)\, ds\, dx\\
& \qquad  = I_1- I_2.\nonumber
\end{align} 
Let us pass with $\epsilon\to 0$. Estimate \eqref{commutator} implies that $\lim_{\epsilon\to 0} I_1 = 0$. We need to deal with $I_2$, namely 
$$
|I_2| = \left|\int_{\Omega}\int_{B(0,\epsilon)}\phi^\epsilon(s)a(x,\xi(x))\cdot s \eta(x)\, ds\, dx\right| \leq \left|\int_{\Omega}\int_{B(0,\epsilon)}\phi^\epsilon(s)|a(x,\xi(x))||\eta(x)| \epsilon\, ds\, dx\right| \leq C \epsilon,
$$
whence 
\begin{equation}\label{eq:limitaux_2}
\lim_{\epsilon\to 0} I_2 = 0.
\end{equation}
Using the fact that $\lim_{\epsilon \to 0} I_1-I_2 = 0$, as well as \eqref{eq:lmconv}, \eqref{aconv}, and \eqref{weak_l1} in \eqref{eq:limitaux_1} we deduce that
$$
\int_{E}(\alpha_k(x)-a(x,\xi(x)))\cdot ((\nabla T_ku)(x) -\xi(x)) \eta(x)\, dx \geq 0.$$
Since the above assertion is valid for any nonnegative $\eta \in C^\infty(\overline{\Omega})$, it follows that 
$$
(\alpha_k(x)-a(x,\xi(x)))\cdot ((\nabla T_ku)(x) -\xi(x)) \geq 0
$$ 
for almost every $x\in E_l$, and hence for almost every $x\in \Omega$, and every $\xi \in L^\infty(\Omega)^d$. We take $\xi(x) = z$, a constant vector of rational numbers in $\R^d$. The set of $x\in \Omega$ such that the above inequality holds for every rational $s$ has a full measure. It follows that for almost every $x$ the set $$\{ (z, a(x,z))\,:\ z\in \textrm{a dense set in}\, \R^d \} \cup \{ ((\nabla T_ku)(x), \alpha_k(x))  \}$$ is a monotone graph, and it can be extended to a maximally monotone graph $\widetilde{A}(x)$. Now, as $A$ is a maximally monotone graph \cite[Corollary 1.5]{Alberti} (also see \cite[Lemma 2.2 and Corollary 2.3]{Bulicek}) implies that $\widetilde{A} = A$, and hence $\alpha_k(x) \in A(x,(\nabla T_ku)(x))$ for almost every $x\in \Omega$. The assertion is proved. 

\bigskip

\noindent \textit{Step 9. The solution satisfies \eqref{renorm}.} We prove that $u$ satisfies the equation in the renormalized sense \eqref{renorm}. To this end let $h\in C^1_c(\R)$ and let $w\in W^{1,\infty}_0(\Omega)$. We test \eqref{eq:weak} with $h(u_\epsilon) w$. By a similar argument as in Step 2 such choice of test function is allowed. This leads us to the equation
$$
\int_{\Omega}a^\epsilon(x,\nabla u_\epsilon)\cdot \nabla (h(u_\epsilon)  w)\, dx = \int_{\Omega} T_{1/\epsilon} f h(u_\epsilon)  w\, dx.
$$
Using the Lebesgue dominated convergence theorem, we deduce that
$$
\lim_{\epsilon\to 0}\int_{\Omega} T_{1/\epsilon}f h(u_\epsilon) w\, dx = \int_{\Omega} f  h(u) w\, dx.
$$
To pass to the limit on the left-hand side note that, for every $K$ such that $\mathrm{supp}(h) \subset [-K,K]$
\begin{align*}
&\int_{\Omega}a^\epsilon(x,\nabla u_\epsilon)\cdot \nabla (h(u_\epsilon)  w)\, dx =  \int_{\Omega}a^\epsilon(x,\nabla  T_Ku_\epsilon)\cdot \nabla T_Ku_\epsilon h'(u_\epsilon)   w\, dx\\
&\quad  + \int_{\Omega}a^\epsilon(x,\nabla T_Ku_\epsilon)\cdot \nabla w h(u)\, dx + \int_{\Omega}(h(u_\epsilon)-h(u))a^\epsilon(x,\nabla T_Ku_\epsilon)\cdot \nabla    w\, dx = I_1 + I_2 + I_3.
\end{align*}
Now by \eqref{weak_l1} and the pointwise convergence and uniform boundedness in $L^\infty(\Omega)$ of $h'(u_\epsilon) w$ by Lemma \ref{convergence_lemma} it follows that
$$
\lim_{\epsilon\to 0} I_1 = \int_{\Omega}\alpha_K\cdot \nabla T_Ku h'(u)   w\, dx.
$$
To deal with $I_2$ note that the convergence \eqref{aconv} and the fact $\nabla w h(u) \in L^\infty(\Omega)^d$ implies that 
$$
\lim_{\epsilon\to 0} I_2 = \int_{\Omega}\alpha_K\cdot \nabla w h(u)\, dx.
$$
Finally, equiintegrability of $\{ a^\epsilon(x,\nabla T_Ku_\epsilon) \}_{\epsilon>0}$ 
and uniform boundedness in $L^\infty(\Omega)^d$ and pointwise convergence to zero of 
$\nabla w (h(u_\epsilon)-h(u))$ imply that 
$$
\lim_{\epsilon\to 0} I_3 = 0.
$$
Concluding, we obtain
$$
\int_{\Omega}\alpha_K\cdot \nabla (h(T_Ku)  w)\, dx = \int_{\Omega} f h(u)  w\, dx.
$$
But, as $\mathrm{supp}(h) \subset [-K,K]$, \eqref{renorm} follows. 
   
   \bigskip 
   
   \noindent \textit{Step 10. Controlled radiation condition.} In the last step of the proof we show that condition \eqref{renorm_control} is satisfied. As $\nabla u_\epsilon = 0$ a.e. in the set $\{ x\in \Omega\,:\ |u_\epsilon| \in \{ l,l+1 \} \}$, the estimate \eqref{cont_rad} implies that
   $$
   \lim_{l\to \infty} \sup_{\epsilon>0} \int_{\{ l-1 < |u_\epsilon| < l+2 \}} a^\epsilon(x,\nabla u_\epsilon)\cdot \nabla u_\epsilon\, dx = 0.
   $$ 
   Now, define continuous functions $g_l:\R\to \R$ by
   $$
   g_l(r) = \begin{cases}
   1& \mathrm{if}\ l\leq |r| \leq l+1,\\
   0& \mathrm{if} \ |r|<l-1\ \mathrm{or}\ |r| > l+2,\\
   \mathrm{affine} & \mathrm{otherwise}.
   \end{cases}
   $$
      There holds
      $$
      \int_{ \{ l<|u|<l+1 \} } \alpha\cdot \nabla u + m(x)\, dx \leq  \int_{ \Omega } g_l(u)(\alpha_{l+2}\cdot \nabla T_{l+2}u + m(x))\, dx.
      $$ 
      Using \eqref{weak_l1}, the pointwise convergence \eqref{conv:pointwise}, Lemma \ref{convergence_lemma}, the estimate \eqref{cont_rad} and the monotone convergence theorem we deduce that
     \begin{align*}
      & 0\leq \int_{ \{ l<|u|<l+1 \} } \alpha\cdot \nabla u + m(x)\, dx \leq  \lim_{\epsilon\to 0}\int_{ \Omega } g_l(u_\epsilon)(a^\epsilon(x,\nabla T_{l+2}u_\epsilon)\cdot \nabla T_{l+2}u_\epsilon + m(x))\, dx\\
      &\qquad \leq \lim_{\epsilon\to 0}\int_{ \{ l-1< |u_\epsilon| < l+2 \} } a^\epsilon(x,\nabla T_{l+2}u_\epsilon)\cdot \nabla T_{l+2}u_\epsilon + m(x)\, dx \leq \gamma(l),
      \end{align*} 
      where $\gamma(l) \to 0$ as $l\to \infty$. This means that 
      $$
      \lim_{l\to \infty} \int_{ \{ l<|u|<l+1 \} } \alpha\cdot \nabla u + m(x)\, dx = 0,
      $$
      which implies the required assertion. 

\section{Proof of Theorem \ref{thm:uniqueness}: uniqueness.}\label{sec:uniq}

Let $(u_1,\alpha_1)$ and $(u_2,\alpha_2)$ be two renormalized solutions. We will prove that $u_1=u_2$ for a.e. $x\in \Omega$. 
\begin{remark}
Note that only $u$ can be proved to be unique and $\alpha$ can stay nonunique in our framework. To this end consider  $\Omega = (0,1)$, $f\equiv 0$, and $A:(0,1)\times \R\to 2^\R$ given by  
$$
A(x,\xi) = \begin{cases}
-\xi\ \ \textrm{for}\ \ \xi<0,\\
[-1,1]\ \ \textrm{for}\ \ \xi=0,\\
\xi\ \ \textrm{for}\ \ \xi>0.
\end{cases}
$$
We are looking for weak solutions $u$ belonging to $H^1_0((0,1))$ of the problem $-\textrm{div} \, \alpha = 0$ with $\alpha$ being a measurable selection of $A(u_x)$. Clearly the solution   $u\equiv 0$ is unique, but any constant function $\alpha \equiv c$ with $c\in [-1,1]$ has both zero divergence and is a selection of $A(u_x)$.    
\end{remark}
 We take $h=h_l$ in \eqref{renorm} and test this equation written for $u_1$ with $T_k(T_{l+1}u_1-T_{l+1}u_2)$. Note that in case (C2) this function can be used as the test function in \eqref{renorm} as it belongs to $L^\infty(\Omega)\cap V^M_0$. On the other hand, if (C1) holds, then we consider only approximable solutions whence both $u_1$ and $u_2$ are obtained as the limits of problems \eqref{approx_1}--\eqref{approx_2}. The solutions of these approximative problems are in turn the limits of the Galerkin solutions in the sense \eqref{eq:galerkin_conv}. As gradients of the $k$-th truncations of all these Galerkin solutions are uniformly bounded in $L_M(\Omega)$ (and weak-* topology of $L_{M}(\Omega)$ on bounded sets is metrizable as this space has separable predual space), we can use the diagonal argument to obtain the sequence of functions belonging to $W^{1,\infty}(\Omega)$ which converge to $T_k(T_{l+1}u_1-T_{l+1}u_2)$ in the sense \eqref{eq:convergence_c1}. This justifies the possibility of taking   $T_k(T_{l+1}u_1-T_{l+1}u_2)$ as a test function in \eqref{renorm}.
 The  same $h$ and test function are taken in the equation written for $u_2$ and both equations are subtracted from each other. This yields
\begin{align*}
& \int_{\Omega}h_l'(u_1) \alpha_1\cdot  \nabla T_{l+1}u_1 T_k(T_{l+1}u_1-T_{l+1}u_2)\, dx - \int_{\Omega}h_l'(u_2) \cdot  \alpha_2 \nabla T_{l+1}u_2 T_k(T_{l+1}u_1-T_{l+1}u_2) \, dx\\
& + \int_{\Omega} (\alpha_1 -\alpha_2)\cdot \nabla  T_k(T_{l+1}u_1-T_{l+1}u_2)   \, dx \\
& + \int_{\Omega} (1-h_l(u_2))\alpha_2\cdot \nabla  T_k(T_{l+1}u_1-T_{l+1}u_2)   \, dx - \int_{\Omega} (1-h_l(u_1))  \alpha_1 \cdot \nabla  T_k(T_{l+1}u_1-T_{l+1}u_2)   \, dx\\
&  = \int_{\Omega} f(h_l(u_1)-h_l(u_2)) T_k(T_{l+1}u_1-T_{l+1}u_2)\, dx.
\end{align*}
We rewrite this equation as 
$$
I_1 - I_2 + I_3 + I_4 - I_5 = I_6,
$$
and pass to the limit with $l\to \infty$. It is clear that the Lebesgue dominated convergence theorem implies that 
$$
\lim_{l\to \infty} I_6 = 0.
$$
We pass to the limit in $I_1$ and $I_2$. As the argument for both terms is analogous we deal only with $I_1$. Clearly,
$$
|I_1| \leq k \int_{\{ l\leq |u_1|\leq l+1 \}} \alpha_1\cdot  \nabla T_{l+1}u_1 \, dx  + 2 k \int_{\{ l\leq |u_1|\leq l+1\}} m\, dx,
$$
and by \eqref{renorm_control} as well as the fact that the measure of sets $\{l\leq |u_1| \leq l+1 \}$ tends to zero as $l\to \infty$ we obtain 
$$
\lim_{l\to \infty} |I_1| = 0.
$$
Now we pass to the limit in $I_4$ and $I_5$. As the argument for both terms is analogous we deal only with $I_4$. 
\begin{align*}
&|I_4| \leq \int_{\{ |u_2|\geq l \} \cap \{ 0 < |T_{l+1}u_1-T_{l+1}u_2| < k\}}  |\alpha_2\cdot \nabla  T_{l+1}u_1| + \alpha_2\cdot \nabla T_{l+1}u_2 + 2m    \, dx \\
& \qquad \leq  \int_{\{ l\leq |u_2|\leq l +k + 1 \}}  \alpha_2\cdot \nabla T_{l+1}u_2\, dx \\
& \qquad \qquad + 2\int_{\{ l\leq |u_2| \}}m\, dx + \int_{\{ l\leq |u_2|\leq l +k + 1 \} \cap \{l-k\leq |u_1|\leq l+1\}} |\alpha_2\cdot \nabla  T_{l+1}u_1|dx.
\end{align*}
The first integral tends to zero by \eqref{renorm_control}, the second one, by the fact that  the measure of sets $\{l\leq |u_2|\}$ tends to zero as $l\to \infty$. To deal with the last one observe that 
\begin{align*}
&
\int_{\{ l\leq |u_2|\leq l +k + 1 \} \cap \{l-k\leq |u_1|\leq l+1\}} |\alpha_2\cdot \nabla  T_{l+1}u_1|dx \\
& \leq \int_{\{ l\leq |u_2|\leq l +k + 1 \}} \widetilde{M}(x,\alpha_2)\, dx + \int_{\{l-k\leq |u_1|\leq l+1\}} M(x, \nabla  T_{l+1}u_1)\, dx\\
& \leq \frac{1}{c_A}\int_{\{ l\leq |u_2|\leq l +k + 1 \}} \alpha_2\cdot \nabla T_{l+k+1}u_2\, dx + \frac{1}{c_A}\int_{\{ l\leq |u_2|\}}m\, dx\\
& \qquad  + \frac{1}{c_A}\int_{\{l-k\leq |u_1|\leq l+1\}} \alpha_1 \cdot \nabla  T_{l+1}u_1\, dx + \frac{1}{c_A} \int_{\{ l-k\leq |u_1| \}}m\, dx, 
\end{align*}
and all terms converge to zero either by \eqref{renorm_control}, or by the fact that we integrate $m$ over the sets with measure shrinking to zero. We deal with $I_3$. Let $l_0$ be arbitrary
and let $l+1\geq l_0$. There holds 
\begin{align*}
&
I_3 = \int_{\{ 0 < |T_{l+1}u_1-T_{l+1}u_2| < k\}} (\alpha_1 -\alpha_2)\cdot \nabla  (T_{l+1}u_1-T_{l+1}u_2)   \, dx \\
& \qquad \geq \int_{\{ 0 < |T_{l+1}u_1-T_{l+1}u_2| < k\} \cap \{|u_1|\leq l_0\}\cap \{ |u_2|\leq l_0 \}} (\alpha_1 -\alpha_2)\cdot \nabla  (T_{l+1}u_1-T_{l+1}u_2)   \, dx\\
& = \int_{\{ 0 < |u_1-u_2| < k\} \cap \{|u_1|\leq l_0\}\cap \{ |u_2|\leq l_0 \}} (\alpha_1 -\alpha_2)\cdot \nabla  (T_{l+1}u_1-T_{l+1}u_2)   \, dx.
\end{align*}
As  we know that 
$\lim_{l\to \infty} I_3 = 0$ it follows that 
$$
0 = \int_{\{ 0 < |u_1-u_2| < k\} \cap \{|u_1|\leq l_0\}\cap \{ |u_2|\leq l_0 \}} (\alpha_1 -\alpha_2)\cdot \nabla  (u_1-u_2)   \, dx,
$$
which means, by the strict monotonicity of $A$ that the set $\{ 0 < |u_1-u_2| < k\} \cap \{|u_1|\leq l_0\}\cap \{ |u_2|\leq l_0 \}$ has null measure. As $k$ and $l_0$ are arbitrary, we deduce that $u_1 = u_2$ a.e. in $\Omega$.

\section{Proof of Theorem \ref{thm:renorm_weak}: boundedness of renormalized solutions.}\label{sec:renorm_weak}

 As  in case $d=1$ there holds $W^{1,1}(\Omega)\subset L^\infty(\Omega)$ we restrict in the sequel to $d\geq 2$. The next argument follows the lines of the proof of \cite[Theorem 3.1]{Cianchi1} with a modification related to account for the presence of $m$ in (A3). Choose $t\in (0,\esssup_{x\in \Omega}|u_\epsilon(x)|)$. We define the function
$$
u^{\epsilon}_{t,h}(x) = \begin{cases}
0 \ \ \textrm{if}\ \ |u_\epsilon(x)| \leq t,\\
(|u_\epsilon(x)| - t) \textrm{sign}(u_\epsilon(x))\ \ \textrm{if}\ \ t<|u_\epsilon(x)| \leq t+h,\\
h\ \textrm{sign}(u_\epsilon(x))\ \ \textrm{if}\ \ t+h < |u_\epsilon(x)|.
\end{cases}
$$
Then $\nabla u^\epsilon_{t,h} = \chi_{\{t<|u_\epsilon|< t+h\}}\nabla u_\epsilon$. Testing the approximating equation with $u^{\epsilon}_{t,h}$, we obtain
\begin{align*}
& \int_{\{t<|u_\epsilon|\leq t+h\}}a^{\epsilon}(x,\nabla u_{\epsilon})\cdot \nabla u_\epsilon \, dx\\
& \ \  = \int_{\{t<|u_\epsilon|\leq t+h\}}(|u_\epsilon(x)| - t) \textrm{sign}(u_\epsilon(x)) T^{1/\epsilon}f(x)\, dx + \int_{\{t+h<|u_\epsilon|\}}h\ \textrm{sign}(u_\epsilon(x))T^{1/\epsilon}f(x)\, dx\\
& \leq h\int_{\{t<|u_{\epsilon}|\}}|f(x)|\, dx.  
\end{align*}
We use \eqref{eq:epsilon_est} and Remark \ref{rem:homo_minorant}, whence
$$
c_A\int_{\{t<|u_\epsilon|\leq t+h\}}M_1(\nabla u_{\epsilon}(x)) \, dx - \int_{\{t<|u_\epsilon|\leq t+h\}}m(x) \, dx \leq  h\int_{\{t<|u_{\epsilon}|\}}|f(x)|\, dx.
$$
Anizotropic P\'{o}lya--Szeg\"{o}  inequality, cf. \cite[Theorem 3.5 and formula (5.4)]{Cianchi} and \cite[formula (3.7)]{Alberico_Cianchi}, implies that
\begin{equation}\label{imp}
c_A\int_{\{t<u_\epsilon^*\leq t+h\}}(M_1)_\blacklozenge(|\nabla u_{\epsilon}^*(x)|) \, dx  \leq  h\int_{\{t<|u_{\epsilon}|\}}|f(x)|\, dx +\int_{\{t<|u_\epsilon|\leq t+h\}}m(x) \, dx.
\end{equation}
Now, as $\mu_{u_\epsilon} = \mu_{u_\epsilon^*}$, the Jensen inequality implies that
$$
(M_1)_\blacklozenge\left(\frac{\int_{\{t<u_\epsilon^*\leq t+h\}}|\nabla u_{\epsilon}^*(x)| \, dx}{\mu_{u_\epsilon}(t)-\mu_{u_\epsilon}(t+h)}\right)\leq \frac{\int_{\{t<u_\epsilon^*\leq t+h\}}(M_1)_\blacklozenge(|\nabla u_{\epsilon}^*(x)|) \, dx}{\mu_{u_\epsilon}(t)-\mu_{u_\epsilon}(t+h)}.
$$
We deduce
$$
(M_1)_\blacklozenge\left(\frac{\int_{\{t<u_\epsilon^*\leq t+h\}}|\nabla u_{\epsilon}^*(x)| \, dx}{\mu_{u_\epsilon}(t)-\mu_{u_\epsilon}(t+h)}\right)\leq \frac{1}{c_A}\frac{ h\int_{\{t<|u_{\epsilon}|\}}|f(x)|\, dx +\int_{\{t<|u_\epsilon|\leq t+h\}}m(x) \, dx.}{\mu_{u_\epsilon}(t)-\mu_{u_\epsilon}(t+h)}.
$$
Now the coarea formula and the relation between the volume of $d$ dimensional ball and area of $d-1$ dimensional sphere, cf \cite[formula (5.9)]{Cianchi} imply that
$$
\int_{\{t<u_\epsilon^*\leq t+h\}}|\nabla u_{\epsilon}^*(x)| \, dx = d \omega_d^{1/d}\int_t^{t+h}\mu_{u_\epsilon}(\tau)^{\frac{d-1}{d}}\, d\tau,
$$
whence
$$
(M_1)_\blacklozenge\left(\frac{d \omega_d^{1/d}\int_t^{t+h}\mu_{u_\epsilon}(\tau)^{\frac{d-1}{d}}\, d\tau}{\mu_{u_\epsilon}(t)-\mu_{u_\epsilon}(t+h)}\right)\leq \frac{1}{c_A}\frac{ h\int_{\{t<|u_{\epsilon}|\}}|f(x)|\, dx +\int_{\{t<|u_\epsilon|\leq t+h\}}m(x) \, dx}{\mu_{u_\epsilon}(t)-\mu_{u_\epsilon}(t+h)}.
$$
As by (W2) $m\in L^\infty(\Omega)$, we obtain $\int_{\{t<|u_\epsilon|\leq t+h\}}m(x) \, dx \leq \|m\|_{L^\infty(\Omega)}(\mu_{u_\epsilon}(t)-\mu_{u_\epsilon}(t+h))$, and 
$$
(M_1)_\blacklozenge\left(\frac{d \omega_d^{1/d}\int_t^{t+h}\mu_{u_\epsilon}(\tau)^{\frac{d-1}{d}}\, d\tau}{\mu_{u_\epsilon}(t)-\mu_{u_\epsilon}(t+h)}\right)\leq \frac{1}{c_A}\left(\|m\|_{L^\infty(\Omega)}+\frac{ h\int_{\{t<|u_{\epsilon}|\}}|f(x)|\, dx}{\mu_{u_\epsilon}(t)-\mu_{u_\epsilon}(t+h)}\right).
$$
Now, let $\beta>0$. If 
$$
\|m\|_{L^\infty(\Omega)}\leq \beta\frac{ h\int_{\{t<|u_{\epsilon}|\}}|f(x)|\, dx}{\mu_{u_\epsilon}(t)-\mu_{u_\epsilon}(t+h)},
$$
then 
$$
(M_1)_\blacklozenge\left(\frac{d \omega_d^{1/d}\int_t^{t+h}\mu_{u_\epsilon}(\tau)^{\frac{d-1}{d}}\, d\tau}{\mu_{u_\epsilon}(t)-\mu_{u_\epsilon}(t+h)}\right)\leq \frac{1+\beta}{c_A}\frac{ h\int_{\{t<|u_{\epsilon}|\}}|f(x)|\, dx}{\mu_{u_\epsilon}(t)-\mu_{u_\epsilon}(t+h)},
$$
and 
\begin{align*}
& \Psi_\blacklozenge\left(\frac{d \omega_d^{1/d}\int_t^{t+h}\mu_{u_\epsilon}(\tau)^{\frac{d-1}{d}}\, d\tau}{\mu_{u_\epsilon}(t)-\mu_{u_\epsilon}(t+h)}\right) = \frac{(M_1)_\blacklozenge\left(\frac{d \omega_d^{1/d}\int_t^{t+h}\mu_{u_\epsilon}(\tau)^{\frac{d-1}{d}}\, d\tau}{\mu_{u_\epsilon}(t)-\mu_{u_\epsilon}(t+h)}\right)}{\frac{d \omega_d^{1/d}\int_t^{t+h}\mu_{u_\epsilon}(\tau)^{\frac{d-1}{d}}\, d\tau}{\mu_{u_\epsilon}(t)-\mu_{u_\epsilon}(t+h)}}\\
& \ \ \ \ \qquad \leq \frac{1+\beta}{c_A}\frac{ h\int_{\{t<|u_{\epsilon}|\}}|f(x)|\, dx}{d \omega_d^{1/d}\int_t^{t+h}\mu_{u_\epsilon}(\tau)^{\frac{d-1}{d}}\, d\tau},
\end{align*}
whence
$$
\frac{d \omega_d^{1/d}\int_t^{t+h}\mu_{u_\epsilon}(\tau)^{\frac{d-1}{d}}\, d\tau}{\mu_{u_\epsilon}(t)-\mu_{u_\epsilon}(t+h)} \leq \Psi_\blacklozenge^{-1}\left(\frac{1+\beta}{c_A}\frac{ h\int_{\{t<|u_{\epsilon}|\}}|f(x)|\, dx}{d \omega_d^{1/d}\int_t^{t+h}\mu_{u_\epsilon}(\tau)^{\frac{d-1}{d}}\, d\tau}\right).
$$
On the other hand, if 
$$
\|m\|_{L^\infty(\Omega)} > \beta\frac{ h\int_{\{t<|u_{\epsilon}|\}}|f(x)|\, dx}{\mu_{u_\epsilon}(t)-\mu_{u_\epsilon}(t+h)},
$$
then
$$
\frac{d \omega_d^{1/d}\int_t^{t+h}\mu_{u_\epsilon}(\tau)^{\frac{d-1}{d}}\, d\tau}{\mu_{u_\epsilon}(t)-\mu_{u_\epsilon}(t+h)} \leq (M_1)_\blacklozenge^{-1}\left(\frac{1}{c_A}\left(1+\frac{1}{\beta}\right)\|m\|_{L^\infty(\Omega)}\right).
$$
In either case
\begin{align*}
& \frac{d \omega_d^{1/d}\frac{1}{h}\int_t^{t+h}\mu_{u_\epsilon}(\tau)^{\frac{d-1}{d}}\, d\tau}{\frac{\mu_{u_\epsilon}(t)-\mu_{u_\epsilon}(t+h)}{h}} \\
& \ \ \leq \left(M_1\right)_\blacklozenge^{-1}\left(\frac{1}{c_A}\left(1+\frac{1}{\beta}\right)\|m\|_{L^\infty(\Omega)}\right) + \Psi_\blacklozenge^{-1}\left(\frac{1+\beta}{c_A}\frac{ \int_{\{t<|u_{\epsilon}|\}}|f(x)|\, dx}{d \omega_d^{1/d}\frac{1}{h}\int_t^{t+h}\mu_{u_\epsilon}(\tau)^{\frac{d-1}{d}}\, d\tau}\right),
\end{align*}
for every $\beta>0$ and almost every $t\in (0,\esssup_{x\in \Omega}|u_\epsilon(x)|)$.
Passing to the limit as $h\to 0^+$ yields
$$
1\leq \frac{-\mu_{u_\epsilon}'(t)}{d \omega_d^{1/d}\mu_{u_\epsilon}(t)^{\frac{d-1}{d}}}\left(\left(M_1\right)_\blacklozenge^{-1}\left(\frac{1+\beta}{c_A \beta}\|m\|_{L^\infty(\Omega)}\right) + \Psi_\blacklozenge^{-1}\left(\frac{1+\beta}{c_A}\frac{ \int_{0}^{\mu_{u_\epsilon}(t)}f^*(s)\, ds}{d \omega_d^{1/d}\mu_{u_\epsilon}(t)^{\frac{d-1}{d}}}\right)\right).
$$
After integrating this inequality from $0$ to $\esssup_{x\in \Omega}|u_\epsilon(t)|$ and choosing $\beta = 1- \lambda$, where $\lambda$ is as in \eqref{ass:f}, we obtain
\begin{align*}
& \esssup_{x\in \Omega}|u_\epsilon(t)|\leq \int_{0}^{|\Omega|}\frac{1}{d \omega_d^{1/d}r^{\frac{d-1}{d}}}\left(\left(M_1\right)_\blacklozenge^{-1}\left(\frac{\lambda}{\lambda-1}\|m\|_{L^\infty(\Omega)}\right) + \Psi_\blacklozenge^{-1}\left(\lambda\frac{ \int_{0}^rf^*(s)\, ds}{d \omega_d^{1/d}r^{\frac{d-1}{d}}}\right)\right)\, dr\\
& \ \ \leq \left(M_1\right)_\blacklozenge^{-1}\left(\frac{\lambda}{c_A(\lambda-1)}\|m\|_{L^\infty(\Omega)}\right)\left(\frac{|\Omega|}{\omega_d}\right)^{\frac{1}{d}} +   \frac{1}{d \omega_d^{1/d}}\int_{0}^{|\Omega|} {r^{\frac{1}{d}-1}}\Psi_\blacklozenge^{-1}\left(\frac{\lambda}{c_A d \omega_d^{1/d}}\frac{ \int_{0}^rf^*(s)\, ds }{r^{\frac{d-1}{d}}}\right)\, dr.
\end{align*}
By (W1) this implies that $u_\epsilon \in L^\infty(\Omega)$, and norm $\|u_\epsilon\|_{L^\infty(\Omega)}$ is uniformly bounded with respect to $\epsilon$, which yields the required assertion.

  \bibliography{musielak_bib}

\begin{bibdiv}
\begin{biblist}

\bib{Ahmida}{article}{
      author={Ahmida, Y.},
      author={Chlebicka, I.},
      author={Gwiazda, P.},
      author={Youssfi, A.},
       title={{G}ossez's approximation theorems in the
  {M}usielak--{O}rlicz--{S}obolev spaces},
        date={2018},
     journal={J. Funct. Anal.},
      volume={275},
       pages={2538\ndash 2571},
}

\bib{Chianci_Alberico}{article}{
      author={Alberico, A.},
      author={Chlebicka, I.},
      author={Cianchi, A.},
      author={Zatorska-Goldstein, A.},
       title={Fully anisotropic elliptic problems with minimally integrable
  data},
        date={2019},
     journal={Calc. Var. PDEs},
      volume={58},
       pages={186},
}

\bib{Alberico_Cianchi}{article}{
      author={Alberico, A.},
      author={Cianchi, A.},
       title={Comparison estimates in anisotropic variational problems},
        date={2008},
     journal={manuscripta math.},
      volume={126},
       pages={481\ndash 503},
}

\bib{Alberti}{article}{
      author={Alberti, G.},
      author={Ambrosio, L.},
       title={A geometrical approach to monotone functions in $\mathbb{R}^n$},
        date={1999},
     journal={Math. Z.},
      volume={230},
       pages={259\ndash 316},
}

\bib{Ambrosio_Gigli_Savare}{book}{
      author={Ambrosio, L.},
      author={Gigli, N.},
      author={Savar\'{e}, G.},
       title={Gradient flows in metric spaces and in the space of probability
  measures},
   publisher={Birkh\"{a}user Verlag},
     address={Basel, Boston, Berlin},
        date={2005},
}

\bib{Ball_fundamental}{inproceedings}{
      author={Ball, J.M.},
       title={A version of the fundamental theorem for {Y}oung measures},
        date={1989},
   booktitle={{P}{D}{E}s and {C}ontinuum {M}odels of {P}hase {T}ransitions,
  {P}roceedings of an {N}{S}{F}-{C}{N}{R}{S} {J}oint {S}eminar {H}eld in
  {N}ice, {F}rance, {J}anuary 18–22, 1988},
      series={Lecture Notes in Physics},
      volume={344},
       pages={207\ndash 215},
}

\bib{ball_murat}{article}{
      author={Ball, J.M.},
      author={Murat, F.},
       title={Remarks on {C}hacon's biting lemma},
        date={1989},
     journal={Proceedings of the American Mathematical Society},
      volume={107},
       pages={655\ndash 663},
}

\bib{Baroni_Colombo_Mignone}{article}{
      author={Baroni, P.},
      author={Colombo, M.},
      author={Mingione, G.},
       title={Nonautonomous functionals, borderline cases and related function
  classes},
        date={2015},
     journal={St. Petersburg Math. J.},
      volume={27},
       pages={347\ndash 379},
}

\bib{Baroni_Colombo_Mignone_2}{article}{
      author={Baroni, P.},
      author={Colombo, M.},
      author={Mingione, G.},
       title={Regularity for general functionals with double phase},
        date={2018},
     journal={Calc. Var. Partial Differential Equations},
      volume={57},
       pages={57\ndash 62},
}

\bib{bendahmahe}{article}{
      author={Bendahmane, M.},
      author={Wittbold, P.},
       title={Renormalized solutions for nonlinear elliptic equations with
  variable exponents and {${L}^1$} data},
        date={2009},
     journal={Nonlinear Analysis: Theory, Methods \& Applications},
      volume={70},
       pages={567\ndash 583},
}

\bib{Boccardo}{article}{
      author={B\'{e}nilan, P.},
      author={Boccardo, L.},
      author={Gallou\"{e}t, T.},
      author={Gariepy, R.},
      author={Pierre, M.},
      author={Vazquez, J.L.},
       title={An ${L}^1$-theory of existence and uniqueness of solutions of
  nonlinear elliptic equations},
        date={1995},
     journal={Annali della Scuola Normale Superiore di Pisa - Classe di
  Scienze},
      volume={22},
       pages={241\ndash 273},
}

\bib{benilan_wittbold}{article}{
      author={Benilan, P.},
      author={Wittbold, P.},
       title={On mild and weak solutions of elliptic-parabolic problems},
        date={1996},
     journal={Advances in Differential Equations},
      volume={1},
       pages={1053\ndash 1073},
}

\bib{leonetti}{article}{
      author={Bhattacharya, T.},
      author={Leonetti, F.},
       title={A new poincar\'{e} inequality and its application to the
  regularity of minimizers of integral functionals with nonstandard growth},
        date={1991},
     journal={Nonlinear Anal.},
      volume={17},
       pages={833\ndash 839},
}

\bib{BG}{article}{
      author={Boccardo, L.},
      author={Gallou\"{e}t, T.},
       title={Nonlinear elliptic equations with right-hand side measures},
        date={1992},
     journal={Comm. Partial Differential Equations},
      volume={17},
       pages={641\ndash 655},
}

\bib{BDS}{article}{
      author={Buli\v{c}ek, M.},
      author={Diening, L.},
      author={Schwarzacher, S.},
       title={Existence, uniqueness and optimal regularity results for very
  weak solutions to nonlinear elliptic systems},
        date={2016},
     journal={Anal. PDE},
      volume={9},
       pages={1115\ndash 1151},
}

\bib{kalousek}{article}{
      author={Bul\'{i}\v{c}ek, M.},
      author={Gwiazda, P.},
      author={Kalousek, M.},
      author={\'{S}wierczewska Gwiazda, A.},
       title={Homogenization of nonlinear elliptic systems in nonreflexive
  {M}usielak--{O}rlicz spaces},
        date={2019},
     journal={Nonlinearity},
      volume={32},
       pages={1073},
}

\bib{Bulicek}{article}{
      author={Bul\'{i}\v{c}ek, M.},
      author={Gwiazda, P.},
      author={Malek, J.},
      author={\'{S}wierczewska Gwiazda, A.},
       title={On unsteady flows of implicitly constituted incompressible
  fluids},
        date={2012},
     journal={SIAM J. Math. Anal.},
      volume={44},
       pages={2756\ndash 2801},
}

\bib{gwiazda_functional}{article}{
      author={Bul\'{i}\v{c}ek, M.},
      author={Gwiazda, P.},
      author={\'{S}wierczewska Gwiazda, A.},
       title={On unified theory for scalar conservation laws with fluxes and
  sources discontinuous with respect to the unknown},
        date={2017},
     journal={J. Differential Equations},
      volume={262},
       pages={313\ndash 364},
}

\bib{byun_oh}{article}{
      author={Byun, S.-S.},
      author={Oh, J.},
       title={Regularity results for generalized double phase functionals},
        date={to appear},
     journal={Anal. \& PDE},
}

\bib{byun2}{article}{
      author={Byun, S.-S.},
      author={Youn, Y.},
       title={Potential estimates for elliptic systems with subquadratic
  growth},
        date={2019},
     journal={J. Math. Pures Appl.},
      volume={131},
       pages={193\ndash 224},
}

\bib{Carillo}{article}{
      author={Carillo, J.},
       title={Conservation laws with discontinuous flux functions and boundary
  condition},
        date={2003},
     journal={J. Evol. Equ.},
      volume={3},
       pages={283\ndash 301},
}

\bib{Chiado}{article}{
      author={Chiado'Piat, V.},
      author={Dal~Maso, G.},
      author={Defrancheschi, A.},
       title={G-convergence of monotone operators},
        date={1990},
     journal={Annales de l'H.P., section C},
      volume={7},
       pages={123\ndash 160},
}

\bib{IC1}{article}{
      author={Chlebicka, I.},
       title={Gradient estimates for problems with orlicz growth},
        date={2020},
     journal={Nonlinear Analysis},
      volume={194},
       pages={111364},
}

\bib{chlebicka_guide}{article}{
      author={Chlebicka, I.},
       title={A pocket guide to nonlinear differential equations in
  {M}usielak--{O}rlicz spaces},
        date={2018},
     journal={Nonlinear Analysis},
      volume={175},
       pages={1\ndash 27},
}

\bib{CDF}{article}{
      author={Chlebicka, I.},
      author={De~Filippis, C.},
       title={Removable sets in non-uniformly elliptic problems},
        date={2020},
     journal={Annali di Matematica Pura ed Applicata},
      volume={199},
       pages={619\ndash 649},
}

\bib{giannetti}{article}{
      author={Chlebicka, I.},
      author={Giannetti, F.},
      author={Zatorska-Goldstein, A.},
       title={Elliptic problems with growth in nonreflexive {O}rlicz spaces and
  with measure or $l^1$ data},
        date={2019},
     journal={J. Math. Anal. Appl.},
      volume={479},
       pages={185\ndash 213},
}

\bib{chlebicka_parabolic2}{article}{
      author={Chlebicka, I.},
      author={Gwiazda, P.},
      author={Zatorska-Goldstein, A.},
       title={Parabolic equation in time and space dependent anisotropic
  {M}usielak--{O}rlicz spaces in absence of {L}avrentiev’s phenomenon},
        date={2019},
     journal={Annales de l'Institut Henri Poincar\'{e} C, Analyse non
  lin\'{e}aire},
      volume={36},
       pages={1431\ndash 1465},
}

\bib{chlebicka_parabolic3}{article}{
      author={Chlebicka, I.},
      author={Gwiazda, P.},
      author={Zatorska-Goldstein, A.},
       title={Renormalized solutions to parabolic equations in time and space
  dependent anisotropic {M}usielak--{O}rlicz spaces in absence of
  {L}avrentiev’s phenomenon},
        date={2019},
     journal={J. Differ. Equations},
      volume={267},
       pages={1129\ndash 1166},
}

\bib{chlebicka_parabolic}{article}{
      author={Chlebicka, I.},
      author={Gwiazda, P.},
      author={Zatorska-Goldstein, A.},
       title={Well-posedness of parabolic equations in the non-reflexive and
  anisotropic {M}usielak--{O}rlicz spaces in the class of renormalized
  solutions},
        date={2018},
     journal={J. Differ. Equations},
      volume={265},
       pages={5716\ndash 5766},
}

\bib{Cianchi}{article}{
      author={Cianchi, A.},
       title={A fully anisotropic {S}obolev inequality},
        date={2000},
     journal={Pacific Journal of Mathematics},
      volume={196},
       pages={283\ndash 294},
}

\bib{Cianchi1}{article}{
      author={Cianchi, A.},
       title={Symmetrization in anisotropic elliptic problems},
        date={2007},
     journal={Comm. Partial Differential Equations},
      volume={32},
       pages={693\ndash 717},
}

\bib{C3}{article}{
      author={Cianchi, A.},
      author={Maz'ya, V.},
       title={Quasilinear elliptic problems with general growth and merely
  integrable, or measure, data},
        date={2017},
     journal={Nonlinear Anal.},
      volume={164},
       pages={189\ndash 215},
}

\bib{Colombo_Mignone}{article}{
      author={Colombo, M.},
      author={Mingione, G.},
       title={Regularity for double phase variational problems},
        date={2015},
     journal={Arch. Ration. Mech. Anal.},
      volume={215},
       pages={443\ndash 496},
}

\bib{defili_Mignone}{article}{
      author={De~Filippis, C.},
      author={Mingione, G.},
       title={Manifold constrained non-uniformly elliptic problems},
        date={2020},
     journal={Journal of Geometric Analysis},
      volume={30},
       pages={1661\ndash 1723},
}

\bib{diperna}{article}{
      author={DiPerna, R.J.},
      author={Lions, P.-L.},
       title={Ordinary differential equations, transport theory and {S}obolev
  spaces},
        date={1989},
     journal={Invent. Math.},
      volume={98},
       pages={511\ndash 547},
}

\bib{evans_gariepy}{book}{
      author={Evans, L.C.},
      author={Gariepy, R.F.},
       title={Measure theory and fine properties of functions},
   publisher={CRC Press},
     address={Boca Raton, London, New York},
        date={2015},
}

\bib{Francfort}{article}{
      author={Francfort, G.},
      author={Murat, F.},
      author={Tartar, L.},
       title={Monotone operators in divergence form with $x$-dependent
  multivalued graphs},
        date={2004},
     journal={Bollettino dell'Unione Matematica Italiana},
      volume={7B},
       pages={23\ndash 59},
}

\bib{Griewank}{article}{
      author={Griewank, A.},
      author={Rabier, P.J.},
       title={On the smoothness of convex envelopes},
        date={1990},
     journal={Transactions of the American Mathematical Society},
      volume={322},
       pages={691\ndash 709},
}

\bib{gwiazda_minakowski}{article}{
      author={Gwiazda, P.},
      author={Minakowski, P.},
      author={Wr\'{o}blewska-Kami\'{n}ska, A.},
       title={Elliptic problems in generalized {O}rlicz--{M}usielak spaces},
        date={2012},
     journal={Centr. Eur. J. Math.},
      volume={10},
       pages={2019\ndash 2032},
}

\bib{gwiazda_skrzypczak}{article}{
      author={Gwiazda, P.},
      author={Skrzypczak, I.},
      author={Zatorska-Goldstein, A.},
       title={Existence of renormalized solutions to elliptic equation in
  {M}usielak--{O}rlicz space},
        date={2018},
     journal={J. Differ. Equations},
      volume={264},
       pages={341\ndash 377},
}

\bib{gwiazda_swierczewska}{article}{
      author={Gwiazda, P.},
      author={\'{S}wierczewska Gwiazda, A.},
       title={On non-{N}ewtonian fluids with a property of rapid thickening
  under different stimulus},
        date={2008},
     journal={Math. Models Methods Appl. Sci.},
      volume={18},
       pages={1073\ndash 1092},
}

\bib{gwiazda_2}{article}{
      author={Gwiazda, P.},
      author={\'{S}wierczewska Gwiazda, A.},
      author={Wittbold, P.},
      author={Zimmerman, A.},
       title={Multi-dimensional scalar balance laws with discontinuous flux},
        date={2014},
     journal={Journal of Functional Analysis},
      volume={267},
       pages={2846\ndash 2883},
}

\bib{gwiazda_wroblewska}{article}{
      author={Gwiazda, P.},
      author={\'{S}wierczewska Gwiazda, A.},
      author={Wr\'{o}blewska, A.},
       title={Monotonicity methods in generalized {O}rlicz spaces for a class
  on non-{N}ewtonian fluids},
        date={2010},
     journal={Math. Methods Appl. Sci.},
      volume={33},
       pages={125\ndash 137},
}

\bib{gwiazda_corrigendum}{article}{
      author={Gwiazda, P.},
      author={Wittbold, P.},
      author={Wr\'{o}blewska, A.},
      author={Zimmermann, A.},
       title={Corrigendum to ''{R}enormalized solutions of nonlinear elliptic
  problems in generalized {O}rlicz spaces'' [{J}. {D}ifferential {E}quations]
  253(2) (2012) 635--666},
        date={2012},
     journal={Journal of Differential Equations},
      volume={253},
       pages={2734\ndash 2738},
}

\bib{gwiazda_wittbold}{article}{
      author={Gwiazda, P.},
      author={Wittbold, P.},
      author={Wr\'{o}blewska, A.},
      author={Zimmermann, A.},
       title={Renormalized solutions of nonlinear elliptic problems in
  generalized {O}rlicz spaces},
        date={2012},
     journal={Journal of Differential Equations},
      volume={253},
       pages={635\ndash 666},
}

\bib{gwiazda_wittbold2}{article}{
      author={Gwiazda, P.},
      author={Wittbold, P.},
      author={Wr\'{o}blewska-Kami\'{n}ska, A.},
      author={Zimmermann, A.},
       title={Renormalized solutions to nonlinear parabolic problems in
  generalized musielak--{O}rlicz spaces},
        date={2015},
     journal={Nonlinear Analysis},
      volume={129},
       pages={1\ndash 36},
}

\bib{Gwiazda_zatorska}{article}{
      author={Gwiazda, P.},
      author={Zatorska-Goldstein, A.},
       title={On elliptic and parabolic systems with $x$-dependent multivalued
  graphs},
        date={2007},
     journal={Math. Methods Appl. Sci.},
      volume={30},
       pages={213\ndash 236},
}

\bib{HH}{book}{
      author={Harjulehto, P.},
      author={H\'{a}st\"{o}, P.},
       title={Orlicz spaces and generalized {O}rlicz spaces},
      series={Lecture Notes in Mathematics},
   publisher={Springer},
        date={2019},
      volume={2236},
}

\bib{Hasto2}{article}{
      author={Harjulehto, P.},
      author={H\'{a}st\"{o}, P.},
      author={Lee, M.},
       title={H\"{o}lder continuity of quasiminimizers and $\omega$-minimizers
  of functionals with generalized orlicz growth},
        date={2019},
     journal={Ann. Sc. Norm. Super. Pisa Cl. Sci.},
      volume={to appear},
}

\bib{Hasto3}{article}{
      author={Harjulehto, P.},
      author={H\'{a}st\"{o}, P.},
      author={Toivanen, O.},
       title={H\"{o}lder regularity of quasiminimizers under generalized growth
  conditions},
        date={2017},
     journal={Calc. Var. Partial Differential Equations},
      volume={56},
       pages={22},
}

\bib{Hasto1}{article}{
      author={H\'{a}st\"{o}, P.},
      author={Ok, J.},
       title={{C}alder\'{o}n--{Z}ygmund estimates in generalized {O}rlicz
  spaces},
        date={2019},
     journal={J. Differential Equations},
      volume={267},
       pages={2792\ndash 2823},
}

\bib{kilpel}{article}{
      author={Kilpel\'{a}inen, T.},
      author={Kuusi, T.},
      author={Tuhola-Kujanp\"{a}\"{a}, A.},
       title={Superharmonic functions are locally renormalized solutions},
        date={2011},
     journal={Ann. Inst. H. Poincar\'{e} Anal. Non Lin\'{e}aire},
      volume={28},
       pages={775\ndash 795},
}

\bib{Lavrentiev}{article}{
      author={Lavrentiev, M.},
       title={Sur quelques problemes du calcul des variations},
        date={1927},
     journal={Ann. Mat. Pura Appl.},
      volume={41},
       pages={107\ndash 124},
}

\bib{muller}{book}{
      author={M\"uller, S.},
       title={Variational models for microstructure and phase transitions},
      series={Lecture Notes},
   publisher={Max Planck Institut f\"ur Mathematik in den Naturwissenschaften},
     address={Leipzig},
        date={1998},
      volume={2},
}

\bib{pedregal}{book}{
      author={Pedregal, P.},
       title={Parametrized measures and variational principles},
      series={Progress in Nonlinear Differential Equations and Their
  Applications},
   publisher={Springer Basel AG},
        date={1997},
      volume={30},
}

\bib{rockafellar}{book}{
      author={Rockafellar, R.T.},
      author={Wets, R.J-B.},
       title={Variational analysis},
   publisher={Springer},
     address={Dordrecht, Heidelberg, London, New York},
        date={2009},
}

\bib{Schappacher}{article}{
      author={Schappacher, G.},
       title={A notion of {O}rlicz spaces for vector values functions},
        date={2005},
     journal={Applications of Mathematics},
      volume={50},
       pages={355\ndash 386},
}

\bib{skaff}{article}{
      author={Skaff, M.S.},
       title={Vector valued {O}rlicz spaces. {I}{I}},
        date={1969},
     journal={Pacific J. Math.},
      volume={28},
       pages={413\ndash 430},
}

\bib{swierczewska2}{article}{
      author={\'{S}wierczewska Gwiazda, A.},
       title={Anisotropic parabolic problems with slowly or rapidly growing
  terms},
        date={2014},
     journal={Colloq. Math.},
      volume={134},
       pages={113\ndash 130},
}

\bib{swierczewska1}{article}{
      author={\'{S}wierczewska Gwiazda, A.},
       title={Nonlinear parabolic problems in {M}usielak--{O}rlicz spaces},
        date={2014},
     journal={Nonlinear Anal.},
      volume={98},
       pages={48\ndash 65},
}

\bib{talenti}{article}{
      author={Talenti, G.},
       title={Boundedness of minimizers},
        date={1990},
     journal={Hokkaido Math. J.},
      volume={19},
       pages={259\ndash 279},
}

\bib{Trudinger}{article}{
      author={Trudinger, N.S.},
       title={An imbedding theorem for $h_0(g, \omega)$ spaces},
        date={1974},
     journal={Studia Math.},
      volume={50},
       pages={17\ndash 30},
}

\bib{wittbold_zimmerman}{article}{
      author={Wittbold, P.},
      author={Zimmermann, A.},
       title={Existence and uniqueness of renormalized solutions to nonlinear
  elliptic equations with variable exponents and ${L}^1$-data},
        date={2010},
     journal={Nonlinear Analysis: Theory, Methods \& Applications},
      volume={72},
       pages={2990\ndash 3008},
}

\bib{wroblewska}{article}{
      author={Wr\'{o}blewska, A.},
       title={Steady flow of non-{N}ewtonian fluids -- monotonicity methods in
  generalized {O}rlicz spaces},
        date={2010},
     journal={Nonlinear Anal.},
      volume={72},
       pages={4136\ndash 4147},
}

\bib{ZZ}{article}{
      author={Zhang, C.},
      author={Zhou, S.},
       title={Entropy and renormalized solutions for the $p(x)$-laplacian
  equation with measure data},
        date={2010},
     journal={Bull. Aust. Math. Soc.},
      volume={82},
       pages={459\ndash 479},
}

\bib{Zhikov1}{article}{
      author={Zhikov, V.V.},
       title={On {L}avrentiev's phenomenon},
        date={1995},
     journal={Russ. J. Math. Phys.},
      volume={3},
       pages={249\ndash 269},
}

\bib{Zhikov}{article}{
      author={Zhikov, V.V.},
       title={On variational problems and nonlinear elliptic equations with
  nonstandard growth conditions},
        date={2011},
     journal={J. Math. Sci. (N.Y.)},
      volume={173},
       pages={463\ndash 570},
}

\end{biblist}
\end{bibdiv}
  
  \appendix

  \section{$N$-functions and Musielak--Orlicz spaces.}\label{appa}
  \noindent \textbf{$N$-functions}. We start from the definition of $N$-functions.  
  \begin{definition}\label{def:N}
  	The function $M:\Omega\times\mathbb{R}^d \to [0,\infty)$ is an $N$-function if
  	\begin{itemize}
  		\item[(N1)] $M$ is Carath\'{e}odory, that is, $M(\cdot,\xi)$ is measurable for every $\xi\in \mathbb{R}^d$ and $M(x,\cdot)$ is continuous for almost every $x\in \Omega$,
  		\item[(N2)] $M(x,\xi) = M(x,-\xi)$ for every $\xi\in \mathbb{R}^d$ a.e. in $\Omega$ and $M(x,\xi) = 0$ is and only if $\xi=0$ a.e. in $\Omega$,
  		\item[(N3)] $M(x,\cdot)$ is convex for almost every $x\in \Omega$,
  		\item[(N4)] $M$ has superlinear growth in $\xi$ at zero and infinity, that is,
  		$$
  		\lim_{|\xi|\to 0}\esssup_{x\in \Omega}\frac{M(x,\xi)}{|\xi|} = 0\qquad \textrm{and}\qquad \lim_{|\xi|\to \infty}\essinf_{x\in \Omega}\frac{M(x,\xi)}{|\xi|} = \infty. 
  		$$ 
  		\item[(N5)] $
  		\essinf_{x\in \Omega}\inf_{|\xi|=s}M(x,\xi) > 0 \  \textrm{for every}\ s\in (0,\infty)$\\  and $\esssup_{x\in \Omega}M(x,\xi) < \infty \ \textrm{for every}\ \xi\neq 0.
  		$ 
  	\end{itemize}
  \end{definition}
  If $d=1$ we will use small letters, such as $m$, to denote $N$-functions, we will call such functions one dimensional $N$-functions. In such case by assumption (N2) there holds $m(x,-\xi) = m(x,\xi)$ for every $\xi\in \R$ so it is enough to define one dimensional $N$-function for $\xi\in [0,\infty)$.  Indeed we will assume that one dimensional $N$-functions are defined only on $\Omega\times [0,\infty)$, and with some abuse of notation we will define one dimensional $N$-functions sometimes on the whole line and sometimes on the half-line. If an $N$-function does not depend on $x$ we will call it homogeneous. So, homogeneous $N$-function leads from $\R^d$ to $ [0,\infty)$ and homogeneous one dimensional $N$-function from $[0,\infty)$ to $[0,\infty)$. 
  If $M:\Omega\times \R^d\to \R$ then its \textit{complementary function} $\widetilde{M}$ is defined by the Fenchel transform in the following way
  \begin{equation}\label{fenchel}
  \widetilde{M}(x,\eta) = \sup_{\xi\in \mathbb{R}^d} \left\{ \xi\cdot \eta - M (x,\xi) \right\}. 
  \end{equation}
  
  In the following results we discuss the assumptions in the definition of an $N$-function and establish some of its properties. We remark that the behavior of an $N$-function close to the origin  is not important for the main results of the present paper, as it does not influence the generated Musielak--Orlicz space nor the arguments of the proofs. In the discussion below, however, for the sake of the exposition completeness, we discuss both the behavior close to infinity as well as close to the origin.  
      
  \begin{remark}
  	It is easy to verify that the complementary function of a function which satisfies (N1)--(N4) satisfies (N1)--(N3). It does not have to satisfy either the first, or  the second assertion of (N4). Indeed, let $\Omega=(0,1)$ and let 
  	$$M(x,\xi) = \begin{cases}
  	x\frac{|\xi|^2}{2}\ \textrm{when}\ |\xi|\leq 1,\\
  	\frac{|\xi|^2}{2}-\frac{1}{2}+\frac{x}{2}\ \textrm{otherwise}.
  	\end{cases}\quad \textrm{Then}\quad \widetilde{M}(x,\eta)=\begin{cases}\frac{|\eta|^2}{2x}\ \textrm{when}\ |\eta|\leq x,\\
  	|\eta|-\frac{x}{2}\ \textrm{when}\ |\eta|\in (x,1),\\
  	\frac{|\eta|^2}{2}-\frac{x}{2}+\frac{1}{2}\ \textrm{when}\ |\eta|>1.
  	\end{cases}.
  	$$ It is clear that
  	$$
  	\lim_{|\xi|\to 0}\esssup_{x\in \Omega}\frac{\widetilde{M}(x,\eta)}{|\eta|} = \frac{1}{2}.
  	$$ 
  	Moreover,
  	$$
  	\textrm{if}\ \  
  	M(x,\xi) = \begin{cases}
  	\frac{|\xi|^2}{2}\quad \textrm{when}\quad |\xi|\leq \sqrt{2},\\
  	\frac{|\xi|^2}{2x} + 1 - \frac{1}{x}\quad \textrm{otherwise},
  	\end{cases}\ \  \textrm{then}\ \  
  	\widetilde{M}(x,\eta) = \begin{cases}
  	\frac{|\eta|^2}{2}\quad \textrm{when}\quad |\eta|\leq \sqrt{2},\\
  	\sqrt{2}|\eta| - 1\quad \textrm{when}\quad \eta\in(\sqrt{2},\sqrt{2}/x),\\
  	\frac{|\eta|^2x}{2} - 1 + \frac{1}{x}\quad \textrm{otherwise}.
  	\end{cases}
  	$$
  	It is not hard to verify that 
  	$$
  	\lim_{|\eta|\to \infty}\essinf_{x\in \Omega}\frac{\widetilde{M}(x,\eta)}{|\eta|} = \sqrt{2}.
  	$$
  	Note that the two examples do not satisfy (N5). As we will later show, the complementary function of an $N$-function is also an $N$-function.
  \end{remark}
  
  In the following Lemma \ref{lem:nfunction} and Remark \ref{rem:N} we establish that $N$-functions always have a minorant and majorant being one dimensional homogeneous $N$-functions. 
   
  \begin{lemma}\label{lem:nfunction}
  	Let $M:\Omega\times\mathbb{R}^d \to \mathbb{R}$ be a function satisfying (N1)--(N4).
  	Then $M:\Omega\times\mathbb{R}^d \to \mathbb{R}$ is an $N$-function (i.e. it satisfies (N5)) if and only if it is \textit{stable}, i.e., if there exist homogeneous one dimensional $N$-functions $m_1, m_2:[0,\infty)\to [0,\infty)$ such that for every $\xi\in \R^d$ and almost every $x\in \Omega$ there holds
  	$$
  	m_1(|\xi|) \leq M(x,\xi) \leq m_2(|\xi|).
  	$$ 
  	In particular every $N$-function is stable.
  \end{lemma}
  \begin{proof}
  	The fact that stability implies (N5) is straightforward. For the opposite implication define $m_2:[0,\infty) \to [0,\infty)$ by the formula
  	$$
  	m_2(s) = \esssup_{x\in \Omega} \sup_{|\xi|=s} M(x,\xi).
  	$$
  	It is straightforward to check that this function is finite, nonzero for $\xi\neq 0$, and satisfies (N1)--(N4).
  	To get the lower bound, let us define
  	$$m_{inf}(s) = \essinf_{x\in \Omega} \inf_{|\xi|=s} M(x,\xi).$$ 
  	The function $m_{inf}$ is nondecreasing and hence it has at most countable number of discontinuities and each of these discontinuities is a jump point. Define
  	$$
  	m_{lsc}(s) = \begin{cases}
  	m_{inf}(s)\ \mathrm{if}\ m_{inf}\ \textrm{is continuous at}\  s,\\
  	m_{inf}(s^-)\ \mathrm{otherwise}.
  	\end{cases}
  	$$
  	This function is in fact the lower semicontinuous envelope of $m_{inf}$. Now define $m_1:[0,\infty) \to [0,\infty)$ as 
  	$$
  	m_1(s) = \widetilde{\widetilde{m_{lsc}}}(s).
  	$$
  	that is, the greatest convex minorant of $m_{lsc}$. 
  	The fact that $m_1$ satisfies (N1)--(N3), (N5), as well as the growth at zero in (N4) is clear. To prove the growth at infinity assume, for contradiction, that there exist  constants $c< \infty$, $R > 0$ and $\epsilon>0$ such that for every $s\geq R$ 
  	$$
  	m_1(s) < cs +\epsilon.
  	$$
  	So, for a sequence $s_n\to \infty$, there exist numbers $s_n^1,s_n^2$ and $\lambda_n^1, \lambda_n^2\geq 0$ such that $\lambda_n^1 + \lambda_n^2 = 1$ and
  	$$
  	(s_n,m_1(s_n)) = \left(\lambda_n^1 s_n^1 + \lambda_n^1 s_n^2, \lambda_n^1 m_{lsc}(s_n^1) + \lambda_n^1 m_{lsc}(s_n^2)\right). 
  	$$
  	Moreover 
  	$$
  	s_n^1 \geq s_n, \quad m_{lsc}(s_n^1) = m_1(s_n^1)\quad \textrm{and}\quad \lambda_n^1 > 0.
  	$$
  	Such a choice of $s_n^1, \lambda_n^1, s_n^2, \lambda_n^2$ is possible due to \cite[Theorem 2.1, Remark 2.1]{Griewank}.  
  	This means that $s_n^1\to \infty$ as $n\to\infty$ and $m_{lsc}(s_n^1) = m_1(s_n^1) < cs_n^1 +\epsilon$, whence
  	$$
  	\frac{m_{lsc}(s_n^1)}{s_n^1} \leq c + \frac{\epsilon}{s_n^1}.
  	$$
  	Thus we can construct a sequence $r_n \to \infty$ such that 
  	$$
  	\limsup_{n\to\infty} \frac{m_{inf}(r_n)}{r_n}\leq c.
  	$$
  	This	is a contradiction with superlinear growth at infinity of $M$. 	 
  \end{proof}

  \begin{remark}\label{rem:N}
  	We are tempted to replace the lower bound in (N5) with its weaker version 
  	$$
  	\essinf_{x\in \Omega}M(x,\xi) > 0\ \textrm{for every}\ \xi\neq 0.
  	$$
  	Unfortunately in such case Lemma \ref{lem:nfunction} does not hold anymore. To demonstrate this consider $\Omega=(0,1)$, $d=2$ and let $\theta(\xi)$ denote the angular polar coordinate of $\xi$. Define the function
  	$$
  	P(x,\xi) = \begin{cases}
  	|\xi|^2\ \ \textrm{if}\ \ |\xi|\neq 1,\\
  	1\ \ \textrm{if}\ \ |\xi|= 1\ \textrm{and}\ \theta(\xi)\notin (0,x)\cup(\pi,\pi+x),\\
  	x+(1-x)\cos^2\left(\frac{\theta(\xi)}{x}\pi\right)\ \textrm{if}\ \ |\xi|= 1\ \textrm{and} \ \theta(\xi)\in (0,x),\\
  	x+(1-x)\cos^2\left(\frac{\theta(\xi)-\pi}{x}\pi\right)\ \textrm{if}\ \ |\xi|= 1\ \textrm{and} \ \theta(\xi)\in (\pi,\pi+x).
  	\end{cases}
  	$$
  	Note that $P$ is lower semicontinuous with respect to $\xi$. Define $M(x,\xi) = \widetilde{\widetilde{P}}(x,\xi)$, the convex envelope of $P$. Then if only $|\xi|=1$ and $\theta(\xi)\in (0,1)\cup(\pi,\pi+1)$ we obtain 
  	$$
  	\essinf_{x\in \Omega} M(x,\xi) \leq 		\essinf_{x\in \Omega} P(x,\xi) = \frac{\theta(\xi)}{2}.
  	$$ 
  	This means that 
  	$$
  	\inf_{|\xi|=1}\essinf_{x\in \Omega} M(x,\xi) = 0,
  	$$
  	and hence we can have only $m_1(1) = 0$. Lower semicontinuity of $P$ with respect to $\xi$ and \cite[Theorem 2.1, Remark 2.1]{Griewank} imply that $\essinf_{x\in \Omega} M(x,\xi) > 0$ for every nonzero $\xi$. We will verify that $M$ satisfies all remaining conditions in the definition on an $N$-function. Indeed, (N1)-(N3), growth at zero in (N4), and condition with $\esssup$ in (N5) are clear. We will verify growth at infinity in (N4). Assume for contradiction that there exists a sequence $\xi_n$ with $|\xi_n|\to \infty$ and a constant $C>0$ such that 
  	$$
  	\essinf_{x\in \Omega}\frac{M(x,\xi_n)}{|\xi_n|} \leq C.
  	$$
  	This means that there exists $x_n \in (0,1)$ such that 
  	$$
  	\frac{M(x_n,\xi_n)}{|\xi_n|} \leq C+1.
  	$$ Now \cite[Theorem 2.1]{Griewank} implies that there exist nonnegative numbers $\lambda^1_n+\lambda^2_n+\lambda^3_n = 1$ and $\xi^1_n,\xi^2_n,\xi^3_n\in \R^2$ such that 
  	$\xi_n = \lambda^1_n\xi^1_n + \lambda^2_n\xi^2_n+\lambda^3_n\xi^3_n$ and 
  	$$
  	M(x_n,\xi_n) = \lambda^1_nP(x_n,\xi^1_n)+\lambda^2_nP(x_n,\xi^2_n)+\lambda^3_nP(x_n,\xi^3_n).
  	$$
  	We can assume that $|\xi^1_n|\geq |\xi_n|$ and $\lambda^1_n > 0$. If, for a given $n$, neither of $|\xi^k_n|$ is on a unit circle then
  	$$
  	|\xi_n|^2\leq \lambda_n^1|\xi_n^1|^2+		\lambda_n^2|\xi_n^2|^2+		\lambda_n^3|\xi_n^3|^2 = M(x_n,\xi_n) \leq P(x_n,\xi_n)= |\xi_n|^2.
  	$$
  	It follows that 
  	$$
  	\frac{M(x_n,\xi_n)}{|\xi_n|} = |\xi_n|\leq C+1.
  	$$
  	If exactly one of $\xi_n^k$ is on the unit circle, say $\xi^3_n$, then
  	\begin{align*}
  	&|\xi_n|^2\leq \lambda_n^1|\xi_n^1|^2+		\lambda_n^2|\xi_n^2|^2+		\lambda_n^3|\xi_n^3|^2 = \lambda_n^1|\xi_n^1|^2+		\lambda_n^2|\xi_n^2|^2+		\lambda_n^3 + \lambda_n^3P(x_n,\xi_n^3) - \lambda^3_nP(x_n,\xi_n^3)\\
  	& \ \  =  M(x_n,\xi_n) + \lambda_n^3(1-P(x_n,\xi_n^3))\leq  M(x_n,\xi_n) + 1 \leq (C+1)|\xi_n| + 1.
  	\end{align*}
  	It must be 
  	$$
  	|\xi_n| \leq (C+1) + \sqrt{2}.
  	$$
  	In the final possibility two points $\xi_n^2$ and $\xi_n^3$ are on the unit circle. Then 
  	\begin{align*}
  	&|\xi_n|^2\leq \lambda_n^1|\xi_n^1|^2+		\lambda_n^2|\xi_n^2|^2+		\lambda_n^3|\xi_n^3|^2\\
  	& \ \  = \lambda_n^1|\xi_n^1|^2+		\lambda_n^2+\lambda_n^2P(x_n,\xi_n^2) - \lambda^2_nP(x_n,\xi_n^2)	+	\lambda_n^3 + \lambda_n^3P(x_n,\xi_n^3) - \lambda^3_nP(x_n,\xi_n^3)\\
  	& \ \  =  M(x_n,\xi_n) + \lambda_n^2(1-P(x_n,\xi_n^2)) + \lambda_n^3(1-P(x_n,\xi_n^3))\leq  M(x_n,\xi_n) + 2 \leq (C+1)|\xi_n| + 2.
  	\end{align*}
  	It follows that $|\xi_n|$ is bounded, and the proof by contradiction is complete. 
  \end{remark}
  
  \begin{remark}
  	Sometimes the assertion (N4) in the definition of an N-function is replaced with its weaker, nonuniform version
  	\begin{itemize}
  		\item[(N4')] \ \ \ $	\lim_{|\xi|\to 0}\frac{M(x,\xi)}{|\xi|} = 0\qquad \textrm{and}\qquad \lim_{|\xi|\to \infty}\frac{M(x,\xi)}{|\xi|} = \infty \quad \textrm{for a.e.}\quad x\in \Omega$. 
  	\end{itemize}
  	It is clear that (N4) implies (N4'). We demonstrate that if we replace (N4) with (N4') in the definition of an $N$-function, then we lose the stability property. Indeed, consider  $\Omega = (0,1)$, $d=1$ and 
  	$$
  	M(x,\xi) = \begin{cases}
  	\frac{|\xi|^2}{x+|\xi|}\quad \textrm{for} \quad |\xi|\leq 1,\\
  	\frac{|\xi|^3}{x+|\xi|}\quad \textrm{otherwise}.
  	\end{cases}
  	$$
  	This function satisfies (N1)--(N3), as well as (N4') and (N5), but it is not stable, as 
  	$$
  	m_2(s) = \begin{cases}
  	s\quad \textrm{for} \quad |\xi|\leq 1,\\
  	s^2\quad \textrm{otherwise},
  	\end{cases}
  	$$
  	which is not an $N$-function as it has linear growth at zero.  Moreover, consider $\Omega = (0,1)$, $d=1$ and 
  	$$
  	M(x,\xi) = \begin{cases}
  	\left(\frac{1}{2}+x\right)|\xi|^2\quad \textrm{for} \quad |\xi|\leq 1,\\
  	|\xi|+x|\xi|^2\quad \textrm{otherwise}.
  	\end{cases}
  	$$
  	This function satisfies (N1)--(N3), (N4'), and (N5) but it does not satisfy (N4) and it is not stable, as
  	$$
  	m_1(s) = \begin{cases}
  	\frac{1}{2}s^2\quad \textrm{for} \quad |\xi|\leq 1,\\
  	s\quad \textrm{otherwise},
  	\end{cases}
  	$$
  	and it has linear, and not superlinear growth at infinity.
  	
  \end{remark}
  
  \begin{remark}\label{rem:homo_minorant}
  	We will make use of the function $M_1:\R^d\to \R$ defined as 
  	$$
  	M_1(\xi) = \widetilde{\widetilde {\essinf_{x\in\Omega} M(x,\xi)}},
  	$$
  	the greatest convex minorant of $\essinf_{x\in\Omega} M(x,\xi)$. 
  	It is not hard to verify that if $M$ is an $N$-function, then  $M_1$ is a homogeneous $N$-function and for every $\xi\in \R^d$ and for almost every $x\in \Omega$ there holds 
  	$$
  	m_1(|\xi|) \leq M_1(\xi) \leq M(x,\xi).
  	$$
  \end{remark}

  The following result shows the integrability of an $N$-function
  \begin{lemma}
  	Every $N$-function is integrable, i.e. for every $\xi\in \R^d$ there holds
  	$$
  	\int_{\Omega}M(x,\xi)\, dx < \infty.
  	$$
  \end{lemma}
  \begin{proof}
  	The result readily follows from stability of $M$ and the fact that $\Omega$ is bounded.
  \end{proof}
  
  We continue by reminding some properties of $N$-functions. 
  
  \begin{lemma} If $L, M$ are $N$-functions, then the following assertions hold
  	\begin{itemize}
  		\item[1.] For almost every $x\in \Omega$ and for every $\xi,\eta \in \R^d$ there holds the following Fenchel--Young inequality
  		$$
  		|\xi\cdot \eta| \leq M(x,\xi) + \widetilde{M}(x,\eta).
  		$$
  		\item[2.] If, for some $x\in \Omega$, there holds $L(x,\xi)\leq M(x,\xi)$ for all $\xi\in \R^d$ then, for this $x$,  $\widetilde{M}(x,\xi)\leq \widetilde{L}(x,\xi)$ for every $\xi\in \R^d$.
  		\item[3.] There holds 
  		$$
  		\lim_{|\xi|\to \infty} \essinf_{x\in \Omega}\frac{\widetilde{M}(x,\xi)}{|\xi|} = \infty \quad \textrm{and}\quad \lim_{|\xi|\to 0}\esssup_{x\in \Omega}\frac{\widetilde{M}(x,\xi)}{|\xi|}=0.
  		$$
  		\item[4.] There holds
  		$$\essinf_{x\in \Omega} \inf_{|\xi|=s} \widetilde{M}(x,\xi) \geq \widetilde{m_2}(s) > 0\quad \textrm{for every nonzero}\quad \xi \in \R^d.$$ 
  		\item[5.] There holds
  		$$\esssup_{x\in \Omega} \widetilde{M}(x,\xi) \leq \widetilde{m_1}(|\xi|) < \infty\quad \textrm{for every}\quad \xi \in \R^d.$$ 
  	\end{itemize}
  	It follows that the complementary function of an $N$-function is also an $N$-function.  
  \end{lemma}
  \begin{proof}
  	The assertions 1. and 2. of the above lemma are standard properties of the Fenchel conjugate valid for functions which satisfy (N1)--(N3) even without (N4) and (N5). To prove 3. take $\xi = K \eta / |\eta|$ in the Fenchel--Young inequality which yields
  	$$
  	K |\eta| \leq \widetilde{M}(x,\eta) + M(x, K\eta / |\eta|) \leq \widetilde{M}(x,\eta) + m_2(K).
  	$$  
  	Dividing by $|\eta|$ and taking $\essinf_{x\in \Omega}$ we obtain
  	$$
  	\essinf_{x\in \Omega}\frac{\widetilde{M}(x,\eta)}{|\eta|} \geq K - \frac{m_2(K)}{|\eta|},
  	$$
  	and the assertion follows by passing with $|\eta|$ to $\infty$. To prove the second assertion of 3. observe that
  	$$
  	\esssup_{x\in \Omega}\frac{\widetilde{M}(x,\xi)}{|\xi|} \leq \frac{\widetilde{m_1}(|\xi|)}{|\xi|}.
  	$$
  	For any $s\in (0,\infty)$ we can find $t=t(s)\in (0,\infty)$ 
  	such that $\widetilde{m_1}(s) + m_1(t) = st$. Such $t$ always exists as the equality in the Fenchel--Young inequality is equivalent to the fact that $t\in \partial m_1(s)$ (i.e. $t$ belongs to the convex subdifferential of $m_1$ at $s$), and for convex functions leading from $\R$ to $\R$ (which have to be continuous) the subdifferential is always nonempty. Let $s_n\to 0^+$. The corresponding sequence $t(s_n)$ must be bounded. Indeed, if this is not the case, then, for a subsequence
  	$$
  	\frac{m_1(t(s_n))}{t(s_n)} \leq \frac{m_1(t(s_n)) + \widetilde{m_1}(s_n)}{t(s_n)} = s_n,
  	$$  
  	a contradiction. So, for a subsequence, $t(s_n) \to t_0$. Passing to the limit with $n$ to $\infty$ in the expression $\widetilde{m_1}(s_n) + m_1(t(s_n)) = s_n t(s_n)$ we deduce that $t_0 = 0$ and the whole sequence converges. Now
  	$$
  	\frac{\widetilde{m_1}(s_n)}{s_n} \leq \frac{\widetilde{m_1}(s_n)+m_1(t(s_n))}{s_n} =   t(s_n) \to 0\quad \textrm{as} \quad n\to \infty.
  	$$
  	The assertion 3. is proved. The first inequality in 4. follows from 2. The fact that $\widetilde{m_2}(s)\neq 0$ for $s\neq 0$ follows from the fact that $m_2$ is an $N$-function. Indeed, assume the contrary, that is 
  	$
  	\widetilde{m_2}(s) = 0
  	$ for some nonzero $s\in \R$. 
  	Taking $t = s/K$ in the Fenchel--Young inequality we obtain
  	$$
  	m_2(s/K) + \widetilde{m_2}(s) \geq s^2 / K.
  	$$ 
  	It follows that
  	$$
  	\frac{m_2(s/K)}{s/K} \geq s,
  	$$
  	and we have the contradiction by passing with $K$ to infinity, as $m_2$, an $N$-function, in particular satisfies the assertion (N4) of Definition \ref{def:N}. Let us prove 5. To this end assume, for contradiction, that there exists $s\in \R$ and the sequence $\{ t_n \}\subset \R$ such that $t_n s - m_1(t_n) \geq n$. It is clear that $t_n \to \infty$. Then we obtain 
  	$$
  	s\geq \frac{m_1(t_n)}{t_n} + \frac{n}{t_n} \geq \frac{m_1(t_n)}{t_n},
  	$$   
  	and we get the contradiction with superlinear growth of $m_1$ at infinity by passing with $n$ to infinity.
  \end{proof}

  Sometimes we will use the so called $\Delta_2$ condition which states that there exists a constant $c>0$ and a nonnegative function $h\in L^1(\Omega)$ such that for a.e. $x\in \Omega$ and every $\xi\in \R^d$ 
  \begin{equation}\label{delta2}
  M(x,2\xi) \leq c M(x,\xi) + h(x).
  \end{equation}
  If $M$ satisfies the above condition for $\xi\in \R^d \setminus B(0,R)$ for some $R>0$ than we write that $M \in \Delta_2^\infty$ and we say that $M$ satisfies $\Delta_2$ far from origin.

  \medskip
  
  \noindent \textbf{Orlicz--Musielak spaces.} We remind the definition of the Orlicz--Musielak class $\mathcal{L}_M(\Omega)$ and two spaces $L_M(\Omega)$, and $E_M(\Omega)$. 
  \begin{definition}\label{def:Musielak}
  	Suppose that $M$ is an $N$-function. 
  	\begin{itemize}
  		\item[1.] $\mathcal{L}_M(\Omega)$, the Orlicz--Musielak class, is the set of all measurable functions $\xi:\Omega\to \R^d$ such that
  		$$
  		\int_{\Omega} M(x,\xi(x))\, dx < \infty.
  		$$
  		\item[2.] $L_M(\Omega)$ is the generalized Orlicz--Musielak space, which is the smallest linear space containing $\mathcal{L}_M(\Omega)$, equipped with the Luxemburg norm
  		$$
  		\|\xi\|_{L_M} = \inf\left\{\lambda > 0\, :\ \int_{\Omega} M\left(x,\frac{\xi(x)}{\lambda}\right) \, dx \leq 1 \right\}.
  		$$
  		\item[3.] $E_M(\Omega)$ is the closure of $L^\infty(\Omega)^d$ in $L_M$ norm.
  \end{itemize}\end{definition}
  We prove the following result
  \begin{theorem}\label{thm:density}
  	The space $C^\infty_0(\Omega)^d$ is dense in $E_M(\Omega)$ in $L_M$ norm.
  \end{theorem}
  \begin{proof}
  	Since $E_M(\Omega)$ is a closure of $L^\infty(\Omega)^d$ in $L_M$ norm, it is enough to prove that $C^\infty_0(\Omega)^d$ is dense in $L^\infty(\Omega)^d$ in $L_M(\Omega)$ norm. We first prove that $C_c(\Omega)^d$ is $L_M$ dense in $L^\infty(\Omega)^d$. Take $v\in L^\infty(\Omega)^d$. By the Luzin theorem there exists a sequence of compact sets $E_n$ such that $|\Omega\setminus E_n| < 1/n$ and the functions $v|_{E_n}:E_n\to \R$ are continuous. These functions, by the Tietze--Urysohn lemma can be extended to functions $v_n$ in $C_c(\Omega)^d$ such that $\|v_n\|_{L^\infty} \leq \|v\|_{L^\infty}$.  Now
  	\begin{align*}
  	&\int_{\Omega}M\left(x,\frac{v_n(x)-v(x)}{\lambda}\right)\, dx = \int_{\Omega\setminus{E_n}}M\left(x,\frac{v_n(x)-v(x)}{\lambda}\right)\, dx  \leq \int_{\Omega\setminus{E_n}} m_2 \left(\frac{|v_n(x)|+|v(x)|}{\lambda}\right) \,dx\\
  	& \quad \leq m_2 \left(\frac{2\|v\|_{L^\infty(\Omega)^d}}{\lambda}\right)|\Omega\setminus{E_n}| \to 0 \quad \textrm{as}\quad n\to \infty\quad \textrm{for every}\ \ \lambda>0.
  	\end{align*}
  	We now prove that $C^\infty_0(\Omega)^d$ is $L_M$ dense in $C_c(\Omega)^d$. To this end take $v\in C_c(\Omega)^d$ and extend it to $\R^d$ by taking $v=0$ outside $\Omega$. Choose a standard mollifier kernel 
  	$$\theta_\epsilon(x) = \frac{1}{\epsilon^n}\theta\left(\frac{x}{\epsilon}\right)\quad \textrm{where}\quad \theta(x) = \begin{cases}
  	Ce^{\frac{1}{|x|^2}-1}\quad \textrm{for}\quad |x|\leq 1,\\
  	0\quad \textrm{otherwise},
  	\end{cases}$$
  	where the constant $C$ is chosen such that $\int_{\R^n}\theta(x)\, dx = 1$. Define $v^\epsilon(x) = \int_{\R^d} \theta_\epsilon(x-y)v(y)\, dy$. Then, if only $\epsilon$ is small enough, $v^\epsilon \in C^\infty_0(\R^d)$ and there holds the pointwise convergence $\lim_{\epsilon\to 0}v^\epsilon(x) = v(x)$ for almost every $x\in \Omega$. Moreover, $|v^\epsilon(x)| \leq \|v\|_{L^\infty(\Omega)^d}$ for almost every $x\in \Omega$. There holds
  	$$
  	M\left(x,\frac{v_\epsilon(x)-v(x)}{\lambda}\right) \leq m_2 \left(\frac{|v_\epsilon(x)|+|v(x)|}{\lambda}\right) \leq m_2 \left(\frac{2\|v\|_{L^\infty(\Omega)^d}}{\lambda}\right)\quad \textrm{for every}\ \ \lambda>0.
  	$$
  	We can use the Lebesgue dominated convergence theorem to deduce that
  	$$
  	\lim_{\epsilon\to 0}\int_{\Omega} M\left(x,\frac{v_\epsilon(x)-v(x)}{\lambda}\right)\, dx = 0\quad \textrm{for every}\ \ \lambda>0,
  	$$
  	whence $\lim_{\epsilon\to 0}\|v_\epsilon-v\|_{L_M} = 0$, and the proof is complete.
  \end{proof}

  It is not hard to verify, that for $v\in L_M(\Omega)$ there holds
  \begin{align}\label{eq:mest}
  &\|v\|_{L_M} \leq \int_{\Omega}M(x,v(x))\, dx+1,\\
  &\|v\|_{L_M} \leq 1 \Rightarrow \int_{\Omega}M(x,v(x))\, dx \leq \|v\|_{L_M}.
  \end{align}
  
  It is clear that $L^\infty(\Omega)^d \subset E_M(\Omega)$ and $\mathcal{L}_M(\Omega) \subset L_M(\Omega)$. We also observe that $E_M(\Omega)\subset \mathcal{L}_M(\Omega)$. Indeed, if $v_k$ is a sequence in $L^\infty(\Omega)^d$ such that 
  $\|v_k-v\|_{L_M} \to 0$, then 
  $$
  \int_{\Omega}M(x,v(x))\, dx \leq \frac{1}{2}\left(\int_{\Omega}M(x,2v_k(x))\, dx+\int_{\Omega}M(x,2(v(x)-v_k(x)))\, dx\right).
  $$
  Taking $k$ large enough, such that $2\|v-v_k\|_{L_M} \leq 1$, it follows that
  $$
  \int_{\Omega}M(x,v(x))\, dx \leq \frac{1}{2}\int_{\Omega}M(x,2v_k(x))\, dx+\|v-v_k\|_{L_M} \leq \frac{1}{2}|\Omega| m_2(2\|v_k\|_{L^\infty}))+\|v-v_k\|_{L_M} < \infty.
  $$

  Some functional analytic properties of the defined spaces are summarized in the following lemmas \cite{Schappacher, gwiazda_swierczewska, wroblewska, kalousek, skaff}.
  \begin{lemma}\label{lem:hol}If $M$ is an $N$-function, $\xi\in L_M(\Omega)$, and $\eta\in L_{\widetilde{M}}(\Omega)$, then the following generalized H\"{o}lder inequality holds
  	$$
  	\left|\int_{\Omega} \xi \cdot \eta\, dx\right| \leq 2  \|\xi\|_{L_M(\Omega)}\|\eta\|_{L_{\widetilde{M}}(\Omega)}
  	$$
  \end{lemma}
  
  \begin{lemma}
  	Let $M$ be an $N$-function. Then\begin{itemize}
  		\item $E_M(\Omega)$ is separable.
  		\item $E_M(\Omega)^* = L_{\widetilde{M}}(\Omega)$.
  		\item $E_M(\Omega) = L_M(\Omega)$ if and only if $M\in \Delta_2^\infty$.
  		\item $L_M(\Omega)$ is separable if and only if $M\in \Delta_2^\infty$.
  		\item $L_M(\Omega)$ is reflexive if and only if both $M\in \Delta_2^\infty$ and $\widetilde{M}\in \Delta_2^\infty$. 
  	\end{itemize}
  \end{lemma}
  In this article we assume nowhere that both $M$ and $\widetilde{M}$ satisfy the $\Delta_2$ condition, so that we have to deal with the lack of reflexivity. We  also deal with the case where $M$ does not satisfy the $\Delta_2$ condition, so we cannot use the separability of $L_M$. Despite this difficulties we are still in position to obtain the existence results using the functional analytic tools developed in 
  \cite{kalousek, gwiazda_skrzypczak, gwiazda_wittbold,
  	gwiazda_corrigendum}.  
  
  If $M$ is an $N$-function, we define the space
  $$
  V_0^M = \{ v\in W^{1,1}_0(\Omega)\, :\ \nabla v \in L_M(\Omega) \}.
  $$
  
  We will need the following version of the modular Poincar\'{e} inequality, cf. \cite[Lemma 1]{leonetti}, \cite[Lemma 3]{talenti}, or more recent works \cite[Theorem 2.2]{gwiazda_skrzypczak}, \cite[Corollary 4.2]{giannetti}.
  \begin{theorem}\label{thm:convergence}
  	Let $m:[0,\infty)\to [0,\infty)$ be an $N$-function. There exist constants $\lambda>0$ and $C>0$ such that for every $u\in W^{1,1}_0(\Omega)$ satisfying $\int_{\Omega}m(\lambda|\nabla u|)\, dx < \infty $ there holds
  	$$
  	\int_{\Omega}m(|u|)\, dx \leq C \int_{\Omega}m(\lambda|\nabla u|)\, dx.
  	$$
  \end{theorem} 
  \begin{remark}
  	In \cite{gwiazda_skrzypczak} it is proved that the above theorem holds with $\lambda=1$ provided $m$ satisfies the $\Delta_2$ condition. A careful analysis of its proof, however, reveals that without the $\Delta_2$ condition the result holds with a constant $\lambda$ not necessary equal to one, but dependent only on $\Omega$ and $d$. 
  \end{remark}
  We remind the definition of modular convergence, cf. \cite{gwiazda_skrzypczak, kalousek, gwiazda_wroblewska, gwiazda_swierczewska}.

  \begin{definition}
  	A sequence $\{ v_m \}_{m=1}^\infty$ of measurable $\R^d$ valued function on $\Omega$ is said to converge modularly to a function $v$ if there exists $\lambda>0$ such that 
  	$$
  	\lim_{m\to\infty}\int_{\Omega} M\left(x,\frac{v_m-v}{\lambda}\right)\, \, dx = 0.
  	$$
  	We denote the modular convergence by $v_m \xrightarrow{M} v$.
  	Equivalently, cf. \cite[Lemma 2.1]{gwiazda_swierczewska}, $\{v_m\}_{m=1}^\infty$ converges modularly to $v$ if $v_m\to v$ in measure and the sequence $\{ M(\cdot, \lambda v_m) \}_{m=1}^\infty$ is uniformly integrable for some $\lambda>0$.
  \end{definition}
  
  \begin{lemma}(cf. \cite[Lemma 2.2]{gwiazda_swierczewska}) Let $M$ be an $N$-function. If, for constants $c,\lambda>0$, we have $\int_{\Omega}M(x,\lambda v_m)\, dx \leq c$ for all $m\in \mathbb{N}$ then the sequence $\{v_m\}_{m=1}^\infty$ is uniformly integrable. 
  \end{lemma}
  \begin{proof}
  	For every $m\in \mathbb{N}$ it holds
  	$$
  	c\geq \int_{\{x\in \Omega\, :\ |v_m(x)|\geq R\}} M(x,\lambda v_m(x))\, dx \geq \lambda  \int_{\{x\in \Omega\, :\ |v_m(x)|\geq R\}} \frac{m_1(\lambda |v_m(x)|)}{\lambda |v_m(x)|} |v_m(x)|\, dx.
  	$$
  	As $m_1$ is an $N$-function, for any $D>0$ there exists $R_0(D)>0$ such that for any $s\geq R_0$ there holds $\frac{m_1(\lambda s)}{\lambda s} \geq D$. Hence
  	$$
  	\frac{c}{\lambda D} \geq \int_{\{x\in \Omega\, :\ |v_m(x)|\geq R_0(D)\}} |v_m(x)|\, dx,
  	$$
  	and the assertion follows easily. 
  \end{proof}
  The following approximation theorem which has been proved in \cite[Theorem 3.1]{chlebicka_parabolic2} is valid in nonreflexive and nonseparable Musielak--Orlicz spaces.
  \begin{theorem}\label{thm:approx}
  	Let $\Omega$ be a Lipschitz domain and let an $N$-function $M$ satisfy (C2). Then for any $u\in L^{\infty}(\Omega)\cap V_0^M$ there exists a sequence $\{u_m\}_{m=1}^\infty$ of functions belonging to $C_0^\infty(\Omega)$ such that 
  	$u_m\to u$ strongly in $L^1(\Omega)$ and $\nabla u_m \xrightarrow{M} \nabla u$ in $L_M(\Omega).$ Moreover, there exists a constant $c=c(\Omega) > 0$, such that $\|u_m\|_{L^{\infty}(\Omega)}\leq c\|u\|_{L^{\infty}(\Omega)}.$
  \end{theorem} 
  We remind an important property of the modular convergence
  \begin{lemma}\label{lem:modular}
  	(cf. \cite[Proposition 2.2]{gwiazda_swierczewska}) Suppose that the sequences $\{v_k\}_{k=1}^\infty$ and $\{w_k\}_{k=1}^\infty$ are uniformly bounded in $L_M(\Omega)$ and $L_{\widetilde{M}}(\Omega)$, respectively. If $v_k\xrightarrow{M} v$ and $w_k\xrightarrow{\widetilde{M}} w$ then $v_k\cdot w_k \to v\cdot w$ in $L^1(\Omega)$.
  \end{lemma}

\section{Some useful tools of nonlinear analysis.}
We recall some tools useful in the arguments of this article. \begin{definition}
	A sequence $\{f_m\}_{m=1}^\infty$ of measurable functions $f_m:\Omega\to \R$ is uniformly integrable if, equivalently, one of the following conditions holds:
	\begin{itemize}
		\item[(i)] $$
		\lim_{R\to \infty}\left(\sup_{m\in\mathbb{N}}\int_{\{ x\in \Omega\, :\ |f_m(x)| \geq R \}}|f_m(x)|\, dx\right) = 0.
		$$
		\item[(ii)] $$
		\forall\, \epsilon>0\ \ \exists\, \delta>0\quad \sup_{m\in\mathbb{N}} \int_{\Omega}\left(|f_m(x)|-\frac{1}{\sqrt{\delta}}\right)_+\, dx \leq \epsilon.
		$$
		\item[(iii)] There exists a continuous, concave, and nondecreasing function $\omega:[0,\infty) \to [0,\infty)$ such that $\omega(0)=0$ and for every measurable set $E\subset \Omega$ and for every $m\in \mathbb{N}$ there holds
		$$
		\int_{E}|f_m(x)|\, dx \leq \omega(|E|).
		$$ 
		\item[(iv)] There exists a function $\Phi:[0,\infty)\to [0,\infty)$ which is convex, $\Phi(0) = 0$, $\lim_{s\to \infty}\frac{\Phi(s)}{s} = \infty$, and
		$$
		\sup_{n\in \mathbb{N}} \int_{\Omega}\Phi(|f_n(x)|)\, dx < \infty.
		$$
		\item[(v)] The set $\{ f_n\}_{n=1}^\infty$ is relatively compact (or, equivalently, relatively sequentially compact) in the weak topology of $L^1(\Omega)$. 
	\end{itemize}
	
	\begin{remark}
		The fact that condition (iv) of the above definition is equivalent to the other three is known as the de la Vall\'{e}e Poussin theorem. The equivalence of assertion (v) to the remaining ones is known as the Dunford--Pettis theorem.  
	\end{remark}
\end{definition}
Young measures are now a standard tool of nonlinear analysis, we refer for example to \cite{pedregal} for a comprehensive exposition of their theory. We will need the following version of the generalized fundamental theorem on Young measures from \cite{Gwiazda_zatorska}, where by $\mathcal{M}(\R^N)$ we denote the space of bounded Radon measures.  

\begin{proposition}	(cf. \cite[Theorem 4.1]{Gwiazda_zatorska}) \label{fundamental}
	Let $\Omega$ be an open and bounded subset of $\mathbb{R}^d.$ Assume that the sequence $\{ \nu^j \}_{j=1}^\infty \subset L_{w}^{\infty}(\Omega,  \mathcal{M}(\R^N))$ of weakly-* measurable mappings is such that  $\nu^j(x)=\nu^j_x$ is a~probability measure for almost every $x\in\Omega$. 
	If the sequence $\nu^j$ satisfies the tightness condition:
	$$
	\lim_{M\to \infty} \sup_j |\{x\in \Omega\:\ \mathrm{supp}(\nu^j_x)\setminus B(0,M)\neq \emptyset\}| \to 0,
	$$
	then the following assertions are true.
	\begin{itemize}
		\item [(1)] There exists a weakly-* measurable mapping $\nu\in L^\infty_w(\Omega,\mathcal{M}(\R^N))$ such that, for a subsequence still denoted by $j$, there holds 
		$$
		\nu^j\to\nu\quad \mathrm{weakly}-*\ \mathrm{in} \quad L_w^{\infty}(\Omega,  \mathcal{M}(\R^N)),
		$$
		\item[(2)]$\|\nu_x\|_{\mathcal{M}(\R^N)}=1$ a.e in $\Omega.$ Moreover, for every $f\in L^{\infty}(\Omega,C_b(\mathbb{R}^N))$ there holds
		$$
		\int_{\mathbb{R}^N}f(x,\lambda)d\nu_x^j(\lambda)\to\int_{\mathbb{R}^N}f(x,\lambda)d\nu_x(\lambda) ~~\mathrm{weakly}-*\ \mathrm{in}~~L^\infty(\Omega).
		$$

		\item [(3)] For every measurable subset $A\subset\Omega$ and for every Carath\'eodory function $f$ (measurable in the first, and continuous in the second variable) such that
		$$
		\lim_{R\to0}\sup_{j\in \mathbb{N}}\int_A\int_{\{\lambda\in \mathbb{R}^N\, :\  |f(x,\lambda)| > R\}} |f(x,\lambda)|d\nu^j_x(\lambda)dx=0,
		$$
there holds
		$$
		\int_{\mathbb{R}^N}f(x,\lambda)d\nu_x^j(\lambda)\to\int_{\mathbb{R}^N}f(x,\lambda)d\nu_x(\lambda) ~~\mathrm{weakly}\ \  \mathrm{in}~~L^1(A)
		$$
	\end{itemize}
\end{proposition}

The following corollary is the generalization of the result on the  lower-semicontinuity of Young measure generated by a sequence of functions, cf. \cite[Corollary 3.3]{muller}, to the case when the measure is generated by a sequence of measures. 
\begin{corollary}\label{lower}
	Let $ \Omega\subset\mathbb{R}^d$ be an open and bounded subset of $\mathbb{R}^d$. Suppose that  
	$$\{ \nu^j \}_{j=1}^\infty \subset L_{w}^{\infty}(\Omega,  \mathcal{M}(\R^N))\quad \mathrm{and}\quad \nu\in L^\infty_w(\Omega, \mathcal{M}(\R^N))$$ are weakly-* measurable mappings such that $\nu^j_x$ and $\nu_x$ are probability measures for a.e. $x\in \Omega$. Moreover assume that 
	$$
	\nu^{j} \to \nu\quad \mathrm{weakly}-*\ \ \mathrm{in}\quad L^\infty_w(\Omega,\mathcal{M}(\R^N)).
	$$
	Then for any measurable set $ E\subset\Omega$ and every Carath\'eodory function such that there exists $m\in L^1(\Omega), $ with  $m\geq 0$ and $f(x,\lambda)> - m(x)$ for almost every $x\in \Omega$ and every $\lambda\in \bigcup_{j=1}^\infty \mathrm{supp}\, \nu_x^j$ , there holds
	$$
	\int_{E}\int_{\mathbb {R}^N}f(x,\lambda)d\nu_x(\lambda)\leq\liminf_{j\to \infty} \int_{E}\int_{\mathbb {R}^N}f(x,\lambda)d\nu_x^j(\lambda)
	$$
\end{corollary}
\begin{proof}
	The proof follows the lines of the proof of Corollary 3.3 in \cite{muller}. First assume that there exist $R>0$ such that $f(x,\lambda)=0$ for $ |\lambda|\geq R$. 
	By the Scorza--Dragoni theorem there exists an increasing sequence of compact sets $E_k$ such that  $|E \setminus E_k| \to 0$ as $k\to \infty$ and  $f|_{E_k\times\mathbb{R^N}}$ is continuous.
	Define $F_k: E\to C_0(\mathbb{R}^N)$ as $F_k(x)= \chi_{E_k}(x)f(x,\cdot).$ We observe that $F_k \in L^1(E,C_0(\mathbb{R}^N))$. Indeed, 
\begin{align*}
	&\int_{E}\|F_k(x)\|_{C_0(\mathbb{R}^N)}dx=\int_{E}\sup_{\lambda\in\mathbb{R}^N}|F_k(x,\lambda)|dx=\int_{E_k}\sup_{\lambda\in\mathbb{R}^N}|f(x,\lambda)|dx \\
&\qquad 	\leq|\Omega| \sup_{(x,\lambda)\in E_k\times\mathbb{R}^N}|f(x,\lambda)|= |\Omega| \sup_{(x,\lambda)\in E_k\times\overline{B(-R,R)}}|f(x,\lambda)| < \infty.
\end{align*}	
Now, as $(L^1(E,C_0(\mathbb{R}^N)))' = L^\infty_w(\Omega,\mathcal{M}(\R^N))$, there holds 
	$$
	\lim_{j\to\infty}\int_{E}\int_{\mathbb {R}^N}F_k(x,\lambda)d\nu_x^{j}(\lambda)dx=  \int_{E}\int_{\mathbb {R^N}}F_k(x,\lambda)d\nu_x(\lambda)dx, 
	$$
	
	$$
	\lim_{j\to\infty} \int_{E_k}\int_{\mathbb {R}^N}(f(x,\lambda)+m(x)-m(x))d\nu_x^{j}(\lambda)dx=\int_{E_k}\int_{\mathbb {R^N}}(f(x,\lambda)+m(x)-m(x))d\nu_x(\lambda)dx.
	$$
	As $\nu^j_x$ and $\nu$ are probability measures, it follows that 
		$$
	\lim_{j\to\infty} \int_{E_k}\int_{\mathbb {R}^N}(f(x,\lambda)+m(x))d\nu_x^{j}(\lambda)dx=\int_{E_k}\int_{\mathbb {R^N}}(f(x,\lambda)+m(x))d\nu_x(\lambda)dx.
	$$
	It follows that
		$$
	\int_{E}\int_{\mathbb {R^N}}\chi_{E_k}(x)(f(x,\lambda)+m(x))d\nu_x(\lambda)dx \leq \liminf_{j\to\infty} \int_{E}\int_{\mathbb {R}^N}(f(x,\lambda)+m(x))d\nu_x^{j}(\lambda)dx.
	$$	
	Letting $k\to \infty$ we obtain the assertion by the monotone convergence theorem. To remove the assumption that $f(x,\lambda) = 0$ $|\lambda|\geq R$ consider an increasing sequence of nonnegative functions $\eta_l\subset C_0^{\infty}(\mathbb{R}^N)$ that converges pointwise to $1$.  We use the above result for $f(x,\lambda)\eta_l(\lambda)$
	\begin{align*}
		&\int_{E}\int_{\mathbb {R^N}}f(x,\lambda)\eta_l(\lambda) d\nu_x(\lambda)dx\leq\liminf_{j\to\infty}\int_{E}\int_{\mathbb {R^N}}f(x,\lambda)\eta_l(\lambda)d\nu_x^{j}(\lambda)dx\\
		& \qquad \leq\liminf_{j\to \infty}\left(\int_{E}\int_{\mathbb {R^N}}f(x,\lambda)+m(x) d \nu_x^{\epsilon}(\lambda)dx - \int_{E}\int_{\mathbb {R^N}}m(x)\eta^l(\lambda) d\nu_x^{j}(\lambda)dx\right).
	\end{align*}
	But $m(x)\eta^l(\lambda) \in L^1(E,C_0(\mathbb{R}^N))$ and hence
	$$
	\int_{E}\int_{\mathbb {R^N}}(f(x,\lambda)+m(x))\eta_l(\lambda) d\nu_x(\lambda)dx \leq \liminf_{j\to \infty}\int_{E}\int_{\mathbb {R^N}}f(x,\lambda)+m(x) d \nu_x^{j}(\lambda)dx.
	$$
	We can pass  to the limit $l\to\infty$ in the left-hand side by the monotone convergence theorem
		$$
	\int_{E}\int_{\mathbb {R^N}}f(x,\lambda)+m(x) d\nu_x(\lambda)dx \leq \liminf_{j\to \infty}\int_{E}\int_{\mathbb {R^N}}f(x,\lambda)+m(x) d \nu_x^{j}(\lambda)dx,
	$$
	and the assertion follows. 
	\end {proof}

We recall the definition of the biting convergence and the statement of the Chacon biting Lemma.
\begin{definition}
Let $\Omega\subset \R^d$ be a measurable set. We say that a sequence $\{ f_j \}_{j=1}^\infty \subset L^1(\Omega)$ converges to a function $f\in L^1(\Omega)$ in a biting sense (and we write $f_j \stackrel{b}{\to} f$) if there exists a sequence of measurable sets $E_l \subset \Omega$ with $|\Omega\setminus E_l| \to 0$ as $l\to \infty$  and $E_1\subseteq E_2 \subseteq \ldots \subseteq E_l \subseteq\ldots \subseteq \Omega$ such that 
$$
f_j \to f\quad \mathrm{weakly}\ \mathrm{in}\ L^1(E_l)\quad \mathrm{for}\ \mathrm{every} \ l\in \mathbb{N}.
$$ 	
\end{definition}
The proof of the following proposition (known as the Chacon biting lemma) can be found for example in \cite{ball_murat}. 
\begin{proposition}\label{chacon}
	Let $\Omega\subset \R^d$ be a measurable set and let the sequence $\{ f_j \}_{j=1}^\infty \subset L^1(\Omega)$ be bounded in $L^1(\Omega)$. There exists a subsequence of indices, still denoted by $j$, and a function $f\in L^1(\Omega)$ such that $f_j \stackrel{b}{\to} f$.
\end{proposition}
We will need the following result which states when the sequence which converges in the biting sense is convergent weakly in $L^1(\Omega)$, cf. \cite[Lemma 4.6]{gwiazda_wittbold2}.
\begin{proposition}\label{biting_weak}
	Let the sequence $\{ a_j \}_{j=1}^\infty \subset L^1(\Omega)$ and let $0\leq a_0\in L^1(\Omega)$ be such that $a_j(x) \geq -a_0(x)$ for almost every $x\in \Omega$. If 
	$$
	a_n \stackrel{b}{\to} a \quad \mathrm{and} \quad \limsup_{j\to \infty}\int_{\Omega}a_j\, dx \leq \int_{\Omega} a\, dx,
	$$ 
	then 
	$$
	a_j \to a\quad \mathrm{weakly} \ \mathrm{in}\quad L^1(\Omega). 
	$$
\end{proposition}
We also make use of the following well known result
\begin{proposition}\label{convergence_lemma}
	Assume that $\Omega\subset \mathbb{R}^d$ is a bounded set. Let the sequence $f_j\to f$ weakly in $L^1(\Omega)$, and let $g_j, g\in L^\infty(\Omega)$ be such that $\|g_j\|_{L^\infty(\Omega)} \leq C$ and $\|g\|_{L^\infty(\Omega)} \leq C$, where the constant $C$ is independent of $j$ and for almost every $x\in \Omega$ there holds the pointwise convergence $g_j(x)\to g(x)$. Then
	$$
	\lim_{j\to\infty} \int_{\Omega}f_j g_j\, dx = \int_{\Omega}f g\, dx
	$$
\end{proposition}

\end{document}